\documentclass[11pt]{amsart}
\usepackage[english]{babel}
\usepackage[margin=1.4in]{geometry}
\usepackage{amsmath}
\usepackage{amssymb}
\usepackage{amsthm}

\usepackage{palatino}

\usepackage{newpxmath}

\usepackage{yfonts}
\usepackage[T1]{fontenc}
\usepackage[utf8x]{inputenc}
\usepackage{enumerate}
\usepackage{enumitem}
\usepackage{verbatim}
\usepackage{graphicx}
\usepackage{faktor}
\usepackage{xcolor}
\usepackage{xfrac}
\usepackage{tikz,tikz-cd,tikz-3dplot}
\usepackage{pgfplots}
\usetikzlibrary{arrows,shadows,positioning, calc, decorations.markings, 
hobby,quotes,angles,decorations.pathreplacing,intersections,shapes}
\usepgflibrary{shapes.geometric}
\usetikzlibrary{fillbetween,backgrounds}

\tikzset{
  arrow/.pic={\path[tips,every arrow/.try,->,>=#1] (0,0) -- +(0,4pt);},
  pics/arrow/.default={triangle 90}
}

\tikzset{->-/.style={decoration={
  markings,
  mark=at position .6 with {\arrow{latex}}},postaction={decorate}}
  }

\tikzset{
  c/.style={every coordinate/.try}
}

\newcommand{\midarrow}{\tikz \draw[-triangle 90] (0,0) -- +(.03,0);}

\usepackage[all]{xy}
\usepackage{hyperref}
\usepackage[normalem]{ulem}
\usepackage{setspace}

\usepackage{calrsfs}
\DeclareMathAlphabet{\pazocal}{OMS}{zplm}{m}{n}

\newcommand{\tropC}{\scalebox{0.8}[1.3]{$\sqsubset$}}
\newcommand{\ttropC}{\widetilde{\tropC}}
\newcommand{\tropT}{\scalebox{0.9}[1.2]{$\top$}}
\newcommand{\tropD}{\scalebox{0.9}[1.3]{$\triangle$}}

\newcommand{\tropA}{\scalebox{0.8}[1.3]{$\sphericalangle$}}
\newcommand{\CO}{\mathcal{O}}
\newcommand{\CE}{\mathcal{E}}
\newcommand{\CL}{\mathcal{L}}
\newcommand{\tC}{\widetilde{C}}
\newcommand{\Do}{\pazocal D_1}
\newcommand{\Dt}{\pazocal D_2}
\newcommand{\VZ}{\pazocal{V\!Z}_{2,n}}

\newcommand{\M}[4]{\overline{\pazocal M}_{#1,#2}(#3,#4)}
\newcommand{\PP}{\mathbb P}
\renewcommand{\k}{\mathbf k}
\newcommand{\m}{\mathfrak m}
\newcommand{\tR}{\widetilde{R}}
\newcommand{\tm}{\widetilde{\mathfrak m}}
\newcommand{\NN}{\mathbb N}
\newcommand{\OO}{\mathcal O}
\newcommand{\ZZ}{\mathbb Z}
\renewcommand{\to}{\rightarrow}
\newcommand{\pr}{\rm{pr}}
\newcommand{\Aaff}{\mathbb A}
\newcommand{\N}{\mathbb N}
\newcommand{\oM}{\overline{\pazocal M}}
\newcommand{\tM}{\widetilde{\pazocal M}}
\newcommand{\R}{\operatorname{R}}
\newcommand{\val}{\operatorname{val}}
\newcommand{\Gm}{\mathbb{G}_{\rm{m}}}
\newcommand{\Ga}{\mathbb{G}_{\rm{a}}}
\newcommand{\vir}[1]{[#1]^{\rm{vir}}}
\newcommand{\virloc}[1]{[#1]^{\rm{vir}}_{\rm{loc}}}
\newcommand{\dvr}{\Delta}
\newcommand{\cC}{\mathcal C}
\newcommand{\TC}{\widetilde{\mathcal C}}
\newcommand{\TA}{\widetilde{\mathcal A}}
\newcommand{\Cp}{\mathcal{C}^\prime}
\newcommand{\Tp}{\mathcal{T}^\prime}
\newcommand{\Z}{\pazocal Z}
\newcommand{\Zp}{\pazocal{Z}^\prime}
\newcommand{\Li}{\mathcal L}
\newcommand{\OC}{\overline{\mathcal C}}

\newcommand{\bq}{\begin{equation}}
\newcommand{\eq}{\end{equation}}
\newcommand{\ba}{\begin{aligned}}
\newcommand{\ea}{\end{aligned}}
\newcommand{\be}{\begin{enumerate}}
\newcommand{\ee}{\end{enumerate}}
\newcommand{\bsm}{\left(\begin{smallmatrix}}
\newcommand{\esm}{\end{smallmatrix}\right)}                   
\newcommand{\bpm}{\begin{pmatrix}}
\newcommand{\epm}{\end{pmatrix}}
\newcommand{\overliner}{\begin{displaymath}\begin{array}{cccc}}
\newcommand{\earr}{\end{array}\end{displaymath}}
\newcommand{\overlinerl}{\begin{displaymath}\begin{array}{lcl}}
\newcommand{\earrl}{\end{array}\end{displaymath}}
\newcommand{\overlinel}{\begin{displaymath}\begin{array}{l}}
\newcommand{\earl}{\end{array}\end{displaymath}}
\newcommand{\bxym}{ \begin{displaymath}\xymatrix }
\newcommand{\exym}{\end{displaymath}}
\newcommand{\bcd}{\begin{center}\begin{tikzcd}}
\newcommand{\ecd}{\end{tikzcd}\end{center}}

\newcommand{\tr}{{\rm tr}}
\newcommand{\Isom}{\text{Isom}}
\newcommand{\Spec}{\underline{\operatorname{Spec}}}
\newcommand{\Proj}{\operatorname{Proj}}
\newcommand{\Pic}{\operatorname{Pic}}
\newcommand{\Hom}{\operatorname{Hom}}
\newcommand{\dist}{\operatorname{dist}}
\newcommand{\hhom}{\mathcal{H}\!om}
\newcommand{\Aut}{\operatorname{Aut}}
\newcommand{\Exc}{\operatorname{Exc}}
\newcommand{\lev}{\operatorname{lev}}
\newcommand{\ev}{\operatorname{ev}}
\newcommand{\id}{{\rm id}}

\theoremstyle{plain}
\newtheorem{thm}{Theorem}[section]

\newtheorem{lem}[thm]{Lemma}

\newtheorem*{thm*}{Theorem}
\newtheorem*{matteo*}{Main Theorem}

\theoremstyle{definition}

\newtheorem{remark}[thm]{Remark}

\makeatletter
\newcommand{\doublewidetilde}[1]{{%
  \mathpalette\double@widetilde{#1}%
}}
\newcommand{\double@widetilde}[2]{%
  \sbox\z@{$\m@th#1\widetilde{#2}$}%
  \ht\z@=.9\ht\z@
  \widetilde{\box\z@}%
}
\makeatother

\newcommand{\etalchar}[1]{$^{#1}$}

\newtheorem{innercustomgeneric}{\customgenericname}
\providecommand{\customgenericname}{}
\newcommand{\newcustomtheorem}[2]{%
  \newenvironment{#1}[1]
  {%
   \renewcommand\customgenericname{#2}%
   \renewcommand\theinnercustomgeneric{##1}%
   \innercustomgeneric
  }
  {\endinnercustomgeneric}
}

\newcustomtheorem{customthm}{Theorem}
\newcustomtheorem{customlemma}{Lemma}
\newcustomtheorem{customrmk}{Remark}

\def\Dk{\tikz[baseline=-3pt]{
\draw (0,0)node[left]{\scriptsize $g=2,d=0$} -- (1,.5) node[right]{\scriptsize $g=0,d_1$} (0,0)--(1,-.5)node[right]{\scriptsize $g=0,d_k$} (0,0) -- (1.2,.2)node[right]{\scriptsize $g=0,d_2$};
\draw (1.2,-.2) node[right]{\scriptsize $\ldots$};
\draw (0,0) circle(2pt)[fill=white];
\fill (1,.5) circle (2pt) (1,-.5) circle (2pt) (1.2,.2) circle (2pt)}
}

\def\Ek{\tikz[baseline=-3pt]{
\draw (-1.5,0) node[left]{\scriptsize $g=1,d_0$} -- (0,0);
\draw (0,0)node[above]{\scriptsize \begin{tabular}{c} $g=1,$ \\ $d=0$ \end{tabular}} -- (1,.5) node[right]{\scriptsize $g=0,d_1$} (0,0)--(1,-.5)node[right]{\scriptsize $g=0,d_k$} (0,0) -- (1.2,.2)node[right]{\scriptsize $g=0,d_2$};
\draw (1.2,-.2) node[right]{\scriptsize $\ldots$};
\draw (0,0) circle(2pt)[fill=white];
\fill (-1.5,0) circle (2pt) (1,.5) circle (2pt) (1,-.5) circle (2pt) (1.2,.2) circle (2pt)}
}

\def\Ekk{\tikz[baseline=-3pt]{
\draw (-3,0)node[above]{\scriptsize \begin{tabular}{c} $g=1,$ \\ $d=0$ \end{tabular}} -- (-4,.5) node[left]{\scriptsize $g=0,d_{1,1}$} (-3,0)--(-4,-.5)node[left]{\scriptsize $g=0,d_{1,k_1}$} (-3,0) -- (-4.2,.2)node[left]{\scriptsize $g=0,d_{1,2}$};
\draw (-3,0) -- (-1.5,0) node[below]{\scriptsize $g=0,d_0$} -- (0,0);
\draw (0,0)node[above]{\scriptsize \begin{tabular}{c} $g=1,$ \\ $d=0$ \end{tabular}} -- (1,.5) node[right]{\scriptsize $g=0,d_{2,1}$} (0,0)--(1,-.5)node[right]{\scriptsize $g=0,d_{2,k_2}$} (0,0) -- (1.2,.2)node[right]{\scriptsize $g=0,d_{2,2}$};
\draw (1.2,-.2) node[right]{\scriptsize $\ldots$} (-4.2,-.2) node[right]{\scriptsize $\ldots$};
\draw (0,0) circle(2pt)[fill=white] (-3,0) circle(2pt)[fill=white];
\fill (-1.5,0) circle (2pt) (1,.5) circle (2pt) (1,-.5) circle (2pt) (1.2,.2) circle (2pt) (-4,.5) circle (2pt) (-4,-.5) circle (2pt) (-4.2,.2) circle (2pt)}
}

\def\pacman{\tikz[baseline=-3pt]{
\draw (1,0) -- (2,.5) node[right]{\scriptsize $g=0, d_1$} (1,0)--(2,-.5)node[right]{\scriptsize $g=0, d_k$} (1,0) -- (2.2,.2)node[right]{\scriptsize $g=0, d_2$};
\draw (2.2,-.2) node[right]{\scriptsize $\ldots$};
\draw (0,0)node[left]{\scriptsize $g=0,d_0$} to[out=20,in=160] (1,0) node[above]{\scriptsize \begin{tabular}{c} $g=1,$ \\ $d=0$ \end{tabular}};
\draw (0,0) to[out=-20,in=-160] (1,0);
\fill (0,0) circle (2pt);
\draw[fill=white] (1,0) circle (2pt);
\fill  (2,.5) circle (2pt) (2,-.5) circle (2pt) (2.2,.2) circle (2pt)}
}

\def\hypell{\tikz[baseline=-3pt]{
\draw (0,0)node[left]{\scriptsize $g=2,d_0=2$} -- (1,.5) node[right]{\scriptsize $g=0, d_1$} (0,0)--(1,-.5)node[right]{\scriptsize $g=0, d_k$} (0,0) -- (1.2,.2)node[right]{\scriptsize $g=0, d_2$};
\draw (1.2,-.2) node[right]{\scriptsize $\ldots$};
\draw (0,0) circle(2pt)[fill=gray!50];
\fill (1,.5) circle (2pt) (1,-.5) circle (2pt) (1.2,.2) circle (2pt)}
}

\newcommand{\todo}[1]{\vspace{5mm}\par \noindent
\framebox{\begin{minipage}[c]{0.95 \textwidth} \tt #1\end{minipage}} \vspace{5mm} \par}

\def\ti{-\allowhyphens}
\newcommand{\thismonth}{\ifcase\month 
  \or January\or February\or March\or April\or May\or June%
  \or July\or August\or September\or October\or November%
  \or December\fi}
\newcommand{\thismonthyear}{{\thismonth} {\number\year}}
\newcommand{\thisdaymonthyear}{{\number\day} {\thismonth} {\number\year}}

\title{A geographical study of $\overline{\pazocal M}_{2}(\PP^2,4)^{\text{main}}$}
\author{Luca Battistella and Francesca Carocci}

\begin{document}

{\setstretch{1.3}
\begin{abstract}
We discuss criteria for a stable map of genus two and degree $4$ to the projective plane to be smoothable, as an application of our modular desingularisation of $\overline{\pazocal M}_{2,n}(\PP^r,d)^{\text{main}}$ via logarithmic geometry and Gorenstein singularities.
\end{abstract}}

\maketitle
\tableofcontents

\vspace{-1cm}

\section{Theory}

The moduli space of stable maps to a projective variety $X$ compactifies the space of smooth curves in $X$ by allowing maps from nodal curves. $\M{0}{n}{\PP^r}{d}$ is a smooth connected DM stack with a normal crossing boundary divisor - paralleling the Deligne--Mumford--Knudsen compactification of the moduli space of curves. On the other hand, the geometry of $\M{g}{n}{X}{\beta}$ is in general hard to describe. $\M{g}{n}{\PP^r}{d}$ is fundamental in that it contains $\M{g}{n}{X}{\beta}$ whenever $i\colon X\hookrightarrow\PP^r$ and $i_*\beta=d\ell$, where $\ell$ is the class of a line. In positive genus, the moduli space of stable maps to projective space consists of various components, and it is not of pure dimension, which prompted the development of virtual fundamental classes. If the degree is higher than the Riemann-Roch bound $d>2g-2$, the locus of maps $f\colon C\to \PP^r$ with $C$ smooth is smooth and connected, giving rise to a \emph{main component}. Its dimension can be computed via standard deformation theory, and is known as the \emph{expected dimension} of the moduli space (it is the dimension of the virtual class). Boundary components - usually larger - arise when $C$ is reducible and the line bundle $f^*\OO_{\PP^r}(1)$ is special on a subcurve of $C$. Originating from a closure, the main component \emph{does not} have a natural modular interpretation; in fact, it is in general hard even to describe its geometric points. In genus one, this task was carried out by R. Vakil and A. Zinger: the general point of a boundary component of $\M{1}{n}{\PP^r}{d}$ represents a map from $(E,q_1\ldots,q_k)\sqcup_\mathbf{q}\bigsqcup (R_i,q_i)$, where $E$ is an elliptic curve and $R_i\simeq\PP^1$ are rational tails, such that $E$ is contracted and the rational tails have positive degree (corresponding to any partition of $d$ into $k$ parts). The condition for such a map to be smoothable is that the lines $\{\operatorname{d}\!f(T_{R_i,q_i})\}_{i=1}^k$ span a subspace of dimension at most $k-1$ in $T_{\PP^r,f(E)}$. The classical example is that of plane cubics \cite{VZann}. 

A modern viewpoint on this problem has been developed by D. Ranganathan, K. Santos-Parker, and J. Wise \cite{RSPW1}: though the main component does not have a modular interpretation, they constructed a birational model dominating it which has one (it is a reincarnation of Vakil-Zinger's desingularisation). Roughly speaking, it adds to a stable map $f\colon C\to \PP^r$ the data of a contraction $C\to\overline C$ to a curve with an isolated Gorenstein singularity such that $f$ factors through $\overline C$ and the factorisation $\overline{f}\colon\overline{C}\to\mathbb P^r$ does not contract the minimal genus one subcurve (\emph{core}) of $\overline{C}$. The inspiration is tropical: the dual graph $\tropC$ of $C$ is a rational forest emanating from the core; we can think of $C\to\overline C$ as contracting a disc  around the core.
The locations where the edges of the tropicalization of $C$ cross the disc
determine a logarithmic modification of the moduli space of curves, and in turn of stable maps. This construction builds on work of D. Smyth \cite{SMY1}, who classified isolated Gorenstein singularities of genus one, and their semistable models. These singularities consist of the cusp, and the singularity you obtain by imposing a general linear dependence relation on the tangent vectors to a rational $m$-fold point (the union of the coordinate axes in $\Aaff^m$). Smyth called them \emph{elliptic $m$-fold points}. We can thus rephrase smoothability as follows: a stable map $f\colon C\to \PP^r$, contracting the minimal ellitic subcurve and having $k$ rational tails of positive degree, is smoothable if and only if there exists an elliptic $k$-fold point $\overline {C}$  through which $f$ factors without contracting the core of $\overline{C}$.
 A characterisation of smoothability for genus one stable maps in this spirit is also given in \cite[Prop.~2.6]{BCM}.

Similar but more robust techniques have led us to construct a modular desingularisation of $\M{2}{n}{\PP^r}{d}^{\text{main}}$ \cite{BC}:
\begin{thm*}
 There exists a log smooth and proper DM stack $\VZ(\PP^r,d)$ over $\mathbb C$, with locally free log structure and a birational morphism to  $\M{2}{n}{\PP^r}{d}^{\text{main}}$, parametrising:
 \begin{itemize}[leftmargin=.5cm]
  \item a pointed admissible hyperelliptic cover \[\psi\colon (C,D_R,p_1,\ldots,p_n)\to(T,D_B,\psi(p_1),\ldots,\psi(p_n))\]
  where $C$ is a prestable curve of arithmetic genus two, $D_R$ and $D_B$ are length six (ramification and branch) divisors, and $(T,D_B+\psi(\mathbf{p}))$  is a stable rational tree;
  \item a map $f\colon C\to \PP^r$;
  \item a bubbling destabilisation $C\leftarrow\widetilde C$ and a contraction $\widetilde C\to \overline C$ to a curve with Gorenstein singularities, such that $f$ factors through $\bar{f}\colon\overline C\to \PP^r$, and the latter is non-special on any subcurve of $\overline C$.
 \end{itemize}
\end{thm*}
Upon studying the image of $\VZ(\PP^r,d)$ in $\M{2}{n}{\PP^r}{d}$, we can therefore deduce which stable maps from a curve of genus two are smoothable. In the present note, we shall do so in the case of degree $4$ curves in $\PP^2$.

\subsection{Singularities}
There are two families of isolated Gorenstein singularities of genus two \cite{B}. Those with at most three branches are classically known among the (planar) ADE singularities. In the following table, we include a parametrisation as a subalgebra of $\widetilde{R}=\mathbb C[\![t_1]\!]\times \ldots\times\mathbb C[\![t_m]\!]$ ($m$ represents the number of branches) and local equations for each of them (here $[k]$ denotes the set $\{1,\ldots,k\}$). In the table below we give explicit  generators $x_i$ for the singularity; these are considered modulo  $\mathfrak m_{\widetilde{R}}^4$ for type I singularities, and modulo $\mathfrak m_{\widetilde{R}}^3$ for type II singularities.

\smallskip

 \begin{tabular}{c||c|c}
  & type I & type I\!I \\
  \hline
  \hline
  Parametr. & 
  \begin{minipage}{.37\textwidth}{
 \begin{align*}
 \begin{split}
  x_1= & \, t_1\oplus0\oplus\ldots\oplus t_m^3\\
  x_2= & \, 0\oplus t_2\oplus\ldots\oplus t_m^3\\
  &\ldots\\
  x_{m-1}= & \, 0\oplus\ldots\oplus t_{m-1}\oplus t_m^3\\
  x_m= & \, 0\oplus\ldots\oplus0\oplus t_m^2\\
  &\pmod{\langle t_1^4,\ldots,t_m^4\rangle}
  \end{split}
 \end{align*}}
  \end{minipage}
&
  \begin{minipage}{.39\textwidth}{
   \begin{align*}
 \begin{split}
  x_1= & \, t_1\oplus0\oplus\ldots\oplus t_m\\
  x_2= & \, 0\oplus t_2\oplus\ldots\oplus t_m^2\\
  &\ldots\\
  x_{m-1}= &  \, 0\oplus\ldots\oplus t_{m-1}\oplus t_m^2 \\
  (y= & \, 0\oplus t_2^3 \quad\text{if } m=2)\\
  &\pmod{\langle t_1^3,\ldots,t_m^3\rangle}.
 \end{split}
 \end{align*}}
  \end{minipage}
\\
\hline
  Equations & 
  \begin{minipage}{.37\textwidth}{
  
  \begin{description}
  \item[$m=1$] $x^5-y^2$ ($A_4$);
  \item[$m=2$] $x_2(x_2^3-x_1^2)$ ($D_5$);
  \item[$m=3$] $\langle x_3(x_1-x_2),x_3^3-x_1x_2\rangle$;
  \item[$m\geq 4$] $\langle x_m^3-x_1x_2,\\x_i(x_j-x_k) \rangle_{\substack{i\in[m];\\j,k\in[m-1]\setminus\{i\}}}$
   \end{description}} 
  \end{minipage}
&
  \begin{minipage}{.39\textwidth}{
  
  \begin{description}
  \item[$m=2$] $y(y-x_1^3)$ ($A_5$);
  \item[$m=3$] $x_1x_2(x_2-x_1^2)$ ($D_6$);
  \item[$m\geq 4$] $\langle x_3(x_1^2-x_2),\\ x_i(x_j-x_k)\rangle_{\substack{i\in[m-1];\\j,k\in[m-1]\setminus\{i\}}}$
   \end{description}} 
  \end{minipage}
 \end{tabular}

\smallskip 

Notice that every singularity of type $I$ with $m\geq 2$ contains a genus one cusp, while every singularity of type $I\!I$ with $m\geq 3$ contains a tacnode; we call the corresponding branches \emph{special}. It will be relevant to our discussion that the special branch(es) of a type $I$ (resp. type $I\!I$) singularity  $\overline{C}$ has pre-image in $C$   attached to a Weierstrass (resp. two conjugate) point(s) of the contracted genus two subcurve, as follows from a semistable reduction analysis \cite[Prop~4.3,4.4]{B}.

In genus two, we are led to consider factorisation through non-reduced curves as well. Classically, a \emph{ribbon} is a scheme $R$ with underlying reduced structure $R_{\text{red}}\simeq\PP^1$, and square-zero ideal sheaf $\mathcal I=\mathcal I_{R_{\text{red}}/R}$ such that $\mathcal I$ is a line bundle on $R_{\text{red}}$. For us, an $(m_1,\ldots,m_k)$\emph{-tailed ribbon} consists of a ribbon of genus $2-k$ glued at $k$ general double points with a rational $m_i$-fold point ($i=1,\ldots,k$), see \cite[Example 2.21]{BC}. There are exact sequences:
 \begin{equation}0\to \OO_C\to \OO_R\oplus\bigoplus_{i=1}^k\OO_{\PP^1}^{\oplus m_i}\to \bigoplus_{i=1}^k(\mathbb C^{m_i-1}\oplus\mathbb C[\epsilon])\to 0,\end{equation}
 and
 \begin{equation}0\to\mathcal I=\OO_{\PP^1}(k-3)\to\OO_R\to\OO_{\PP^1}\to0.\end{equation}
                                                                        
 Local equations around an $m$-tailed point are given by:
 \[\mathbb C[\![x_1,\ldots,x_m,y]\!]/(x_ix_j,(x_i-x_j)y)_{1\leq i <j\leq m}.\]

 Roughly speaking, the general one-parameter smoothing of a ribbon is not normal (it is singular along $R$), and its semistable model is the normalisation, restricting to the hyperelliptic cover over the ribbon. In particular, the tail attaching points must correspond under the hyperelliptic cover.

\subsection{Irreducible components of $\overline{\pazocal M}_2(\PP^r,d)$}\label{sec:components}

We draw the weighted dual graph of the general member of all possible irreducible components of $\overline{\pazocal{M}}_2(\PP^r,d)$.
Our running convention is that a white vertex corresponds to a contracted component, a gray one to a genus two subcurve covering a line two-to-one, and a black vertex corresponds to a non-special subcurve. Vertices are labelled with their genus and weight.

\begin{enumerate}[leftmargin=.6cm]
 \item \emph{main} is the closure of the locus of maps from a smooth curve of genus two;
 \item $\pazocal D^{(d_1,\ldots,d_k)}=\left\{\Dk\right\}$
 \item ${}^{\rm{hyp}}\!\pazocal D^{(d_1,\ldots,d_k)}=\left\{\hypell\right\}$
 \item ${}^{d_0}\!\pazocal E^{(d_1,\ldots,d_k)}=\left\{\Ek\right\}$
 \item ${}^{(d_{1,1},\ldots,d_{1,k_1})}\!\pazocal E^{d_0}\pazocal E^{(d_{2,1},\ldots,d_{2,k_2})}=\left\{\Ekk\right\}$
 \item ${}^{\rm{br}=d_0}\!\pazocal E^{(d_1,\ldots,d_k)}=\left\{\pacman\right\}$
\end{enumerate}
This is taken from the first author's PhD thesis \cite{Bthesis}, and is implicit in \cite{HLN}.

In order to compute the dimension of the boundary components, it is useful to recall that they are the  image under clutching of products (fibered over $\PP^r$)  of irreducible components of  stable map  spaces of lower genus and/or degree. Since the gluing morphisms are finite, such a description allows us to compute the dimensions. 
\begin{itemize}[leftmargin=.6cm]
\item   For completeness, we recall that the dimension of \emph{main} coincides with the virtual dimension $\text{vdim}=(3-r)+d(r+1)$.
\item $\pazocal D^{(d_1,\ldots,d_k)}$ is the image under gluing of \[\oM_{2,k}\times\M{0}{1}{\PP^r}{d_1}\times_{\PP^r}\M{0}{1}{\PP^r}{d_2}\times_{\PP^r}\dots \times_{\PP^r}\M{0}{1}{\PP^r}{d_k}\]
 where the maps to $\mathbb P^r$ are the evaluation morphisms. Since all the spaces are smooth and the evaluation maps are flat, we can easily compute that the dimension of this component is $\dim_{(2)}=d(r+1) +r-k+3$
\item ${}^{\rm{hyp}}\!\pazocal D^{(d_1,\ldots,d_k)}$ is the image under gluing of
 \[\oM^{\text{main}}_{2,k}(\mathbb P^r,2)\times_{\PP^r}\M{0}{1}{\PP^r}{d_1}\times_{\PP^r}\M{0}{1}{\PP^r}{d_2}\times_{\PP^r}\dots \times_{\PP^r}\M{0}{1}{\PP^r}{d_k};\] 
  the component $\oM^{\text{main}}_{2,k}(\mathbb P^r,2)$ is actually smooth, but of dimension higher than the expected one, i.e. $2r+4+k$ ( we have to choose a line in $\mathbb P^r$, the $6$ ramification points, and the markings) instead of $r+5+k$ (which comes out of the virtual dimension computation); all the other spaces are smooth and so we can compute the dimension: $\dim_{(3)}= d(r+1)-k+2.$ 
 \item ${}^{d_0}\!\pazocal E^{(d_1,\ldots,d_k)}$ is the image of 
  \[\oM_{1,k+1}\times\oM^{\text{main}}_{1,1}(\mathbb P^r,d_0)\times_{\PP^r}\M{0}{1}{\PP^r}{d_1}\times_{\PP^r}\M{0}{1}{\PP^r}{d_2}\times_{\PP^r}\dots \times_{\PP^r}\M{0}{1}{\PP^r}{d_k};\]
  while $\M{1}{1}{\PP^r}{d_0}$ has several components of different dimensions, the main component is irreducible of the expected dimension; notice that we do not have to consider what happens when the maps degenerate further and becomes of degree zero on the genus one curve of weight $d_0$ since these cases are included in the irreducible component  $\pazocal D^{(d_1,\ldots,d_k)}$; the dimension is $\dim_{(4)}= d(r+1)-k+2.$ 
  \item The description for the remaining two types of components is analogous and thus we omit it; again the only remark to make for the last case is that for the genus one curve carrying weight $d_0$ we can restrict to main. The dimensions are  respectively $\dim_{(5)}= d(r+1)+r-(k_1+k_2)+1$ and    $\dim_{(6)}= d(r+1)-k+1$
\end{itemize}

\subsection{The algorithm to compute the intersection with main}

Given the  smooth modular compactification  constructed in \cite{BC}, we can characterise the closed points of  $\overline{\pazocal M}^{\text{main}}_{2,n}(\PP^r,d)$ as those in the image of the map $\VZ(\PP^r,d)\to \overline{\pazocal M}_{2,n}(\PP^r,d).$ 
To determine whether a closed point $[f\colon C\to\PP^r]$  such that $H^1(C,f^*\OO_{\PP^r}(1))\neq 0$ lies in the boundary of the \emph{main} component we proceed as follows:

\begin{enumerate}[label=(\roman*), leftmargin=.6cm]

\item Promote $C$ to a hyperelliptic admissible cover $C'\xrightarrow{\psi} T$ over the standard log point $\operatorname{Spec}(\mathbb N\to\mathbb C)$ \emph{in all possible ways}, where $C'$ is obtained by sprouting from $C$, and $T$ is a rational tree.

\begin{remark}\label{rem:lifttoadmissiblecover}
Enhancing the prestable curve $C$ to a hyperelliptic admissible cover over the standard log point involves a series of choices
of both an algebro-geometric and a tropical nature. First of all, there is in general more than one way to destabilise the curve $C$  to make it into the source of a hyperelliptic admissible cover $C'\to T$: combinatorially, this amounts  to choosing which irreducible components support the Weierstrass points; where exactly these lie gives some continuous moduli - notice that, given a generic smoothing of $C$, all this information is determined. Once the combinatorial type is fixed, it determines a cone in the \emph{tropical moduli space} of admissible covers: the cone is isomorphic to $\mathbb R_{\geq 0}^{|E(T)|}$, where $\tropT$ is the dual graph of $T$, and $E(\tropT)$ its edge set. In \cite{BC} we construct a \emph{finite} polyhedral subdivision of this cone such that the combinatorial and analytic type of the bubbling destabilization $\widetilde{C}\to C'$ and the Gorenstein contraction $\widetilde{C}\to\overline{C}$ are constant on every cone. The classically minded reader may think that the contraction depends on the ratios among the smoothing parameters of the nodes, although in a rather manageable fashion.
Let us remark that, by simply looking at the analytic type of the contraction $\widetilde{C}\to\overline{C} $ and at the induced weight on $\overline{C}$, we can often identify some analytic types through which factorisation is not possible, and thus 
discard several strata of the log modification $\widetilde{\pazocal A}_{2}$.
\end{remark}
 
\item  Check if the sections of $\operatorname{H}^0(\widetilde{f}^*\mathcal O_{\mathbb P^r}(1))$ descend to $\overline{C}$,where we have denoted by $\widetilde{f}$ the map induced by $f$ on the destabilisation $\widetilde{C}\to C$.
\end{enumerate}

By construction, if $f$ descends to $\overline C$, its degree is sufficiently high that it is unobstructed over the moduli space of aligned admissible covers (see the next lemma); since the latter is smooth, smoothability of the map follows from that of the curve.

For the benefit of the reader, we recall some conditions ensuring that a map from a Gorenstein curve $\bar f\colon\overline C\to\PP^r$ satisfies $H^1(\overline C,\bar{f}^*\OO_{\PP^r}(1))=0$ 
(together with the conditions for a map to descend to a ribbon).
  \begin{lem}\cite[Lemma~2.33, 2.34]{BC}\label{lem:unobstructdness}
  The obstructions $H^1(\overline C,\bar{f}^*\OO_{\PP^r}(1))$ vanish if:
 \begin{enumerate}[leftmargin=.5cm]
  \item \emph{Isolated singularities.} $\bar{f}$ has positive degree on every subcurve of genus one (so in particular on
   the special branches of a genus two singularity), and has degree at least three on every subcurve of genus two.
  \item \emph{Non-reduced curves.} If the ribbon $R$ is contracted, there are at least two tails (with distinct attaching points) of positive degree; moreover, if $k$ is the number of such tails, the images of their tangent vectors in $T_{\PP^r,\bar f(R)}$ span a linear subspace of dimension at most $k-2$.
  
  \noindent If the ribbon $R$ is hyperelliptic (i.e. mapped by $\bar{f}$  two-to-one to a line $\ell\subseteq\PP^r$), there is at least one tail of positive degree; moreover, if $k$ is the number of such tails, the images of their tangent vectors span a linear subspace of dimension at most $k-1$ after projecting away from $\ell$.
 \end{enumerate}
\end{lem}

Let $\pazocal Z$ denote an irreducible component of $\overline{\pazocal{M}}_2(\PP^r,d)$ (see \S \ref{sec:components} for a complete list), and let $\pazocal Z^\circ$ denote the open locus where the weighted dual graph of the curve is generic, i.e. as depicted in the previous section.

\begin{lem}\label{lemma:irrcomponents}
 Every component of $\overline{\pazocal{M}}_2(\PP^r,d)^\text{main}\cap\pazocal Z$ is in the closure of a component of $\overline{\pazocal{M}}_2(\PP^r,d)^\text{main}\cap \pazocal Z^\circ$.
\end{lem}
\begin{proof}
For the proof of this Lemma we referfreely to the results of \cite{BC}.
Let $\widetilde{\mathcal A}_2(\PP^r,d)$ be the space of all aligned admissible maps.  The main component of the latter is $\VZ(\PP^r,d)$, and it is unobstructed over the moduli space of aligned weighted admissible covers $\widetilde{\mathcal A}_2$.
Let $\widetilde{\pazocal Z}$ be a boundary component of $\widetilde{\mathcal A}_2(\PP^r,d)$. Note that this space may have more irreducible components than the space of stable maps, due to the moduli of lifting $C$ to an admissible cover $C'\to T$.
Still, it is enough to argue that $\widetilde{\pazocal Z}\cap \VZ(\PP^r,d)$  is contained in the closure of  $\widetilde{\pazocal Z}^{\circ}\cap \VZ(\PP^r,d)$.

So, let us consider $[\widetilde{f}]\in\widetilde{\pazocal Z}\cap \VZ(\PP^r,d)$. We observe that not all nodes of the source curve can be smoothed independently due to the alignment. On the other hand, all the nodes corresponding to edges of the dual graph of the source curve which are identified by the aligning function $\lambda$ (see \cite[\S~3]{BC}) can be smoothed simultaneously. In particular, 
 all the nodes corresponding to edges of the dual graph which are entirely contained in the strict interior of a connected component of the support of $\lambda$, respectively outside its support,  can be smoothed simultaneously.

 Tropically, smoothing has the effect of setting the corresponding edge lengths to zero. We are left with a generic dual graph, in which every connected component of the support of $\lambda$ contains a unique vertex, and all remaining vertices are on the circle and have positive weight. Comparing with \S \ref{sec:components} this concludes the proof.
\end{proof}

\section{Practice}
\subsection{Classification of singular plane quartics}
Our analysis is based on Chung-Man Hui's thesis \cite{Hui}, in which he classified the singularities of plane quartics, their normal forms, and the dimension of the corresponding strata.  For the benefit of the reader, we collect the results of \cite{Hui}  in
the table below. 
From left to right we indicate: the geometric genus (in the reducible case the geometric genus of the irreducible components), the number of singular points, the analytic singularities that can occur, and the dimension of the strata in the linear system of quartics where such singularities appear; we use these in the computation of $\dim \pazocal Z\cap\overline{\pazocal M}^{\text{main}}$ below.

\smallskip 

 \begin{tabular}{c||c|c|c|c}
  & $g_g$ & $|C^{\text{sing}}|$ & Analytic Sing & dim strata\\
  \hline
  \hline
   & 2 & 1 & $A_1, \; A_2$ & $13,\; 12$\\

   & 1 & 1 & $A_3,\; A_4$ & $11,\; 10$\\
     
irreducible    & 1 & 2& $A_1^2,\; A_1A_2,\; A_2^2$ & $12,\; 11,\; 10$\\
      
   quartics & 0 & 1& $A_5,\; A_6,\; D_4,\; D_5,\; E_6$ &  $9,\; 8,\; 10,\; 9,\; 8$\\
    
      & 0 & 2& $A_1A_3,\; A_2A_3,\; A_1A_4,\; A_2A_4$ &  $10,\; 9,\; 9,\; 8$\\
          
      & 0 & 3& $A_1^3,\; A_1^2A_2,\; A_1A_2^2,\; A_2^3$ &  $11,\; 10,\; 9,\; 8$\\
  \hline
  \hline
  
    & $1\sqcup 0$& 1 & $A_5$ & $9$\\
    & $1\sqcup 0$& 2 & $A_1A_3$ & $10$\\
 cubic   & $1\sqcup 0$& 3 & $A_1^3$ & $11$\\
   and & $0\sqcup 0$& 1 & $D_6, \; E_7$ & $8,\; 7$\\
  line  & $0\sqcup 0$& 2 & $A_1D_4,\; A_1A_5,\; A_1D_5,\;A_2A_5$ & $9,\; 8,\;8,\;  7$\\
    & $0\sqcup 0$& 3& $A_1^2A_3,\; A_1A_2 A_3$ & $9,\; 8$\\
     & $0\sqcup 0$& 4& $A_1^4,\; A_1^3A_2$ & $10,\; 9$\\
  \hline
  \hline
  
  &$0\sqcup 0$& 1 & $A_7$ & $ 7$\\
 two&  $0\sqcup 0$& 2 & $A_1A_5,\; A_3^2$ & $ 8,\; 8$\\
conics   & $0\sqcup 0$& 3 & $A_1^2A_3$ & $ 8$\\
   & $0\sqcup 0$& 3 & $A_1^4$ & $ 10$\\
  \hline
  \hline
  & $0\sqcup 0\sqcup 0$& 2 & $A_1D_6,$ & $ 7$\\
 conic and & $0\sqcup 0\sqcup 0$& 3 & $A_1^2D_4,$ & $ 8$\\
two lines  & $0\sqcup 0\sqcup 0$& 4 & $A_1^3A_3$ & $ 8$\\
    & $0\sqcup 0\sqcup 0$& 5 & $A_1^5$ & $ 9$\\
   \hline
  \hline
    & $0\sqcup 0\sqcup 0 \sqcup 0$& 1 &planar $4$-fold & $ 6$\\
  four lines & $0\sqcup 0\sqcup 0 \sqcup 0$& 4 & $A_1^3D_4$& $ 7$\\
   & $0\sqcup 0\sqcup 0 \sqcup 0$& 6 & $A_1^4$ & $ 8$\\

  \end{tabular}
  \smallskip 
  
  Moreover, the following non-reduced singularities appear: a conic and generic double line ( this stratum has $\dim=7$); a conic and a tangent double line ( this stratum has $\dim=6$); three generic lines, one of which is a double line, ( this stratum has $\dim=6$);
  three concurrent lines,  one of which is a double line, ( this stratum has $\dim=5$); a double conic  (this stratum has $\dim=5$); two double lines ( this stratum has $\dim=4$); two lines, one of which is triple ( this stratum has $\dim=4$); a quadruple line  (this stratum has $\dim=2$).

\subsection{Applying the algorithm to  $\overline{\pazocal M}_2(\PP^2,4)$ }

The moduli space $\overline{\pazocal M}_2(\PP^2,4)$ has more than twenty irreducible components; in this note we analyse them and their intersection with \emph{main}. In the illustrations we adopt the following \emph{convention}: gray components are contracted, dashed lines come from sprouting, blue points are branch points of a finite cover, blue components correspond to the hyperelliptic cover of a genus two curve.
We warn the reader that, for the sake of conciseness, after giving full details for the computations in the first couple of instances, we will largely omit the parts which are similar in the remaining cases.

\begin{description}[leftmargin=0cm]
\setlength\itemsep{0.5cm}
 \item[Main] Plane quartics have arithmetic genus $3$. The main component of $\overline{\pazocal M}_2(\PP^2,4)$ parametrises (the normalisation of) nodal quartics. It has dimension $13$.
 
 \begin{figure}[h]

   \begin{tikzpicture}[scale=.4]
   
    
 \coordinate (A) at (2,4);
 \coordinate (B) at (1.7,3.8);
 \coordinate (C) at   (2,3.6);
 \coordinate (D) at (2.3,3.4) ;
  \coordinate (E) at (2,3.2) ;
  \coordinate (F) at (1.7,3);
  \coordinate (G) at  (2,2.8);
  \coordinate (H1) at (2.3,4.3);
 \coordinate (H2) at (2.4,4.8);
\coordinate (J1) at (2.3,2.5);
 \coordinate (J2) at (2.4,2);
 
 \coordinate (Fs) at (3,3.5);
  \coordinate (Ft) at (4.6,3.5);
   \coordinate (Flabel) at (3.8,3.5);

\begin{scope}[every coordinate/.style={shift={(0,0)}}]
\draw[thick] ([c]A) .. controls ([c]B) ..
       ([c]C) .. controls ([c]D)  .. ([c]E) .. controls ([c]F).. ([c]G);
      \draw[thick]  plot [smooth, tension=3] coordinates { ([c]A) ([c]H1) ([c]H2)};
       \draw[thick]  plot [smooth, tension=3] coordinates { ([c]G) ([c]J1) ([c]J2)};
       \node at ([c]J2) [below] {\scriptsize{$g=2$}};
       \draw[->,thick] ([c]Fs) to ([c]Ft);
        \node at ([c]Flabel) [above] {\scriptsize{$f$}};
       \end{scope}
 \coordinate(Q1) at  (1,5);
 \coordinate (Q2) at (10,5);
 \coordinate (Q3) at (10,1 );
 \coordinate (Q4) at (1,1);

 \draw (Q1)-- ( Q2)--(Q3)--(Q4)--(Q1);        

 \coordinate (Am) at (0,0);
 \coordinate (Bm) at(0.2,-.3);
 \coordinate (Cm) at  ( .4,0);
 \coordinate (Dm) at (.6,.3)  ;
  \coordinate (Em) at (.8,0) ;
  \coordinate (Fm) at (1,-.3);
  \coordinate (Gm) at  (1.2,0);
  \coordinate (H1m) at (-.5,.3);
 \coordinate (H2m) at (-.2,.2);
\coordinate (J1m) at (1.4,.2);
 \coordinate (J2m) at (1.7,.3);
 \coordinate (N1m) at (3.2,-.3);
 \coordinate (N2m) at (.2,-.3);
  \coordinate (N3m) at (3,.3);
 \coordinate(P1) at  (-0.2,1);
 \coordinate (P2) at ( 3.5,1);
 \coordinate (P3) at (2.5,-1 );
 \coordinate (P4) at (-1.2,-1);
 
 \begin{scope}[every coordinate/.style={shift={(6,3)}}]
 \draw[thick]  plot [smooth, tension=3] coordinates {([c]H1m) ([c]H2m) ([c]Am)  };
\draw[thick] ([c]Am) .. controls ([c]Bm) ..
       ([c]Cm) .. controls ([c]Dm)  .. ([c]Em) .. controls ([c]Fm).. ([c]Gm);
    \draw[thick]  plot [smooth, tension=3] coordinates { ([c]Gm) ([c]J1m) ([c]J2m)  };   
     \draw[thick]   ([c]J2m).. controls ([c]N1m) and ([c]N2m).. ([c]N3m);
 \draw[dashed] ([c]P1)-- ( [c]P2)--([c]P3)--([c]P4)--([c]P1);  
    \node at ([c]Gm) [below] {\scriptsize{$p_a=3$}};
    \end{scope}       
\end{tikzpicture}
	\caption{$\overline{\pazocal M}_2(\PP^2,4)^{\text{main}}$, $\dim=13$}
	\label{main}
\end{figure}

 \item[${\pazocal D^{(4)}}$ component] The general member contracts the genus two core, and restricts on the rational tail to the normalisation of a three-nodal quartic (the core can be contracted to any point of the latter). The dimension is $16$. 
  The intersection with \emph{main} has two components.
 
   \begin{figure}[h]
   \begin{tikzpicture}[scale=.4]
 \coordinate (A) at (2,4);
 \coordinate (B) at (1.7,3.8);
 \coordinate (C) at   (2,3.6);
 \coordinate (D) at (2.3,3.4) ;
  \coordinate (E) at (2,3.2) ;
  \coordinate (F) at (1.7,3);
  \coordinate (G) at  (2,2.8);
  \coordinate (H1) at (2.3,4.3);
 \coordinate (H2) at (2.4,4.8);
 \coordinate (R1) at (1.9,4.3);
 \coordinate (R2) at (3.7,4.5);
 \coordinate (RT1) at (2.6,4);
 \coordinate (RT2) at (4.2,5.1);

  \coordinate (R1b) at (2.1,2.2);
 \coordinate (R2b) at (3.5,1.8);
 \coordinate (J1) at (2.3,2.4);
 \coordinate (J2) at (2.4,2);
 
   \coordinate(fact1) at (2.4,1);
   \coordinate(fact2) at (3.1,-.4);
   
 \coordinate (Fs) at (3,3.5);
  \coordinate (Ft) at (4.6,3.5);
   \coordinate (Flabel) at (3.8,3.5);
   
    \coordinate (Fbs) at (5.3,-.3);
  \coordinate (Fbt) at (6.4,1);
    \coordinate (Fblabel) at (5.9,.1);
   
    
      \coordinate(T1) at (3,-.9);
    \coordinate(T2) at (5.3,-.9);
    \coordinate (O) at (4,-.9);
    \coordinate (O1) at (3,-1.9);
   \coordinate (O2) at (5,.1);
    \coordinate (O3) at (3.5,-1.03);
   \coordinate (O4) at (4.5,-.77);

    \coordinate(S1) at (3.4,-.9);
    \coordinate(S2) at (3.6,-.9);
    \coordinate(U1) at (5,.5);
    \coordinate(U2) at (5,-1.5);
    \coordinate(TT1) at (4,-.2);
    \coordinate(TT2) at (5.7,-.7);

   \begin{scope}[every coordinate/.style={shift={(0,0)}}]
\draw[draw=gray] ([c]A) .. controls ([c]B) ..
       ([c]C) .. controls ([c]D)  .. ([c]E) .. controls ([c]F).. ([c]G);
      \draw[draw=gray]  plot [smooth, tension=3] coordinates { ([c]A) ([c]H1) ([c]H2)};
       \draw[draw=gray]  plot [smooth, tension=3] coordinates { ([c]G) ([c]J1) ([c]J2)};
       \draw[thick] ([c]R1)node[left]{\scriptsize{$T$}} to ([c]R2);
       \draw[->,thick] ([c]Fs) to ([c]Ft);
        \node at ([c]Flabel) [above] {\scriptsize{$f$}};
             \node at ([c]J2) [below] {\begin{color}{gray}\scriptsize{$Z,g=2$}\end{color}};
 \end{scope}

 
\coordinate(Q1) at  (.5,5.5);
 \coordinate (Q2) at (10,5.5);
 \coordinate (Q3) at (10,1);
 \coordinate (Q4) at (.5,1);
 \coordinate (Qmiddle) at (5.25,1);
 \coordinate (Q3long) at (10,-2.5 );
 \coordinate (Q4long) at (.5,-2.5);
\coordinate (Qlongmiddle) at (5.5,-2.5);

  \begin{scope}[every coordinate/.style={shift={(0,0)}}]     
 \draw (Q1)-- (Q2)--(Q3)--(Q4)--(Q1);   
 \end{scope}


 \coordinate (An) at (-.5,.5);
 \coordinate (B1m) at(2,-.2);
  \coordinate (B2m) at(-1.8,-.2);
 \coordinate (Cn) at  (1,.5);
 \coordinate (D1n) at (3,-.2)  ;
  \coordinate (D2n) at (-.5,-.2) ;
  \coordinate (En) at (2,.5);
  \coordinate (F1n) at  (4,-.2);
   \coordinate (F2n) at (1,-.2);
   \coordinate (Gn) at (3,.5);

 \coordinate(P1) at  (-0.2,1);
 \coordinate (P2) at ( 3.5,1);
 \coordinate (P3) at (2.5,-1 );
 \coordinate (P4) at (-1.2,-1);
 
 \begin{scope}[every coordinate/.style={shift={(6,3)}}]

\draw[thick] ([c]An) .. controls ([c]B1m) and ([c]B2m) ..
       ([c]Cn) .. controls ([c]D1n) and ([c]D2n)  .. ([c]En) .. controls ([c]F1n) and ([c]F2n)  .. ([c]Gn);
   
 \draw[dashed] ([c]P1)-- ( [c]P2)--([c]P3)--([c]P4)--([c]P1);  
    
     \fill[gray] ([c]Cn) circle (3pt);
          \node at ([c]Cn) [above] {\begin{color}{gray}\scriptsize{$f(Z)$}\end{color}};

    \end{scope}

\begin{scope}[every coordinate/.style={shift={(14,0)}}]
\draw[thick] ([c]A) .. controls ([c]B) ..
       ([c]C) .. controls ([c]D)  .. ([c]E) .. controls ([c]F).. ([c]G);
      \draw[thick]  plot [smooth, tension=3] coordinates { ([c]A) ([c]H1) ([c]H2)};
       \draw[thick]  plot [smooth, tension=3] coordinates { ([c]G) ([c]J1) ([c]J2)};
       \draw[->,thick] ([c]Fs) to ([c]Ft);
        \node at ([c]Flabel) [above] {\scriptsize{$f$}};
       \end{scope}

  \begin{scope}[every coordinate/.style={shift={(14,0)}}]     
 \draw ([c]Q1)-- ( [c]Q2)--([c]Q3)--([c]Q4)--([c]Q1);   
 \end{scope}
 

 \coordinate (Am) at (0,0);
 \coordinate (Bm) at(0.2,-.3);
 \coordinate (Cm) at  ( .4,0);
 \coordinate (Dm) at (.6,.3)  ;
  \coordinate (Em) at (.8,0) ;
  \coordinate (Fm) at (1,-.3);
  \coordinate (Gm) at  (1.2,0);
  \coordinate (H1m) at (-.5,.3);
 \coordinate (H2m) at (-.2,.2);
\coordinate (J1m) at (1.4,.2);
 \coordinate (J2m) at (1.7,.3);
 \coordinate (N1m) at (3.2,-.3);
 \coordinate (N2m) at (.2,-.3);
  \coordinate (N3m) at (3,.3);
 \coordinate(P1) at  (-0.2,1);
 \coordinate (P2) at ( 3.5,1);
 \coordinate (P3) at (2.5,-1 );
 \coordinate (P4) at (-1.2,-1);
 
 \begin{scope}[every coordinate/.style={shift={(20,3)}}]
 
 \draw[thick]  plot [smooth, tension=3] coordinates {([c]H1m) ([c]H2m) ([c]Am)  };
 
\draw[thick] ([c]Am) .. controls ([c]Bm) ..
       ([c]Cm) .. controls ([c]Dm)  .. ([c]Em) .. controls ([c]Fm).. ([c]Gm);
    \draw[thick]  plot [smooth, tension=3] coordinates { ([c]Gm) ([c]J1m) ([c]J2m)  };   
     \draw[thick]   ([c]J2m).. controls ([c]N1m) and ([c]N2m).. ([c]N3m);
 \draw[dashed] ([c]P1)-- ( [c]P2)--([c]P3)--([c]P4)--([c]P1);  
    \end{scope}   
       
    \draw (7.5,1)--(8.5,-.5);
\draw (7.5,1)--(19,-.5);

 \node at (12.3,3) [above] { $\cap$};

\node at (5,1) [below] {\scriptsize{$\pazocal D^{(4)}$}};

 \begin{scope}[every coordinate/.style={shift={(3,-6)}}]
\draw[draw=gray] ([c]A) .. controls ([c]B) ..
       ([c]C) .. controls ([c]D)  .. ([c]E) .. controls ([c]F).. ([c]G);
      \draw[draw=gray]  plot [smooth, tension=3] coordinates { ([c]A) ([c]H1) ([c]H2)};
       \draw[draw=gray]  plot [smooth, tension=3] coordinates { ([c]G) ([c]J1) ([c]J2)};
       \draw[thick] ([c]R1) to ([c]R2);
          \draw[dashed] ([c]R1b) to ([c]R2b);
       \draw[->,thick] ([c]Fs) to ([c]Ft);
        \node at ([c]Flabel) [above] {\scriptsize{$f$}};
             \node at ([c]J2) [below] {\begin{color}{gray}\scriptsize{$Z$}\end{color}};
             \draw[->,dotted] ([c]fact1) to ([c]fact2);
             \node at ([c]R1) [left] {\scriptsize{$T$}};
             \node at ([c]R1b) [left] {\scriptsize{$\overline{T}$}};
                  \draw[dashed] ([c]T1) to ([c]T2);
                  \draw[thick]  plot [smooth ] coordinates {([c]O1) ([c]O3) ([c]O)} ;
  \draw[thick]  plot [smooth ] coordinates {([c]O) ([c]O4) ([c]O2) };
            \draw[->,thick] ([c]Fbs) to ([c]Fbt);         
 \node at ([c]Fblabel) [right] {\scriptsize{$\bar{f}$}};
 \end{scope}

  \begin{scope}[every coordinate/.style={shift={(3,-6)}}]     
 \draw ([c]Q1)-- ( [c]Q2)--([c]Q3long)--([c]Q4long)--([c]Q1);
 \node at ([c]Qlongmiddle) [below] {\scriptsize{$\pazocal D^{(4)}\cap\overline{\pazocal M}^{\text{main}}$, general core}};
 \end{scope}
  
\coordinate (AE0) at (1,-.3);
\coordinate (AE1) at (1,0);
\coordinate (AE2) at (1.1,0.12);
\coordinate (AE3) at (.9,0.12);
\coordinate (AE4) at (1.5,0.4);
\coordinate (AE5) at (.5,0.4);
\coordinate (AE6) at (1.67,0.58);
\coordinate (AE7) at (.33,0.58);

\coordinate (AE8) at (0,.7);
\coordinate (AE9) at (2,1);

 \coordinate(P1) at  (-0.2,1);
 \coordinate (P2) at ( 3.5,1);
 \coordinate (P3) at (2.5,-1 );
 \coordinate (P4) at (-1.2,-1);
 
 \begin{scope}[every coordinate/.style={shift={(9,-3)}}]
 \draw[thick] ([c]AE0) --([c]AE1);
 \draw[thick]  plot [smooth ] coordinates {([c]AE1) ([c]AE2) ([c]AE4) ([c]AE6) ([c]AE9)};
  \draw[thick]  plot [smooth ] coordinates {([c]AE1) ([c]AE3) ([c]AE5) ([c]AE7) ([c]AE8)};
\draw[dashed] ([c]P1)-- ( [c]P2)--([c]P3)--([c]P4)--([c]P1); 
     \fill[gray] ([c]AE0) circle (3pt);
      \node at ([c]AE1) [xshift=.5cm] {\begin{color}{gray}\scriptsize{$f(Z)$}\end{color}};
            \node at ([c]AE0) [below,left] {\scriptsize{$E_6$}};

    \end{scope}   

\begin{scope}[every coordinate/.style={shift={(14,-6)}}]
\draw[draw=gray] ([c]A) .. controls ([c]B) ..
       ([c]C) .. controls ([c]D)  .. ([c]E) .. controls ([c]F).. ([c]G);
      \draw[draw=gray]  plot [smooth, tension=3] coordinates { ([c]A) ([c]H1) ([c]H2)};
       \draw[draw=gray]  plot [smooth, tension=3] coordinates { ([c]G) ([c]J1) ([c]J2)};
       \draw[thick] ([c]R1) to ([c]R2);
       \draw[->,thick] ([c]Fs) to ([c]Ft);
        \node at ([c]Flabel) [above] {\scriptsize{$f$}};
             \node at ([c]J2) [below] {\begin{color}{gray}\scriptsize{$Z$}\end{color}};
             \draw[->,dotted] ([c]fact1) to ([c]fact2);
             
             \node at ([c]R1) [below] {\scriptsize{$T=\overline T$}};

   \draw[thick]  ([c]S1)--([c]S2);
  \draw[thick] ([c]S2) parabola ([c]U1);
    \draw[thick] ([c]S2) parabola ([c]U2);


            \draw[->,thick] ([c]Fbs) to ([c]Fbt);         
 \node at ([c]Fblabel) [right] {\scriptsize{$\bar{f}$}};
 \end{scope}

  \begin{scope}[every coordinate/.style={shift={(14,-6)}}]     
 \draw ([c]Q1)-- ( [c]Q2)--([c]Q3long)--([c]Q4long)--([c]Q1);
 \node at ([c]Qlongmiddle) [below] {\scriptsize{$\pazocal D^{(4)}\cap\overline{\pazocal M}^{\text{main}}$, Weierstrass tail}};
 \end{scope}
\coordinate (AA0) at (0,.7);
\coordinate (AA1) at (0,.5);
\coordinate (AA2) at (-.8,-.3);
\coordinate (AA3) at (.8,-.3);
\coordinate (AA4) at (3.5,1);
\coordinate (AA5) at (.5,1);
\coordinate (AA6) at (1.6,-.5);

 \coordinate(P1) at  (-0.2,1);
 \coordinate (P2) at ( 3.5,1);
 \coordinate (P3) at (2.5,-1 );
 \coordinate (P4) at (-1.2,-1);
 
 \begin{scope}[every coordinate/.style={shift={(20,-3)}}]
 \draw[thick] ([c]AA0) parabola ([c]AA1); 
 \draw[thick] ([c]AA2) parabola ([c]AA1); 
  \draw[thick] ([c]AA3) parabola ([c]AA1); 
   \draw[thick] ([c]AA3) .. controls ([c]AA4) and ([c]AA5).. ([c]AA6); 
     
\draw[dashed] ([c]P1)-- ( [c]P2)--([c]P3)--([c]P4)--([c]P1); 
     \fill[gray] ([c]AA0) circle (3pt);
      \node at ([c]AA0) [xshift=.35cm,yshift=-.05cm] {\begin{color}{gray}\scriptsize{$f(Z)$}\end{color}};
            \node at ([c]AA0) [above=.05cm] {\scriptsize{$A_4$}};
	    \node at ([c]AA3) [below right,xshift=-.05cm,yshift=.1cm] {\scriptsize{$A_1$}};
    \end{scope}

\end{tikzpicture}

	\caption{$\pazocal{D}^{(4)}\,\dim=16$; $\pazocal D^{(4)}\cap\overline{\pazocal M}^{\text{main}},\dim=12$}
	\label{D4}
\end{figure}

There are only two types of admissible cover which appear enhancing $C$; moreover, once the lift to an admissible cover is chosen, there is no subdivision to be made (since there is only one tropical parameter), and the log structure is the minimal one. 
\begin{itemize}[leftmargin=.5cm]
\item [$\square$]  the node $T\cap Z$ is generic, then sprouting the conjugate point and contracting $Z$ we obtain an $A_5$ dangling singularity (contracting the dangling $\mathbb P^1$ gives the unibranch non Goreinstein singularity with analytic local ring $\mathbb C[\![t^3,t^4,t^5]\!]$), 
\item[$\square$]  the node is Weierstrass and the choice of the other ramification point on $T$ (namely the  choice of lifts to an admissible cover) determines the contraction to an $A_4$ singularity (there is an $\mathbb A_1$ worth of moduli of attaching data in this second case \cite{B})\footnote{The two cases correspond to different irreducible components of $\widetilde{\pazocal A}_2(\PP^2,4)$ lifting $\pazocal{D}^{(4)}$: combinatorially, they are distinguished by the sixth Weierstrass point lying either on the core or on the tail.}.
\end{itemize}

If $f$ factors through $\overline{f}\colon{\overline{C}}\to\mathbb P^2$ for  $\overline{C}$ one of the above, then $\overline{C}$ is a partial normalization for $f(C)$, which we can assume reduced for $f$ with generic weighted dual graph (by Lemma~\ref{lemma:irrcomponents}). Looking at the classification of irreducible quartics of geometric genus zero we deduce that the only possibilities are: 
  \begin{itemize}[leftmargin=.5cm]
  \item either the quartic contains an $E_6$-singularity (with local equations $x^4+y^3$) : intersecting $f(C)$ with the coordinate lines we see that the sections defining the map vanish with order $3$ and $4$ at the point mapping to the singularity; these descend to sections of a complete linear system of bidegree $(4,0)$ on an $A_5$-singularity;
  \item or the quartic is of type $A_1-A_4$ and the rational tail is attached to a Weierstrass point of the core. The core is contracted to the non-rational singularity; moreover in this case the image determines a unique $\overline{C}$ with an $A_4$ singularity through which factorisation holds.
 \end{itemize}
 By Lemma~\ref{lem:unobstructdness}, in any of the cases where factorisation  occur $\operatorname{H}^1(\overline{f}^*\mathcal O_{\PP^2}(1))$ vanishes, hence the maps admitting factorisation are indeed in the image of the non obstructed locus.
   In any case, the dimension is
   \begin{align*}12=&(\dim E_6\;\text{ stratum})+\dim(\overline{\pazocal M}_{2,1})= 8+4\\
=& \dim(A_4A_1\; \text{ stratum})+\dim(\pazocal W_{2,1})=9+3\\
  \end{align*}
   where $\pazocal W_{2,1}$ denote the divisor in $\overline{ \pazocal M}_{2,1}$ where the marking is Weierstrass.    Hence the intersection is a divisor in \emph{main}. See Figure \ref{D4}.

 \item[$\pazocal D^{(3,1)}$ component] The image is the union of a nodal cubic with a line; the dimension is $15$. The smoothable locus has three components (see Figure \ref{D31}).
 

\begin{figure}[h]
   \begin{tikzpicture}[scale=.45]
 \coordinate (A) at (2,4);
 \coordinate (B) at (1.7,3.8);
 \coordinate (C) at   (2,3.6);
 \coordinate (D) at (2.3,3.4) ;
  \coordinate (E) at (2,3.2) ;
  \coordinate (F) at (1.7,3);
  \coordinate (G) at  (2,2.8);
  \coordinate (H1) at (2.3,4.3);
 \coordinate (H2) at (2.4,4.8);
 \coordinate (R1) at (1.9,4.3);
 \coordinate (R2) at (3.5,4.5);
  \coordinate (R1b) at (2.1,2.2);
 \coordinate (R2b) at (3.5,1.8);
 \coordinate (J1) at (2.3,2.4);
 \coordinate (J2) at (2.4,2);
 
 \coordinate (R1) at (1.6,4.1);
 \coordinate (R2) at (3.3,4.3);
  \coordinate (T1) at (1.6,3.2);
 \coordinate (T2) at (3.5,3);
  \coordinate (R1b) at (1.7,2.7);
 \coordinate (R2b) at (3.1,2.3);

   \coordinate(fact1) at (2.4,1);
   \coordinate(fact2) at (3.1,-.4);
 \coordinate (Fs) at (3,3.5);
  \coordinate (Ft) at (4.6,3.5);
   \coordinate (Flabel) at (3.8,3.5);
   \coordinate (Fbs) at (5.3,-.3);
  \coordinate (Fbt) at (6.4,1);
      \coordinate (Fblabel) at (5.9,.1);

 \coordinate (An) at (-.5,.5);
 \coordinate (B1m) at(2,-.2);
  \coordinate (B2m) at(-1.8,-.2);
 \coordinate (Cn) at  (1,.5);
 \coordinate (D1n) at (3,-.2)  ;
  \coordinate (D2n) at (-.5,-.2) ;
  \coordinate (En) at (2,.5);
  \coordinate (F1n) at  (4,-.2);
   \coordinate (F2n) at (1,-.2);
   \coordinate (Gn) at (3,.5);

 \begin{scope}[every coordinate/.style={shift={(0,0)}}]
\draw[draw=gray] ([c]A) .. controls ([c]B) ..
       ([c]C) .. controls ([c]D)  .. ([c]E) .. controls ([c]F).. ([c]G);
      \draw[draw=gray]  plot [smooth, tension=3] coordinates { ([c]A) ([c]H1) ([c]H2)};
       \draw[draw=gray]  plot [smooth, tension=3] coordinates { ([c]G) ([c]J1) ([c]J2)};
       \draw[thick] ([c]R1) to ([c]R2);
            \draw[thick] ([c]T1) to ([c]T2);
            \node at ([c]R1) [above] {\scriptsize{$T_1$}};
              \node at ([c]T1) [below] {\scriptsize{$T_2$}};
             \node at ([c]R2) [above] {\scriptsize{$3$}};
              \node at ([c]T2) [below] {\scriptsize{$1$}};
   \draw[->,thick] ([c]Fs) to ([c]Ft);
        \node at ([c]Flabel) [above] {\scriptsize{$f$}};
             \node at ([c]J2) [below] {\begin{color}{gray}\scriptsize{$Z,g=2$}\end{color}};
 \end{scope}
 
 \coordinate(P1) at  (-0.2,1);
 \coordinate (P2) at ( 3.5,1);
 \coordinate (P3) at (2.5,-1 );
 \coordinate (P4) at (-1.2,-1);

  \coordinate (G1) at (0,.5);
 \coordinate (G2) at ( 4,-.6);
  \coordinate (G3) at (-1.5,-.6);
 \coordinate (G4) at  (1.8,.5);
 \coordinate (G5) at (-.5,-.1) ;
  \coordinate (G6) at (2,-.1) ;
   \coordinate (G7) at (.7,-.1) ;

   \begin{scope}[every coordinate/.style={shift={(6,3)}}]
\draw[thick] ([c]G1) .. controls ([c]G2) and ([c]G3) ..
       ([c]G4);
\draw[thick] ([c]G4) to[out=30,in=60] (8.3,2.6);
       \draw[thick] ([c]G5)--(8.5,2.9);
   
 \draw[dashed] ([c]P1)-- ( [c]P2)--([c]P3)--([c]P4)--([c]P1);  
    
     \fill[gray] ([c]G7) circle (3pt);
          \node at ([c]G7) [below] {\begin{color}{gray}\scriptsize{$f(Z)$}\end{color}};

    \end{scope}

 \coordinate(Q1) at  (.5,5.5);
 \coordinate (Q2) at (10,5.5);
 \coordinate (Q3) at (10,1);
 \coordinate (Q4) at (.5,1);
 \coordinate (Qmiddle) at (5.25,1);
 \coordinate (Q3long) at (10,-2.5 );
 \coordinate (Q4long) at (.5,-2.5);
\coordinate (Qlongmiddle) at (5.5,-2.5);
  \begin{scope}[every coordinate/.style={shift={(0,0)}}]     
 \draw ([c]Q1)-- ( [c]Q2)--([c]Q3)--([c]Q4)--([c]Q1);   
  \node at ([c]Qmiddle) [below] {\scriptsize{$\pazocal D^{(3,1)}$}};

 \end{scope}

\begin{scope}[every coordinate/.style={shift={(14,0)}}]
\draw[thick] ([c]A) .. controls ([c]B) ..
       ([c]C) .. controls ([c]D)  .. ([c]E) .. controls ([c]F).. ([c]G);
      \draw[thick]  plot [smooth, tension=3] coordinates { ([c]A) ([c]H1) ([c]H2)};
       \draw[thick]  plot [smooth, tension=3] coordinates { ([c]G) ([c]J1) ([c]J2)};
       \draw[->,thick] ([c]Fs) to ([c]Ft);
        \node at ([c]Flabel) [above] {\scriptsize{$f$}};
       \end{scope}

  \begin{scope}[every coordinate/.style={shift={(14,0)}}]     
 \draw ([c]Q1)-- ( [c]Q2)--([c]Q3)--([c]Q4)--([c]Q1);   
 \end{scope}

 \coordinate (Am) at (0,0);
 \coordinate (Bm) at(0.2,-.3);
 \coordinate (Cm) at  ( .4,0);
 \coordinate (Dm) at (.6,.3)  ;
  \coordinate (Em) at (.8,0) ;
  \coordinate (Fm) at (1,-.3);
  \coordinate (Gm) at  (1.2,0);
  \coordinate (H1m) at (-.5,.3);
 \coordinate (H2m) at (-.2,.2);
\coordinate (J1m) at (1.4,.2);
 \coordinate (J2m) at (1.7,.3);
 \coordinate (N1m) at (3.2,-.3);
 \coordinate (N2m) at (.2,-.3);
  \coordinate (N3m) at (3,.3);
 \coordinate(P1) at  (-0.2,1);
 \coordinate (P2) at ( 3.5,1);
 \coordinate (P3) at (2.5,-1 );
 \coordinate (P4) at (-1.2,-1);
 
 \begin{scope}[every coordinate/.style={shift={(20,3)}}]
 \draw[thick]  plot [smooth, tension=3] coordinates {([c]H1m) ([c]H2m) ([c]Am)  };
\draw[thick] ([c]Am) .. controls ([c]Bm) ..
       ([c]Cm) .. controls ([c]Dm)  .. ([c]Em) .. controls ([c]Fm).. ([c]Gm);
    \draw[thick]  plot [smooth, tension=3] coordinates { ([c]Gm) ([c]J1m) ([c]J2m)  };   
     \draw[thick]   ([c]J2m).. controls ([c]N1m) and ([c]N2m).. ([c]N3m);
 \draw[dashed] ([c]P1)-- ( [c]P2)--([c]P3)--([c]P4)--([c]P1);  
    \end{scope}   
       
       
\draw (7.5,1)--(4,-1.5);
\draw (7.5,1)--(13,-1.5);
\draw (7.5,1)--(23,-1.5);

 \node at (12.3,3) [above] { $\cap$};


 
  \coordinate(S1) at (3,-.9);
    \coordinate(S2) at (5.3,-.9);
        \coordinate (Y1) at (4.15,-2);
   \coordinate (Y2) at (4.15,.1);
           \coordinate (O) at (4.15,-.9);
    \coordinate (O3) at (3.2,-1.8);
   \coordinate (O4) at (5,-1.8);

 \begin{scope}[every coordinate/.style={shift={(-2,-7)}}]
\draw[draw=gray] ([c]A) .. controls ([c]B) ..
       ([c]C) .. controls ([c]D)  .. ([c]E) .. controls ([c]F).. ([c]G);
      \draw[draw=gray]  plot [smooth, tension=3] coordinates { ([c]A) ([c]H1) ([c]H2)};
       \draw[draw=gray]  plot [smooth, tension=3] coordinates { ([c]G) ([c]J1) ([c]J2)};
       \draw[thick] ([c]R1) to ([c]R2);
            \draw[thick] ([c]T1) to ([c]T2);
                   \draw[dashed] ([c]R1b) to ([c]R2b);
                   \node at ([c]R1) [above] {\scriptsize{$T_1$}};
              \node at ([c]T1) [xshift=-.15cm,yshift=-.05cm]{\scriptsize{$T_2$}};
              \node at ([c]R1b) [xshift=-.15cm,yshift=-.15cm] {\scriptsize{$\overline T_1$}};
             \node at ([c]R2) [above] {\scriptsize{$3$}};
              \node at ([c]T2) [below] {\scriptsize{$1$}};
   \draw[->,thick] ([c]Fs) to ([c]Ft);
        \node at ([c]Flabel) [above] {\scriptsize{$f$}};
             \node at ([c]J2) [below] {\begin{color}{gray}\scriptsize{$Z$}\end{color}};
                          \draw[->,dotted] ([c]fact1) to ([c]fact2);
                          
                   \draw[dashed] ([c]S1) to ([c]S2);
                   \draw[thick] ([c]Y1) to ([c]Y2);
                    \draw[thick] ([c]O) parabola ([c]O3);
                        \draw[thick] ([c]O) parabola ([c]O4);
                        
                        \draw[->,thick] ([c]Fbs) to ([c]Fbt);         
 \node at ([c]Fblabel) [right] {\scriptsize{$\bar{f}$}};

 \end{scope}

  \begin{scope}[every coordinate/.style={shift={(-2,-7)}}]     
 \draw ([c]Q1)-- ( [c]Q2)--([c]Q3long)--([c]Q4long)--([c]Q1);   
  \node at ([c]Qlongmiddle) [below] {\scriptsize{$\pazocal D^{(3,1)}\cap\overline{\pazocal M}^{\text{main}}$, general core}};

 \end{scope}

   \coordinate (E3) at (1.8,-.8);
 \coordinate (E4) at  (1.8,.8);
 \coordinate (E5) at (-.2,.2) ;
  \coordinate (E6) at (2,.2) ;
   \coordinate (E7) at (.7,.2) ;

  \begin{scope}[every coordinate/.style={shift={(4,-4)}}]
\draw[thick] ([c]E7) parabola  ([c]E3) ;
 \draw[thick] ([c]E7) parabola  ([c]E4) ;

       \draw[thick] ([c]E5)--([c]E6);
   
 \draw[dashed] ([c]P1)-- ( [c]P2)--([c]P3)--([c]P4)--([c]P1);  
    
     \fill[gray] ([c]E7) circle (3pt);
          \node at ([c]E7) [below] {\begin{color}{gray}\scriptsize{$f(Z)$}\end{color}};
          \node at ([c]E7) [above] {\scriptsize{$E_7$}};
\end{scope}
  

        \coordinate (y1) at (3.5,-2);
   \coordinate (y2) at (3.5,.1);
           \coordinate (o) at (3.5,-.9);
    \coordinate (o3) at (4.2,-1.8);
   \coordinate (o4) at (4.2,.3);

 \begin{scope}[every coordinate/.style={shift={(8,-7)}}]
\draw[draw=gray] ([c]A) .. controls ([c]B) ..
       ([c]C) .. controls ([c]D)  .. ([c]E) .. controls ([c]F).. ([c]G);
      \draw[draw=gray]  plot [smooth, tension=3] coordinates { ([c]A) ([c]H1) ([c]H2)};
       \draw[draw=gray]  plot [smooth, tension=3] coordinates { ([c]G) ([c]J1) ([c]J2)};
       \draw[thick] ([c]R1) to ([c]R2);
            \draw[thick] ([c]T1) to ([c]T2);
            
                  \node at ([c]R1) [xshift=0cm,yshift=.2cm] {\scriptsize{$T_1=\overline T_1$}};
                  \node at ([c]T1) [below] {\scriptsize{$T_2$}};
             \node at ([c]R2) [above] {\scriptsize{$3$}};
              \node at ([c]T2) [below] {\scriptsize{$1$}};
   \draw[->,thick] ([c]Fs) to ([c]Ft);
        \node at ([c]Flabel) [above] {\scriptsize{$f$}};
             \node at ([c]J2) [below] {\begin{color}{gray}\scriptsize{$Z$}\end{color}};
                          \draw[->,dotted] ([c]fact1) to ([c]fact2);
       
                   \draw[thick] ([c]y1) to ([c]y2);
                    \draw[thick] ([c]o) parabola ([c]o3);
                        \draw[thick] ([c]o) parabola ([c]o4);
                        
                        \draw[->,thick] ([c]Fbs) to ([c]Fbt);         
 \node at ([c]Fblabel) [right] {\scriptsize{$\bar{f}$}};

 \end{scope}

  \begin{scope}[every coordinate/.style={shift={(8,-7)}}]     
 \draw ([c]Q1)-- ( [c]Q2)--([c]Q3long)--([c]Q4long)--([c]Q1);   
  \node at ([c]Qlongmiddle) [below] {\scriptsize{$\pazocal D^{(3,1)}\cap\overline{\pazocal M}^{\text{main}}$, Weierstrass tail}};
 \end{scope}

 \coordinate (e5) at  (1,.8) ;
   \coordinate (e6) at (.3,-.6);
   \coordinate (e3) at (1.8,-.3);
     \coordinate (e2) at (1.8,-.38);
     \coordinate (e4) at (1.2,-.6);
    \coordinate (e7) at (.1,-.3);

  \begin{scope}[every coordinate/.style={shift={(14,-4)}}]
    
 \draw[dashed] ([c]P1)-- ( [c]P2)--([c]P3)--([c]P4)--([c]P1);  

      \end{scope}

\begin{scope}[every coordinate/.style={shift={(14.3,-4)}}]

\draw[thick] ([c]E7) parabola  ([c]e3) ;
 \draw[thick] ([c]E7) parabola  ([c]E4) ;
\draw[thick]   ([c]e3) --  ([c]e2) .. controls ([c]e4) .. ([c]e7) ;

       \draw[thick] ([c]e5)--([c]e6);

    \fill[gray] ([c]E7) circle (3pt);
          \node at ([c]E7) [left] {\begin{color}{gray}\scriptsize{$f(Z)$}\end{color}};
\end{scope}


 
  \coordinate(S1) at (3,-.9);
    \coordinate(S2) at (5.3,-.9);
         \coordinate (O) at (4.15,-.9);
    \coordinate (O3) at (3.2,-1.8);
   \coordinate (bO4) at (5.3,.3);

 \begin{scope}[every coordinate/.style={shift={(18,-7)}}]
\draw[draw=gray] ([c]A) .. controls ([c]B) ..
       ([c]C) .. controls ([c]D)  .. ([c]E) .. controls ([c]F).. ([c]G);
      \draw[draw=gray]  plot [smooth, tension=3] coordinates { ([c]A) ([c]H1) ([c]H2)};
       \draw[draw=gray]  plot [smooth, tension=3] coordinates { ([c]G) ([c]J1) ([c]J2)};

       \draw[thick] ([c]R1) to ([c]R2);
          \draw[thick] ([c]R1b) to ([c]R2b);
                   
             \node at ([c]R2) [above] {\scriptsize{$3$}};
              \node at ([c]T2) [below] {\scriptsize{$1$}};
              
               \node at ([c]R1) [above] {\scriptsize{$T_1$}};
              \node at ([c]R1b) [above=-.05cm] {\scriptsize{$T_2=\overline{T}_1$}};
              
   \draw[->,thick] ([c]Fs) to ([c]Ft);
        \node at ([c]Flabel) [above] {\scriptsize{$f$}};
             \node at ([c]J2) [below] {\begin{color}{gray}\scriptsize{$Z$}\end{color}};
                          \draw[->,dotted] ([c]fact1) to ([c]fact2);
                          
                   \draw[thick] ([c]S1) to ([c]S2);
                   
                    \draw[thick] ([c]O) parabola ([c]O3);
                        \draw[thick] ([c]O) parabola ([c]bO4);
                        
                        \draw[->,thick] ([c]Fbs) to ([c]Fbt);         
 \node at ([c]Fblabel) [right] {\scriptsize{$\bar{f}$}};

 \end{scope}

  \begin{scope}[every coordinate/.style={shift={(18,-7)}}]     
 \draw ([c]Q1)-- ( [c]Q2)--([c]Q3long)--([c]Q4long)--([c]Q1);   
  \node at ([c]Qlongmiddle) [below] {\scriptsize{$\pazocal D^{(3,1)}\cap\overline{\pazocal M}^{\text{main}}$, conjugate tails}};
 \end{scope}

   \coordinate (F3) at (-.2,-.8);

 \coordinate (F5) at (-.2,-.1) ;
  \coordinate (F6) at (2,-.1) ;
   \coordinate (F7) at (.7,-.1) ;
  
  \coordinate (F4) at  (4,.8);
   \coordinate (F2) at  (.5,.8);
    \coordinate (F1) at  (2.5,.3);

  \begin{scope}[every coordinate/.style={shift={(24,-4)}}]
  
\draw[thick] ([c]F7) parabola  ([c]F3) ;

 \draw[thick] ([c]F7) .. controls  ([c]F4) and ([c]F2) .. ([c]F1) ;

       \draw[thick] ([c]F5)--([c]F6);
   
 \draw[dashed] ([c]P1)-- ( [c]P2)--([c]P3)--([c]P4)--([c]P1);

     \fill[gray] ([c]F7) circle (3pt);
          \node at ([c]F7) [below] {\begin{color}{gray}\scriptsize{$f(Z)$}\end{color}};
\end{scope}

  \begin{scope}[every coordinate/.style={shift={(18,-7)}}]
\draw[draw=gray] ([c]A) .. controls ([c]B) ..
       ([c]C) .. controls ([c]D)  .. ([c]E) .. controls ([c]F).. ([c]G);
      \draw[draw=gray]  plot [smooth, tension=3] coordinates { ([c]A) ([c]H1) ([c]H2)};
       \draw[draw=gray]  plot [smooth, tension=3] coordinates { ([c]G) ([c]J1) ([c]J2)};

       \draw[thick] ([c]R1) to ([c]R2);
          \draw[thick] ([c]R1b) to ([c]R2b);
                   
             \node at ([c]R2) [above] {\scriptsize{$3$}};
              \node at ([c]T2) [below] {\scriptsize{$1$}};
              
               \node at ([c]R1) [above] {\scriptsize{$T_1$}};
              \node at ([c]R1b) [above=-.05cm] {\scriptsize{$T_2=\overline{T}_1$}};
              
   \draw[->,thick] ([c]Fs) to ([c]Ft);
        \node at ([c]Flabel) [above] {\scriptsize{$f$}};
             \node at ([c]J2) [below] {\begin{color}{gray}\scriptsize{$Z$}\end{color}};
                          \draw[->,dotted] ([c]fact1) to ([c]fact2);
                          
                   \draw[thick] ([c]S1) to ([c]S2);
                   
                    \draw[thick] ([c]O) parabola ([c]O3);
                        \draw[thick] ([c]O) parabola ([c]bO4);
                        
                        \draw[->,thick] ([c]Fbs) to ([c]Fbt);         
 \node at ([c]Fblabel) [right] {\scriptsize{$\bar{f}$}};

 \end{scope}

\end{tikzpicture}

\caption{$\pazocal D^{(3,1)},\dim=15$; $\pazocal D^{(3,1)}\cap\overline{\pazocal M}^{\text{main}},\dim=12$}
\label{D31}
\end{figure}

Again, let us start by examining which singularities can appear. 
 \begin{itemize}[leftmargin=.5cm]
 \item[$\square$]  the tails $T_1,T_2$, on which the map has degree $3$ and $1$ respectively, are attached to generic points of the core, and there is a unique lift to an admissible cover $\psi_1\colon C_1'\to R_1$ of this combinatorial type obtained by sprouting  the conjugate points of $T_i\cap Z$.  Enhancing the
 dual graph to a tropical admissible cover and letting the edge lengths $l_1,l_2$ vary  in  the cone $\sigma_1\cong\mathbb R^2_{\geq 0}$, we find five possibilities\footnote{The \cite{BC} subdivision of $\sigma_1$ has seven cones, but for three of them what changes is only the aligning function $\lambda$, not the singularity obtained via contraction and push-out.} for the analytic syingularities of $\overline{C}$. For the benefit of the reader, we add a picture of all the singularities in  Figure~\ref{fig:singularities}.
 
 \begin{figure}[h]
\begin{tikzpicture} [scale=.4]
\tikzstyle{every node}=[font=\normalsize]
		\tikzset{arrow/.style={latex-latex}}
		
	\tikzset{cross/.style={cross out, draw, thick,
         minimum size=2*(#1-\pgflinewidth), 
         inner sep=1.2pt, outer sep=1.2pt}}

	     \coordinate (y1) at (6,0,0);
	   \coordinate (y2) at (6,-3,0);
	    \coordinate (x1) at (4.5,-1.5,0);
	   \coordinate (x2) at (7.5,-1.5,0);
	   \coordinate (oo) at (6,-1.5,0);
	    \coordinate (m1) at (7,-.5,0);
	   \coordinate (m2) at (5,-.5,0);
	    \coordinate (b1) at (5.7,-.4,0);
	    \coordinate (b2) at (7,.2,0);
	    
	       \coordinate (y1e) at (6.2,0,0);
	   \coordinate (y2e) at (6.2,-3,0);
	   \coordinate (b1e) at (5.7,-2.7,0);
	    \coordinate (b2e) at (7,-2.2,0);
	       \coordinate (m3) at (7,-2.5,0);
		  
  \begin{scope}[every coordinate/.style={shift={(-2,0,0)}}]

           \draw[gray] ([c]y1)-- ([c]y2);
            \draw[gray] ([c]x1)-- ([c]x2);
              \draw[thick] ([c]b1)-- ([c]b2);
              \draw[thick] ([c]oo) parabola ([c]m1);
               \draw[thick] ([c]oo) parabola ([c]m2);
            
            \node at ([c]m1) [right] {{\scriptsize{$T_1$}}};
             \node at ([c]x2) [right] {{\scriptsize{$\overline{T_1}$}}};
              \node at ([c]b2) [right] {{\scriptsize{$T_2$}}};

            \end{scope}

		  \begin{scope}[every coordinate/.style={shift={(4,0,0)}}]

           \draw[ thick] ([c]y1)-- ([c]y2);
            \draw[gray] ([c]x1)-- ([c]x2);
            
              \draw [thick]([c]oo) parabola ([c]m1);
               \draw[thick] ([c]oo) parabola ([c]m2); 
            
            \node at ([c]m1) [right] {{\scriptsize{$T_1$}}};
             \node at ([c]x2) [right] {{\scriptsize{$\overline{T_1}$}}};
              \node at ([c]y2) [right] {{\scriptsize{$T_2$}}};

            \end{scope}

		\begin{scope}[every coordinate/.style={shift={(10,0,0)}}]

           \draw[gray] ([c]y1)-- ([c]y1e)--([c]y2e)--([c]y2)--([c]y1);
            \draw[thick] ([c]b1)-- ([c]b2);
            \draw[thick] ([c]b1e)-- ([c]b2e);

        
             \node at ([c]b2) [right] {{\scriptsize{$T_1$}}};
              \node at ([c]b2e) [right] {{\scriptsize{$T_2$}}};

            \end{scope}

	  \begin{scope}[every coordinate/.style={shift={(16,0,0)}}]

           \draw[ thick] ([c]y1)-- ([c]y2);
            \draw[gray] ([c]x1)-- ([c]x2);
            
              \draw [thick]([c]oo) parabola ([c]m1);
               \draw[thick] ([c]oo) parabola ([c]m2); 
            
            \node at ([c]m1) [right] {{\scriptsize{$T_2$}}};
             \node at ([c]x2) [right] {{\scriptsize{$\overline{T_2}$}}};
              \node at ([c]y2) [right] {{\scriptsize{$T_1$}}};

            \end{scope}

		  \begin{scope}[every coordinate/.style={shift={(22,0,0)}}]
		               
           \draw[gray] ([c]y1)-- ([c]y2);
            \draw[gray] ([c]x1)-- ([c]x2);
              \draw[thick] ([c]b1)-- ([c]b2);
              \draw[thick] ([c]oo) parabola ([c]m1);
               \draw[thick] ([c]oo) parabola ([c]m2);
            
            \node at ([c]m1) [right] {{\scriptsize{$T_2$}}};
             \node at ([c]x2) [right] {{\scriptsize{$\overline{T_2}$}}};
              \node at ([c]b2) [right] {{\scriptsize{$T_1$}}};

          \end{scope}

	\end{tikzpicture}
	\caption{Singularities type that can appear for the lift $C_1\xrightarrow{\psi_1} T_1$}
 \label{fig:singularities}
\end{figure}
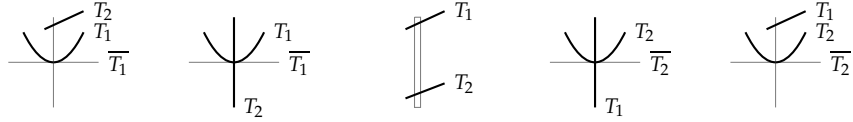
Notice that factorisation through the  last three singularities on the right  does not occur. Sections of a line bundle on $T_i$ descend to a contracted ribbon with two tails if  and only if they ramify (i.e. vanish with order at least two) at the nodes $T_i\cap Z$, and this cannot happen since we only have linear sections of $T_2$. In the two right-most cases the line bundle on the genus one subcurve given by the special branches would have degree one, so there is no morphism of the prescribed degree from these singularities to $\mathbb P^2.$ 
 
 \item[$\square$] $T_1$ is attached to a Weierstrass point, and a lift to an admissible cover $C'_2\xrightarrow{\psi_2} R_2$ is determined by the position of the other ramification point on $T_1$.  Enhancing the
 dual graph to a tropical admissible cover and letting the edge lengths $l_1,l_2$ vary  in $\sigma_2$, we find five possibilities for the singularities of $\overline{C}$, three of them  have already appeared for the previous choice of lift to an admissible cover (the last three on the right of Figure~\ref{fig:singularities}) and we  have discussed that factorisation through them does not occur. There are two new possibilities for $\overline{C}$: either it has a $D_5$ singularity and $T_1$ is the special (cuspidal) branch; or it has a dangling $D_5$ singularity and $T_2$ is nodally attached to the weight $0$ line of the singularity;
  \item[$\square$] Lifting $C$ to an admissible cover with $T_2$ Weierstrass is not interesting, as the new analytic singularity we  find has a cuspidal special branch of degree one, and thus  factorisation cannot occur.  
   \item[$\square$] The $T_i$ are attached to conjugate points, and a  lift to an admissible cover $C'_3\xrightarrow{\psi_3} R_3$ is defined  choosing an identification of $T_1$ and $T_2$ with their image under the two-to-one cover. In this case, once we have chosen the admissible cover, there is no alignment,  the log structure is the minimal one,  and the contraction gives an $A_5$ singularity.
  \end{itemize}
 Since we can assume that $f$ has generic weighted dual graph, its restriction to the non contracted components is birational onto its image in $\mathbb P^2$; as in the previous case, if the map factors then $\overline{C}$ is a partial normalization of $f(C)$. Looking at the classification of singularities of quartics which decompose as a rational cubic and a line we see that the only possibilities are:
 
 \begin{itemize}[leftmargin=.5cm]
  \item either the core is general, and the image has an $E_7$-singularity (local equation $y(y^2-x^3)$, i.e. the union of a cuspidal cubic with its reduced normal cone; it is the image of a complete linear system of degree $(3,1,0)$ on a type $I\!I_3$ singularity, where the special branches are the first and third ones).
  \item Or the attaching point of the degree $3$ tail is Weierstrass, and the image has a $D_5$-singularity (a cuspidal cubic with a general line through the cusp). 
  \item Or the two attaching points are conjugate, and the image is of type $A_1-A_5$ (a nodal cubic with a flex line). 
   \end{itemize}
 In any case, the image determines uniquely the singularity $\overline{C}$ through which we have factorisation.
 
    Notice that  factorisation through a sprouted $D_5$ singularity or a sprouted $I\!I_3$ singularity is also possible. In both  cases there is only a non constant section on the degree $3$ special branch which vanishes of order $3$ at the singular point, hence the map covers $3:1$ a line; these loci correspond to the degeneration of the cubic to a triple line and are contained in the closure of the components of the intersection listed above.
 
Again, by  Lemma~\ref{lem:unobstructdness}, in any of the cases where factorisation  occur $\operatorname{H}^1(\overline{f}^*\mathcal O_{\PP^2}(1))$ vanishes, hence the maps admitting factorisation are indeed in the image of the non obstructed locus.

  In all three cases the dimension is 
  \begin{align*}12= & (\dim E_7\;\text{ stratum})+5=\dim(\overline{\pazocal M}_{2,2})=7+5\\
  =&(\dim D_5A_1\;\text{ stratum})+\dim(\overline{\pazocal W}_{2,2})=8+4\\
  =& (\dim A_5A_1\;\text{ stratum})+\dim(\overline{\pazocal K}_{2,2})=8+4
  \end{align*}
  where $\overline{\pazocal W}_{2,2},\; \overline{\pazocal K}_{2,2}$ are respectively the divisors in $\overline{\pazocal M}_{2,2}$ where one of the markings is Weierstrass, and where the two points are conjugate.  This a divisor in $\emph{main}$.

 \item[$\pazocal D^{(2^2)}$ component] The image is the union of two conics; the dimension is $15$. The intersection with \emph{main} has two components ( See Figure \ref{D22}.)

 \begin{figure}[h]
   \begin{tikzpicture}[scale=.4]
 \coordinate (A) at (2,4);
 \coordinate (B) at (1.7,3.8);
 \coordinate (C) at   (2,3.6);
 \coordinate (D) at (2.3,3.4) ;
  \coordinate (E) at (2,3.2) ;
  \coordinate (F) at (1.7,3);
  \coordinate (G) at  (2,2.8);
  \coordinate (H1) at (2.3,4.3);
 \coordinate (H2) at (2.4,4.8);
 \coordinate (R1) at (1.9,4.3);
 \coordinate (R2) at (3.5,4.5);
  \coordinate (R1b) at (2.1,2.2);
 \coordinate (R2b) at (3.5,1.8);
 \coordinate (J1) at (2.3,2.4);
 \coordinate (J2) at (2.4,2);
 
 \coordinate (R1) at (1.6,4.1);
 \coordinate (R2) at (3.3,4.3);
  \coordinate (T1) at (1.6,3.2);
 \coordinate (T2) at (3.5,3);
  \coordinate (R1b) at (1.7,2.7);
 \coordinate (R2b) at (3.1,2.3);

   \coordinate(fact1) at (2.4,1);
   \coordinate(fact2) at (3.1,-.4);
 \coordinate (Fs) at (3,3.5);
  \coordinate (Ft) at (4.6,3.5);
   \coordinate (Flabel) at (3.8,3.5);
   \coordinate (Fbs) at (5.3,-.3);
  \coordinate (Fbt) at (6.4,1);
      \coordinate (Fblabel) at (5.9,.1);

 \coordinate (An) at (-.5,.5);
 \coordinate (B1m) at(2,-.2);
  \coordinate (B2m) at(-1.8,-.2);
 \coordinate (Cn) at  (1,.5);
 \coordinate (D1n) at (3,-.2)  ;
  \coordinate (D2n) at (-.5,-.2) ;
  \coordinate (En) at (2,.5);
  \coordinate (F1n) at  (4,-.2);
   \coordinate (F2n) at (1,-.2);
   \coordinate (Gn) at (3,.5);

 \begin{scope}[every coordinate/.style={shift={(0,0)}}]
\draw[draw=gray] ([c]A) .. controls ([c]B) ..
       ([c]C) .. controls ([c]D)  .. ([c]E) .. controls ([c]F).. ([c]G);
      \draw[draw=gray]  plot [smooth, tension=3] coordinates { ([c]A) ([c]H1) ([c]H2)};
       \draw[draw=gray]  plot [smooth, tension=3] coordinates { ([c]G) ([c]J1) ([c]J2)};
       \draw[thick] ([c]R1) to ([c]R2);
            \draw[thick] ([c]T1) to ([c]T2);
            \node at ([c]R1) [above] {\scriptsize{$T_1$}};
              \node at ([c]T1) [below] {\scriptsize{$T_2$}};
             \node at ([c]R2) [above] {\scriptsize{$2$}};
              \node at ([c]T2) [below] {\scriptsize{$2$}};
   \draw[->,thick] ([c]Fs) to ([c]Ft);
        \node at ([c]Flabel) [above] {\scriptsize{$f$}};
             \node at ([c]J2) [below] {\begin{color}{gray}\scriptsize{$Z,g=2$}\end{color}};
 \end{scope}
 
 \coordinate(P1) at  (-0.2,1);
 \coordinate (P2) at ( 3.5,1);
 \coordinate (P3) at (2.5,-1 );
 \coordinate (P4) at (-1.2,-1);

  
  \coordinate (G1) at (1,0);
 \coordinate (G2) at (1.4,.45);
  
   \begin{scope}[every coordinate/.style={shift={(6,3)}}]
\draw[thick] ([c]G1)  ellipse (1.3cm and 0.5cm);
      \draw[thick] ([c]G1)  ellipse (.5cm and .9cm);

 \draw[dashed] ([c]P1)-- ( [c]P2)--([c]P3)--([c]P4)--([c]P1);  
    
     \fill[gray] ([c]G2) circle (3pt);
          \node at ([c]G2) [right] {\begin{color}{gray}\scriptsize{$f(Z)$}\end{color}};

    \end{scope}

  \coordinate(Q1) at  (.5,5.5);
 \coordinate (Q2) at (10,5.5);
 \coordinate (Q3) at (10,1);
 \coordinate (Q4) at (.5,1);
 \coordinate (Qmiddle) at (5.25,1);
 \coordinate (Q3long) at (10,-2.5 );
 \coordinate (Q4long) at (.5,-2.5);
\coordinate (Qlongmiddle) at (5.5,-2.5);
  \begin{scope}[every coordinate/.style={shift={(0,0)}}]     
 \draw ([c]Q1)-- ( [c]Q2)--([c]Q3)--([c]Q4)--([c]Q1);   
  \node at ([c]Qmiddle) [below] {\scriptsize{$\pazocal D^{(2^2)}$}};

 \end{scope}

\begin{scope}[every coordinate/.style={shift={(14,0)}}]
\draw[thick] ([c]A) .. controls ([c]B) ..
       ([c]C) .. controls ([c]D)  .. ([c]E) .. controls ([c]F).. ([c]G);
      \draw[thick]  plot [smooth, tension=3] coordinates { ([c]A) ([c]H1) ([c]H2)};
       \draw[thick]  plot [smooth, tension=3] coordinates { ([c]G) ([c]J1) ([c]J2)};
       \draw[->,thick] ([c]Fs) to ([c]Ft);
        \node at ([c]Flabel) [above] {\scriptsize{$f$}};
       \end{scope}
  \coordinate(Q1) at  (.5,5.5);
 \coordinate (Q2) at (10,5.5);
 \coordinate (Q3) at (10,1);
 \coordinate (Q4) at (.5,1);
 \coordinate (Qmiddle) at (5.25,1);
 \coordinate (Q3long) at (10,-2.5 );
 \coordinate (Q4long) at (.5,-2.5);
\coordinate (Qlongmiddle) at (5.5,-2.5);

  \begin{scope}[every coordinate/.style={shift={(14,0)}}]     
 \draw ([c]Q1)-- ( [c]Q2)--([c]Q3)--([c]Q4)--([c]Q1);   
 \end{scope}

 \coordinate (Am) at (0,0);
 \coordinate (Bm) at(0.2,-.3);
 \coordinate (Cm) at  ( .4,0);
 \coordinate (Dm) at (.6,.3)  ;
  \coordinate (Em) at (.8,0) ;
  \coordinate (Fm) at (1,-.3);
  \coordinate (Gm) at  (1.2,0);
  \coordinate (H1m) at (-.5,.3);
 \coordinate (H2m) at (-.2,.2);
\coordinate (J1m) at (1.4,.2);
 \coordinate (J2m) at (1.7,.3);
 \coordinate (N1m) at (3.2,-.3);
 \coordinate (N2m) at (.2,-.3);
  \coordinate (N3m) at (3,.3);
 \coordinate(P1) at  (-0.2,1);
 \coordinate (P2) at ( 3.5,1);
 \coordinate (P3) at (2.5,-1 );
 \coordinate (P4) at (-1.2,-1);
 
 \begin{scope}[every coordinate/.style={shift={(20,3)}}]
 \draw[thick]  plot [smooth, tension=3] coordinates {([c]H1m) ([c]H2m) ([c]Am)  };
\draw[thick] ([c]Am) .. controls ([c]Bm) ..
       ([c]Cm) .. controls ([c]Dm)  .. ([c]Em) .. controls ([c]Fm).. ([c]Gm);
    \draw[thick]  plot [smooth, tension=3] coordinates { ([c]Gm) ([c]J1m) ([c]J2m)  };   
     \draw[thick]   ([c]J2m).. controls ([c]N1m) and ([c]N2m).. ([c]N3m);
 \draw[dashed] ([c]P1)-- ( [c]P2)--([c]P3)--([c]P4)--([c]P1);  
    \end{scope}   
       
       
\draw (7.5,1)--(8.5,-.5);
\draw (7.5,1)--(19,-.5);

 \node at (12.3,3) [above] { $\cap$};


 
  \coordinate(S1) at (3,-.9);
    \coordinate(S2) at (5.3,-.9);
         \coordinate (O) at (4.15,-.9);
    \coordinate (O3) at (3.2,-1.8);
   \coordinate (bO4) at (5.3,.3);

 \begin{scope}[every coordinate/.style={shift={(3,-6)}}]
\draw[draw=gray] ([c]A) .. controls ([c]B) ..
       ([c]C) .. controls ([c]D)  .. ([c]E) .. controls ([c]F).. ([c]G);
      \draw[draw=gray]  plot [smooth, tension=3] coordinates { ([c]A) ([c]H1) ([c]H2)};
       \draw[draw=gray]  plot [smooth, tension=3] coordinates { ([c]G) ([c]J1) ([c]J2)};

       \draw[thick] ([c]R1) to ([c]R2);
          \draw[thick] ([c]R1b) to ([c]R2b);
                   
             \node at ([c]R2) [above] {\scriptsize{$2$}};
              \node at ([c]T2) [below] {\scriptsize{$2$}};
              
               \node at ([c]R1) [above] {\scriptsize{$T_1$}};
              \node at ([c]R1b) [yshift=.2cm] {\scriptsize{$T_2=\overline{T}_1$}};
              
   \draw[->,thick] ([c]Fs) to ([c]Ft);
        \node at ([c]Flabel) [above] {\scriptsize{$f$}};
             \node at ([c]J2) [below] {\begin{color}{gray}\scriptsize{$Z$}\end{color}};
                          \draw[->,dotted] ([c]fact1) to ([c]fact2);
                          
                   \draw[thick] ([c]S1) to ([c]S2);
                   
                    \draw[thick] ([c]O) parabola ([c]O3);
                        \draw[thick] ([c]O) parabola ([c]bO4);
                        
                        \draw[->,thick] ([c]Fbs) to ([c]Fbt);         
 \node at ([c]Fblabel) [right] {\scriptsize{$\bar{f}$}};

 \end{scope}

  \begin{scope}[every coordinate/.style={shift={(3,-6)}}]     
 \draw ([c]Q1)-- ( [c]Q2)--([c]Q3long)--([c]Q4long)--([c]Q1);   
  \node at ([c]Qlongmiddle) [below] {\scriptsize{$\pazocal D^{(2^2)}\cap\overline{\pazocal M}^{\text{main}}$, conjugate tails}};
 \end{scope}

   
    \coordinate (F3) at  (-.2,0);
       \coordinate (F4) at  (0.05,0);
   \coordinate (F2) at  (.75,.15);
    \coordinate (F1) at (1.3,0);

  \begin{scope}[every coordinate/.style={shift={(9,-3)}}]
  \draw[thick] ([c]F1)  ellipse (1.3cm and 0.5cm);
  \draw[thick] ([c]F2)  ellipse (.8cm and 0.5cm);
 \draw[dashed] ([c]P1)-- ( [c]P2)--([c]P3)--([c]P4)--([c]P1);

               \fill[gray] ([c]F4) circle (3pt);
          \node at ([c]F3) [below] {\begin{color}{gray}\scriptsize{$f(Z)$}\end{color}};
          \node at ([c]F3) [xshift=-.2cm] {\scriptsize{$A_5$}};
          \node at ([c]F3) [xshift=.7cm,yshift=.35cm] {\scriptsize{$A_1$}};
\end{scope}


 
  \coordinate(S1) at (3,-.9);
    \coordinate(S2) at (5.3,-.9);
        \coordinate (Y1) at (4.15,-2);
   \coordinate (Y2) at (4.15,.1);
           \coordinate (O) at (4.15,-.9);
    \coordinate (O3) at (3.2,-1.8);
   \coordinate (O4) at (5,-1.8);

 \begin{scope}[every coordinate/.style={shift={(14,-6)}}]
\draw[draw=gray] ([c]A) .. controls ([c]B) ..
       ([c]C) .. controls ([c]D)  .. ([c]E) .. controls ([c]F).. ([c]G);
      \draw[draw=gray]  plot [smooth, tension=3] coordinates { ([c]A) ([c]H1) ([c]H2)};
       \draw[draw=gray]  plot [smooth, tension=3] coordinates { ([c]G) ([c]J1) ([c]J2)};
       \draw[thick] ([c]R1) to ([c]R2);
            \draw[thick] ([c]T1) to ([c]T2);
                   \draw[dashed] ([c]R1b) to ([c]R2b);
             \node at ([c]R2) [above] {\scriptsize{$2$}};
              \node at ([c]T2) [below] {\scriptsize{$2$}};
              \node at ([c]R1) [above] {\scriptsize{$T_1$}};
              \node at ([c]T1) [xshift=-.2cm,yshift=.1cm] {\scriptsize{$T_2$}};
              \node at ([c]R1b) [xshift=-.2cm,yshift=-.1cm] {\scriptsize{$\overline T_1$}};
   \draw[->,thick] ([c]Fs) to ([c]Ft);
        \node at ([c]Flabel) [above] {\scriptsize{$f$}};
             \node at ([c]J2) [below] {\begin{color}{gray}\scriptsize{$Z$}\end{color}};
                          \draw[->,dotted] ([c]fact1) to ([c]fact2);
                          
                   \draw[dashed] ([c]S1) to ([c]S2);
                   \draw[thick] ([c]Y1) to ([c]Y2);
                    \draw[thick] ([c]O) parabola ([c]O3);
                        \draw[thick] ([c]O) parabola ([c]O4);
                        \node at ([c]O4) [right] {\scriptsize{$T_1$}};
                        \draw[->,thick] ([c]Fbs) to ([c]Fbt);         
 \node at ([c]Fblabel) [right] {\scriptsize{$\bar{f}$}};

 \end{scope}

 \begin{scope}[every coordinate/.style={shift={(14,-6)}}]     
 \draw ([c]Q1)-- ( [c]Q2)--([c]Q3long)--([c]Q4long)--([c]Q1);   
  \node at ([c]Qlongmiddle) [below] {\scriptsize{$\pazocal D^{(2^2)}\cap\overline{\pazocal M}^{\text{main}}$, general core}};
 \end{scope}

 \coordinate (E4) at  (.7,.2);
 \coordinate (E5) at (-.3,-.2) ;
  \coordinate (E6) at (2,-.2) ;
   \coordinate (E8) at (1.7,-.2) ;
   \coordinate (E7) at (.7,-.2) ;

  \begin{scope}[every coordinate/.style={shift={(20,-3)}}]

 \draw[thick] ([c]E4)  ellipse (1cm and 0.4cm);
       \draw[thick] ([c]E5)--([c]E6);
   
 \draw[dashed] ([c]P1)-- ( [c]P2)--([c]P3)--([c]P4)--([c]P1);  
    
     \fill[gray] ([c]E7) circle (3pt);
      \fill[blue] ([c]E8) circle (3pt);
          \node at ([c]E7) [below] {\begin{color}{gray}\scriptsize{$f(Z)$}\end{color}};
          \node at ([c]E7) [yshift=.15cm] {\scriptsize{$A_3$}};
          \node at ([c]E6) [xshift=.18cm] {\scriptsize{$f(T_1)$}};
\end{scope}

  \end{tikzpicture}
\caption{$\pazocal D^{(2^2)},\dim=15$; $\pazocal D^{(2^2)}\cap\overline{\pazocal M}^{\text{main}},\dim=12$}
\label{D22}	
\end{figure}

 \begin{itemize}[leftmargin=.5cm]
  \item  the two attaching points are conjugate and the image is of type $A_1-A_5$ (the conics intersect at two points, of multiplicity one and three respectively); 
  notice that the map determines the $A_5$ singularity to factor through;
  \item the core is general and the image consists of a line (covered two-to-one) with a conic tangent to it at one of the branch points of the double cover (the image of a complete linear system of degree $(2,2,0)$ on a type $I\!I_3$ singularity); the lift to the admissible covers is unique since $\overline T_1$ has only one special point, and the choice of ramification determines the $I\!I_3$ singularity to factor through.
 \end{itemize}
 In any case, the dimension is $8+4=7+5=12$.

 Notice that in this case it was not necessary to consider the lift to an admissible cover with one of the two tails being attached to a Weierstrass point. Indeed, in that case we would  have that the map factors through a $D_5$ singularity if it restricts to a two-to-one cover ramified at the node on the Weierstrass tail (say it is $T_1$) and $f(T_2)$ is a conic tangent to $f(T_1)$ at $f(Z)$, i.e. the  factorisation conditions are like in the $I\!I_3$ singularity case. This locus is in the closure of the second component of the intersection, corresponding to degeneration of the genus two source curve into the divisor $\pazocal W_{2,2}.$
 

 \item[$\pazocal D^{(2,1^2)}$ component] The image consists of a conic and two lines, all concurrent in a point (where the core is contracted). The dimension is $14$. The smoothable locus consists of three components (See Figure \ref{D211}.
)

\begin{figure}[hbt]
   \begin{tikzpicture}[scale=.5]
 \coordinate (A) at (2,4);
 \coordinate (B) at (1.7,3.8);
 \coordinate (C) at   (2,3.6);
 \coordinate (D) at (2.3,3.4) ;
  \coordinate (E) at (2,3.2) ;
  \coordinate (F) at (1.7,3);
  \coordinate (G) at  (2,2.8);
  \coordinate (H1) at (2.3,4.3);
 \coordinate (H2) at (2.4,4.8);
 \coordinate (R1) at (1.9,4.3);
 \coordinate (R2) at (3.5,4.5);
  \coordinate (R1b) at (2.1,2.2);
 \coordinate (R2b) at (3.5,1.8);
 \coordinate (J1) at (2.3,2.4);
 \coordinate (J2) at (2.4,2);
 
 \coordinate (R1) at (1.6,4.1);
 \coordinate (R2) at (3.3,4.3);
  \coordinate (T1) at (1.6,3.2);
 \coordinate (T2) at (3.5,3);
  \coordinate (R1b) at (1.7,2.7);
 \coordinate (R2b) at (3.1,2.3);

   \coordinate(fact1) at (2.4,1);
   \coordinate(fact2) at (3.1,-.4);
 \coordinate (Fs) at (3,3.5);
  \coordinate (Ft) at (4.6,3.5);
   \coordinate (Flabel) at (3.8,3.5);
   \coordinate (Fbs) at (5.3,-.3);
  \coordinate (Fbt) at (6.4,1);
      \coordinate (Fblabel) at (5.9,.1);

 \coordinate (An) at (-.5,.5);
 \coordinate (B1m) at(2,-.2);
  \coordinate (B2m) at(-1.8,-.2);
 \coordinate (Cn) at  (1,.5);
 \coordinate (D1n) at (3,-.2)  ;
  \coordinate (D2n) at (-.5,-.2) ;
  \coordinate (En) at (2,.5);
  \coordinate (F1n) at  (4,-.2);
   \coordinate (F2n) at (1,-.2);
   \coordinate (Gn) at (3,.5);

 \begin{scope}[every coordinate/.style={shift={(0,0)}}]
\draw[draw=gray] ([c]A) .. controls ([c]B) ..
       ([c]C) .. controls ([c]D)  .. ([c]E) .. controls ([c]F).. ([c]G);
      \draw[draw=gray]  plot [smooth, tension=3] coordinates { ([c]A) ([c]H1) ([c]H2)};
       \draw[draw=gray]  plot [smooth, tension=3] coordinates { ([c]G) ([c]J1) ([c]J2)};
       \draw[thick] ([c]R1) to ([c]R2);
            \draw[thick] ([c]T1) to ([c]T2);
            \draw[thick] ([c]R1b) to ([c]R2b);
             \node at ([c]R1) [above] {\scriptsize{$T_1$}};
             \node at ([c]T1) [xshift=-.15cm,yshift=.1cm] {\scriptsize{$T_2$}};
              \node at ([c]R1b) [xshift=-.2cm,yshift=-.1cm] {\scriptsize{$T_3$}};
             \node at ([c]R2) [above] {\scriptsize{$2$}};
             \node at ([c]R2b) [below] {\scriptsize{$1$}};
              \node at ([c]T2) [below] {\scriptsize{$1$}};
   \draw[->,thick] ([c]Fs) to ([c]Ft);
        \node at ([c]Flabel) [above] {\scriptsize{$f$}};
             \node at ([c]J2) [yshift=-.35cm] {\begin{color}{gray}\scriptsize{$Z,g=2$}\end{color}};
 \end{scope}
 
 \coordinate(P1) at  (-0.2,1);
 \coordinate (P2) at ( 3.5,1);
 \coordinate (P3) at (2.5,-1 );
 \coordinate (P4) at (-1.2,-1);

  
  \coordinate (G1) at (1,0);
 \coordinate (G2) at (1.4,.45);
 \coordinate (G3) at (.5,-.8);
 \coordinate (G4) at (1.1,-.8);
 \coordinate (G5) at (1.65,.8);
 \coordinate (G6) at (1.4,.8);
  
   \begin{scope}[every coordinate/.style={shift={(6,3)}}]
\draw[thick] ([c]G1)  ellipse (1.3cm and 0.5cm);
      \draw[thick] ([c]G5)--([c]G3) ;
      \draw[thick] ([c]G6)--([c]G4) ;

 \draw[dashed] ([c]P1)-- ( [c]P2)--([c]P3)--([c]P4)--([c]P1);  
    
     \fill[gray] ([c]G2) circle (3pt);
          \node at ([c]G2) [right] {\begin{color}{gray}\scriptsize{$f(Z)$}\end{color}};

    \end{scope}

 \coordinate(Q1) at  (.5,5.5);
 \coordinate (Q2) at (10,5.5);
 \coordinate (Q3) at (10,1);
 \coordinate (Q4) at (.5,1);
 \coordinate (Qmiddle) at (5.25,1);
 \coordinate (Q3long) at (10,-2.5 );
 \coordinate (Q4long) at (.5,-2.5);
\coordinate (Qlongmiddle) at (5.5,-2.5);
  \begin{scope}[every coordinate/.style={shift={(0,0)}}]     
 \draw ([c]Q1)-- ( [c]Q2)--([c]Q3)--([c]Q4)--([c]Q1);   
  \node at ([c]Qmiddle) [below] {\scriptsize{$\pazocal D^{(2,1^2)}$}};

 \end{scope}

\begin{scope}[every coordinate/.style={shift={(14,0)}}]
\draw[thick] ([c]A) .. controls ([c]B) ..
       ([c]C) .. controls ([c]D)  .. ([c]E) .. controls ([c]F).. ([c]G);
      \draw[thick]  plot [smooth, tension=3] coordinates { ([c]A) ([c]H1) ([c]H2)};
       \draw[thick]  plot [smooth, tension=3] coordinates { ([c]G) ([c]J1) ([c]J2)};
       \draw[->,thick] ([c]Fs) to ([c]Ft);
        \node at ([c]Flabel) [above] {\scriptsize{$f$}};
       \end{scope}
 \coordinate (Q2) at (10,5.5);
 \coordinate (Q3) at (10,1);
 \coordinate (Q4) at (.5,1);
 \coordinate (Qmiddle) at (5.25,1);
 \coordinate (Q3long) at (10,-2.5 );
 \coordinate (Q4long) at (.5,-2.5);
\coordinate (Qlongmiddle) at (5.5,-2.5);

  \begin{scope}[every coordinate/.style={shift={(14,0)}}]     
 \draw ([c]Q1)-- ( [c]Q2)--([c]Q3)--([c]Q4)--([c]Q1);   
 \end{scope}

 \coordinate (Am) at (0,0);
 \coordinate (Bm) at(0.2,-.3);
 \coordinate (Cm) at  ( .4,0);
 \coordinate (Dm) at (.6,.3)  ;
  \coordinate (Em) at (.8,0) ;
  \coordinate (Fm) at (1,-.3);
  \coordinate (Gm) at  (1.2,0);
  \coordinate (H1m) at (-.5,.3);
 \coordinate (H2m) at (-.2,.2);
\coordinate (J1m) at (1.4,.2);
 \coordinate (J2m) at (1.7,.3);
 \coordinate (N1m) at (3.2,-.3);
 \coordinate (N2m) at (.2,-.3);
  \coordinate (N3m) at (3,.3);
 \coordinate(P1) at  (-0.2,1);
 \coordinate (P2) at ( 3.5,1);
 \coordinate (P3) at (2.5,-1 );
 \coordinate (P4) at (-1.2,-1);
 
 \begin{scope}[every coordinate/.style={shift={(20,3)}}]
 \draw[thick]  plot [smooth, tension=3] coordinates {([c]H1m) ([c]H2m) ([c]Am)  };
\draw[thick] ([c]Am) .. controls ([c]Bm) ..
       ([c]Cm) .. controls ([c]Dm)  .. ([c]Em) .. controls ([c]Fm).. ([c]Gm);
    \draw[thick]  plot [smooth, tension=3] coordinates { ([c]Gm) ([c]J1m) ([c]J2m)  };   
     \draw[thick]   ([c]J2m).. controls ([c]N1m) and ([c]N2m).. ([c]N3m);
 \draw[dashed] ([c]P1)-- ( [c]P2)--([c]P3)--([c]P4)--([c]P1);  
    \end{scope}   
       
       
\draw (7.5,1)--(4,-1.5);
\draw (7.5,1)--(13.,-1.5);
\draw (7.5,1)--(23.,-1.5);

 \node at (12.3,3) [above] { $\cap$};


 
  \coordinate(S1) at (3,-.9);
    \coordinate(S2) at (5.3,-.9);
        \coordinate (Y1) at (4.15,-2);
   \coordinate (Y2) at (4.15,.1);
           \coordinate (O) at (4.15,-.9);
    \coordinate (O3) at (3.2,-1.8);
   \coordinate (O4) at (5,-1.8);

 \begin{scope}[every coordinate/.style={shift={(-2,-7)}}]
\draw[draw=gray] ([c]A) .. controls ([c]B) ..
       ([c]C) .. controls ([c]D)  .. ([c]E) .. controls ([c]F).. ([c]G);
      \draw[draw=gray]  plot [smooth, tension=3] coordinates { ([c]A) ([c]H1) ([c]H2)};
       \draw[draw=gray]  plot [smooth, tension=3] coordinates { ([c]G) ([c]J1) ([c]J2)};
       \draw[thick] ([c]R1) to ([c]R2);
            \draw[thick] ([c]T1) to ([c]T2);
                   \draw[thick] ([c]R1b) to ([c]R2b);
                   
             \node at ([c]R1) [above] {\scriptsize{$T_1$}};
             \node at ([c]T1) [xshift=-.15cm,yshift=.1cm] {\scriptsize{$T_2$}};
              \node at ([c]R1b) [xshift=-.1cm,yshift=-.25cm] {\scriptsize{$T_3=\overline T_1$}};
              
             \node at ([c]R2) [above] {\scriptsize{$2$}};
              \node at ([c]T2) [below] {\scriptsize{$1$}};
              \node at ([c]R2b) [below] {\scriptsize{$1$}};
              
   \draw[->,thick] ([c]Fs) to ([c]Ft);
        \node at ([c]Flabel) [above] {\scriptsize{$f$}};
             \node at ([c]J2) [below] {\begin{color}{gray}\scriptsize{$Z$}\end{color}};
                          \draw[->,dotted] ([c]fact1) to ([c]fact2);
                          
                   \draw[thick] ([c]S1) to ([c]S2);
                   \draw[thick] ([c]Y1) to ([c]Y2);
                    \draw[thick] ([c]O) parabola ([c]O3);
                        \draw[thick] ([c]O) parabola ([c]O4);
                        
                        \draw[->,thick] ([c]Fbs) to ([c]Fbt);         
 \node at ([c]Fblabel) [right] {\scriptsize{$\bar{f}$}};

 \end{scope}

 \begin{scope}[every coordinate/.style={shift={(-2,-7)}}]     
 \draw ([c]Q1)-- ( [c]Q2)--([c]Q3long)--([c]Q4long)--([c]Q1);   
  \node at ([c]Qlongmiddle) [below] {\scriptsize{$\pazocal D^{(2,1^2)}\cap\overline{\pazocal M}^{\text{main}}$, conjugate tails}};
 \end{scope}

  
 \coordinate (E4) at  (.9,.2);
 \coordinate (E5) at (-.3,-.2) ;
  \coordinate (E6) at (2,-.2) ;
   \coordinate (E8) at (.9,-.6) ;
   \coordinate (E9) at (.9,1) ;
   
   \coordinate (E7) at (.9,-.2) ;

  \begin{scope}[every coordinate/.style={shift={(4,-4)}}]

 \draw[thick] ([c]E4)  ellipse (1cm and 0.4cm);
       \draw[thick] ([c]E5)--([c]E6);
        \draw[thick] ([c]E8)--([c]E9);
        \node at ([c]E6) [xshift=.35cm]{\scriptsize{$f(T_3)$}};
   
 \draw[dashed] ([c]P1)-- ([c]P2)--([c]P3)--([c]P4)--([c]P1);  
    
     \fill[gray] ([c]E7) circle (3pt);
 
          \node at ([c]E7) [below] {\begin{color}{gray}\scriptsize{$f(Z)$}\end{color}};
           \node at ([c]E7) [xshift=.2cm,yshift=.2cm] {\scriptsize{$D_6$}};
          
\end{scope}



        \coordinate (y1) at (4.5,-2);
   \coordinate (y2) at (3.8,.1);
   \coordinate (z1) at (4.5,.1);
   \coordinate (z2) at (3.8,-2);
           \coordinate (O) at (4.15,-.9);

 \begin{scope}[every coordinate/.style={shift={(8,-7)}}]
\draw[draw=gray] ([c]A) .. controls ([c]B) ..
       ([c]C) .. controls ([c]D)  .. ([c]E) .. controls ([c]F).. ([c]G);
      \draw[draw=gray]  plot [smooth, tension=3] coordinates { ([c]A) ([c]H1) ([c]H2)};
       \draw[draw=gray]  plot [smooth, tension=3] coordinates { ([c]G) ([c]J1) ([c]J2)};
       
       \draw[thick] ([c]R1) to ([c]R2);
            \draw[thick] ([c]T1) to ([c]T2);
                   \draw[thick] ([c]R1b) to ([c]R2b);
           \draw[dashed] (9.7,-2.5)node[xshift=-.15cm,yshift=.09cm]{\scriptsize{$\overline T_1$}} -- (11.5,-2);
             \node at ([c]R2) [yshift=-.15cm] {\scriptsize{$2$}};
              \node at ([c]T2) [yshift=-.15cm] {\scriptsize{$1$}};
              \node at ([c]R2b) [yshift=-.15cm] {\scriptsize{$1$}};
              
             \node at ([c]R1) [xshift=-.15cm] {\scriptsize{$T_1$}};
             \node at ([c]T1) [xshift=-.15cm,yshift=.1cm] {\scriptsize{$T_2$}};
              \node at ([c]R1b) [xshift=-.2cm,yshift=-.1cm] {\scriptsize{$T_3$}};

   \draw[->,thick] ([c]Fs) to ([c]Ft);
        \node at ([c]Flabel) [above] {\scriptsize{$f$}};
             \node at ([c]J2) [below] {\begin{color}{gray}\scriptsize{$Z$}\end{color}};
                          \draw[->,dotted] ([c]fact1) to ([c]fact2);
                          
                   \draw[dashed] ([c]S1) to ([c]S2);
                   \draw[thick] ([c]y1) to ([c]y2);
                   \draw[thick] ([c]z1) to ([c]z2);
                    \draw[thick] ([c]O) parabola ([c]O3);
                        \draw[thick] ([c]O) parabola ([c]O4);
                         \node at ([c]O4) [right] {\scriptsize{$T_1$}};
                        \draw[->,thick] ([c]Fbs) to ([c]Fbt);         
 \node at ([c]Fblabel) [right] {\scriptsize{$\bar{f}$}};

 \end{scope}

 \begin{scope}[every coordinate/.style={shift={(8,-7)}}]     
 \draw ([c]Q1)-- ( [c]Q2)--([c]Q3long)--([c]Q4long)--([c]Q1);   
  \node at ([c]Qlongmiddle) [below] {\scriptsize{$\pazocal D^{(2,1^2)}\cap\overline{\pazocal M}^{\text{main}}$, general core ($I\!I_4$)}};
 \end{scope}


 \coordinate (F5) at (-.3,.2) ;
  \coordinate (F6) at (2,.2) ;
   \coordinate (F8) at (.1,.9) ;
   \coordinate (F9) at (1.4,-.3) ;
   
   \coordinate (F7) at (.9,.2) ;
  
  \coordinate (F10) at (1.7,.2) ;

  \begin{scope}[every coordinate/.style={shift={(14,-4)}}]

       \draw[thick] ([c]F5)--([c]F6);
        \draw[thick] ([c]E8)--([c]E9);
         \draw[thick] ([c]F8)--([c]F9);
   
 \draw[dashed] ([c]P1)-- ( [c]P2)--([c]P3)--([c]P4)--([c]P1);  
    \node at ([c]F6) [xshift=.3cm] {\scriptsize{$f(T_1)$}};
    
     \fill[gray] ([c]F7) circle (3pt);
     \fill[blue] ([c]F10) circle (3pt);
 
          \node at ([c]F7) [xshift=-.3cm,yshift=-.2cm] {\begin{color}{gray}\scriptsize{$f(Z)$}\end{color}};
          \node at ([c]F7) [xshift=.2cm,yshift=.2cm] {\scriptsize{$D_5$}};
\end{scope}



        \coordinate (w1) at (3.15,-2);
   \coordinate (w2) at (3.15,.1);
    \coordinate (u1) at (3.35,-2);
   \coordinate (u2) at (3.35,.1);
           
     \coordinate (t1) at (3.,-.3);
   \coordinate (t2) at (4.8,.1);
   
      \coordinate (s1) at (3.,-1);
   \coordinate (s2) at (4.8,-.8);
   
      \coordinate (x1) at (3.,-1.7);
   \coordinate (x2) at (4.8,-1.5);

 \begin{scope}[every coordinate/.style={shift={(18,-7)}}]
\draw[draw=gray] ([c]A) .. controls ([c]B) ..
       ([c]C) .. controls ([c]D)  .. ([c]E) .. controls ([c]F).. ([c]G);
      \draw[draw=gray]  plot [smooth, tension=3] coordinates { ([c]A) ([c]H1) ([c]H2)};
       \draw[draw=gray]  plot [smooth, tension=3] coordinates { ([c]G) ([c]J1) ([c]J2)};
       \draw[thick] ([c]R1) to ([c]R2);
            \draw[thick] ([c]T1) to ([c]T2);
                   \draw[thick] ([c]R1b) to ([c]R2b);
             \node at ([c]R2) [above] {\scriptsize{$2$}};
              \node at ([c]T2) [below] {\scriptsize{$1$}};
              \node at ([c]R2b) [below] {\scriptsize{$1$}};
              
             \node at ([c]R1) [above] {\scriptsize{$T_1$}};
             \node at ([c]T1) [xshift=-.15cm,yshift=.1cm] {\scriptsize{$T_2$}};
              \node at ([c]R1b) [xshift=-.2cm,yshift=-.1cm] {\scriptsize{$T_3$}};

   \draw[->,thick] ([c]Fs) to ([c]Ft);
        \node at ([c]Flabel) [above] {\scriptsize{$f$}};
             \node at ([c]J2) [below] {\begin{color}{gray}\scriptsize{$Z$}\end{color}};
                          \draw[->,dotted] ([c]fact1) to ([c]fact2);

                        \draw[->,thick] ([c]Fbs) to ([c]Fbt);         
 \node at ([c]Fblabel) [right] {\scriptsize{$\bar{f}$}};

 \end{scope}
 
  \begin{scope}[every coordinate/.style={shift={(18.3,-7)}}]
    \draw[gray] ([c]w1) -- ([c]w2)--([c]u2) -- ([c]u1)-- ([c]w1);
                     \draw[thick] ([c]t1) to ([c]t2);
                       \draw[thick] ([c]s1) to ([c]s2);
                         \draw[thick] ([c]x1) to ([c]x2);
 \end{scope}

 \begin{scope}[every coordinate/.style={shift={(18,-7)}}]     
 \draw ([c]Q1)-- ( [c]Q2)--([c]Q3long)--([c]Q4long)--([c]Q1);   
  \node at ([c]Qlongmiddle) [below] {\scriptsize{$\pazocal D^{(2,1^2)}\cap\overline{\pazocal M}^{\text{main}}$, general core (ribbon)}};
 \end{scope}


   \coordinate (G8) at (-.3,-.1) ;
   \coordinate (G9) at (2,-.1) ;
   
   \coordinate (E7) at (.9,-.2) ;
  
    \coordinate (G10) at (2.3,0.1) ; 

  \begin{scope}[every coordinate/.style={shift={(24,-4)}}]

 \draw[thick] ([c]E4)  ellipse (1cm and 0.4cm);
       \draw[thick] ([c]E5)--([c]E6);
        \draw[thick] ([c]G8)--([c]G9);
   
 \draw[dashed] ([c]P1)-- ( [c]P2)--([c]P3)--([c]P4)--([c]P1);  
    
     \fill[gray] ([c]E7) circle (3pt);
 
          \node at ([c]E7) [below] {\begin{color}{gray}\scriptsize{$f(Z)$}\end{color}};
           \node at ([c]G10) [above] {\scriptsize{$f(T_1)$}};
          \node at ([c]E6) [xshift=.5cm] {\scriptsize{$f(T_2)=$}};
          \node at ([c]E6) [xshift=.5cm,yshift=-.3cm] {\scriptsize{$f(T_3)$}};
\end{scope}
 \end{tikzpicture}
\caption{$\pazocal D^{(2,1^2)},\dim=14$; $\pazocal D^{(2,1^2)}\cap\overline{\pazocal M}^{\text{main}},\dim=12$}
\label{D211}
\end{figure}

  \begin{itemize}[leftmargin=.5cm]
  \item either the attaching point of the degree $2$ tail is conjugate to another attaching point, in which case factoring through a type $I\!I_3$ (i.e. $D_6$) singularity implies that the corresponding line is tangent to the conic at the concurrency point;
  \item or the core is general, and  the (partially detsabilized) maps factors through a type $I\!I_4$ singularity with special branches the degree 2 tail $T_1$ and its conjugate (obtained by sprouting); this implies that the image is the union of three concurrent lines, one covered two-to-one with one ramification point over the singular point;
  \item or the core is general, and factoring through a contracted ribbon (which in the modification $C\to\overline C$ is the image of the irreducible genus 2 component) with three tails shows that the image is the union of a conic with a tangent line (the two degree $1$ tails must map to the same line).
 \end{itemize}
  In any case the dimension is $7+5=5+1+6=6+6=12$. See Figure \ref{D211}.
  
 Notice that we did not consider the lifts of $C$ to an admissible cover such that one of the tails is attached to a Weierstrass point. For $T_2$ and $T_3$, this is simply because there is no factorisation through a singularity with special cuspidal branch of degree one; if instead is $T_1$ that is attached to a Weierstrass point,  then a map factoring through an $I_3$ singularity will have for image three concurrent lines, with $T_1$ covering $2:1$ its image and ramifying at the node. The latter however is the same factorisation condition we saw in case two above.

 \item[$\pazocal D^{(1^4)}$ component] The image consists of four concurrent lines. The dimension is $13$. The intersection with main has only one component consisting of maps  factoring through a contracted ribbon with four tails. We will show that, if we just look at the valuation of the sections around the nodes, we can always construct \emph{some} ribbon with four tails to which they descend.
However, we will see that there is an extra condition on cross-ratios, matching that of the four lines with that of the four nodes under the hyperelliptic projection on the core.

 To verify the first claim, denote by $L$ the line bundle which restricts to the trivial one on the core, and has degree $1$ on every tail.
 Let
  $V$ be a subspace of the space $\operatorname{H}^0(C,L)_0$ of sections vanishing along $Z.$   For $V$ to be smoothable it is necessary and sufficient that it has codimension $2$ in $\operatorname{H}^0(C,L)_0$. The sections extending to a given smoothing span a codimension $2$ subspace in $\operatorname{H}^0(C,L)$, because the dimension of the space of  global sections on the generic fiber is $3$, and it is $5$  on the nodal curve. Hence the codimension $2$ condition is necessary.
  We now want to argue that it is also sufficient by showing that we can always construct a  tailed ribbon $\overline{C},$ 
  to which the sections in $V$  descend.  
 

We will see that a tailed ribbon $\overline C=R\cup_{\epsilon_i} T_i$ to which these sections descend is completely determined by a choice of tangent vectors on the tails, which is itself determined by a choice of local coordinates.
 
The sheaf of regular functions on $\overline C$ is defined as the kernel of a morphism:
  \[\mathcal O_R\oplus\bigoplus_i\mathcal O_{T_i }\xrightarrow{\nu}\bigoplus_i \mathbb C[\epsilon_i]/\epsilon_i^2 \]
  where $ O_R\cong\mathcal O_{\PP^1}\oplus\mathcal O_{\PP^1}(1)$, and the second summand gives regular functions on the ribbon vanishing on the reduced curve.
The morphism $\nu$ restricted to the tails $T_i$ only depends on the behaviour of the functions to the first order, and it is thus fully determined by a choice of tangent vectors on the tails.
Once $\nu$ is defined on the tails, in order to determine $\overline C,$  we may choose an identification of the sections of the ribbon vanishing on the reduced curve $\operatorname{H}^0(R,\mathcal O_{\PP^1}(1))$  with the two dimensional vector space $V$.  For any such choice, the sections of  $V$ descend to $\overline C$ by construction.


 To see that factorisation through a given ribbon involves a condition on the cross ratio, consider a general smoothing $f\colon\mathcal C\to\PP^2_\dvr$ over the spectrum of a discrete valuation ring $\dvr$ (sometimes called \emph{trait} in this note). The image of the general fibre $\mathcal C_\eta$ is a cubic with one node, whose limit at $0$ is the point $x$ where the core $Z$ of $\mathcal C_0$ is contracted, i.e. the concurrency point of the four lines. Let us blow $\PP^2$ up at $x$. We thus obtain a diagram:
 \bcd
 \mathcal C\ar[d]\ar[r,dashed, "\widetilde{f}" ]\ar[dr] &\operatorname{Bl}_x\PP^2\ar[d]\\
 \overline{\mathcal C}\ar[r] & \PP^2
 \ecd
 The strict transform of $\mathcal C_\eta$ intersects the exceptional divisor $E$ transversely at two points $p_1$ and $p_2$. Write $\widetilde{f}^*E=p_1+p_2+\alpha Z$. Matching with a tail $T_i$ of $\mathcal C_0$ we get:
 \[1=\widetilde{f}^*E\cdot T_i=\alpha.\]
 Now, matching with $Z$ we get:
 \[-\operatorname{deg}(\widetilde{f}_{|Z})=\widetilde{f}^*E\cdot Z=\operatorname{deg}(p_1+p_2+Z^2)=\operatorname{deg}\left(\sum_{i=1}^2p_i-\sum_{j=1}^4q_j\right),\]
 where $q_j$ denotes a node of $\mathcal C_0$. In particular, $\widetilde{f}_{|Z}\colon Z\to E$ has to be the hyperelliptic cover. This shows that, for a map in $\pazocal D^{(1^4)}$ to be smoothable, the cross-ratio of the image lines in $\PP^2$ and that of the tail attaching points (after hyperelliptic map) must coincide. This cuts the dimension down by one to $12$.
 Notice that the condition on the cross-ratio only depends on the marked curve $\left(Z, \overline{C\setminus Z}\cap Z\right)$ and on its image along $f$, but not  on the  choice of a smoothing.  
 
 \item[${}^4\!\pazocal{E}$ component] The general point of this component has a reducible core, consisting of an elliptic curve $E_1$ normalising a two-nodal quartic, and another elliptic curve $E_2$ contracted anywhere on it. The dimension is $14$. Smoothability is really a genus one problem: a map is smoothable if its image contains a cusp, i.e. it is of type $A_1-A_2$. The dimension is $12$. See Figure \ref{4E}.
 
\begin{figure}[h]
   \begin{tikzpicture}[scale=.4]

\coordinate (H1) at (2.3,4.3);
 \coordinate (H2) at (2.4,4.8);
 
 \coordinate (AE1) at (2,4);
 \coordinate (BE1) at (1.7,3.8);
 \coordinate (CE1) at   (2,3.6);
 \coordinate (HE1) at (2.3,3.3);
 \coordinate (HE2) at (2.4,2.8);
 
 \coordinate (JE1) at (2.4,3.6);
 \coordinate (DE2) at (2.3,3.4) ;
 
  \coordinate (EE2) at (2,3.2) ;
  \coordinate (FE2) at (1.7,3);
  \coordinate (GE2) at  (2,2.8);
 \coordinate (J1) at (2.3,2.5);
 \coordinate (J2) at (2.4,2);
 
   \coordinate(fact1) at (2.4,1);
   \coordinate(fact2) at (3.1,-.4);
   
 \coordinate (Fs) at (3,3.5);
  \coordinate (Ft) at (4.6,3.5);
   \coordinate (Flabel) at (3.8,3.5);
   
    \coordinate (Fbs) at (5.3,-.3);
  \coordinate (Fbt) at (6.4,1);
    \coordinate (Fblabel) at (5.9,.1);

   \begin{scope}[every coordinate/.style={shift={(0,0)}}]
\draw[draw=gray] ([c]AE1) .. controls ([c]BE1) ..
       ([c]CE1);
     \draw[thick] ([c]EE2)  .. controls ([c]FE2) .. ([c]GE2);
     
      \draw[draw=gray]  plot [smooth, tension=3] coordinates { ([c]AE1) ([c]H1) ([c]H2)};
       \draw[draw=gray]  plot [smooth, tension=3] coordinates { ([c]CE1) ([c]HE1) ([c]HE2)};
       
        \draw[thick]  plot [smooth, tension=3] coordinates { ([c]EE2) ([c]DE2) ([c]JE1)};
       \draw[thick]  plot [smooth, tension=3] coordinates { ([c]GE2) ([c]J1) ([c]J2)};

       \draw[->,thick] ([c]Fs) to ([c]Ft);
        \node at ([c]Flabel) [above] {\scriptsize{$f$}};
             \node at ([c]J2) [yshift=-.2cm] {\scriptsize{$E_1,g=1,d=4$}};
         \node at ([c]H2) [xshift=.15cm,yshift=-.08cm] {\begin{color}{gray}\scriptsize{$E_2,g=1$}\end{color}};

 \end{scope}

 \coordinate(Q1) at  (0,5.5);
 \coordinate (Q2) at (10,5.5);
 \coordinate (Q3) at (10,1);
 \coordinate (Q4) at (0,1);
 \coordinate (Qmiddle) at (5.25,1);
 \coordinate (Q3long) at (10,-2.5 );
 \coordinate (Q4long) at (0,-2.5);
\coordinate (Qlongmiddle) at (5,-2.5);

  \begin{scope}[every coordinate/.style={shift={(0,0)}}]     
 \draw (Q1)-- ( Q2)--(Q3)--(Q4)--(Q1);   
 \end{scope}


 \coordinate (An) at (-.5,.3);
 \coordinate (B1m) at(2,-.5);
  \coordinate (B2m) at(-1.8,-.5);
 \coordinate (Cn) at  (1,.3);
 \coordinate (D1n) at (3,-.5)  ;
  \coordinate (D2n) at (-.5,-.5) ;
  \coordinate (En) at (1.9,.4);

   \coordinate (F2n) at (2.4,.6);
     \coordinate (F1n) at  (2.2,0);
     \coordinate (F3n) at  (2,-.4);
   \coordinate (Gn) at (2.5,-.9);

 \coordinate(P1) at  (-0.2,1);
 \coordinate (P2) at ( 3.5,1);
 \coordinate (P3) at (2.5,-1 );
 \coordinate (P4) at (-1.2,-1);
 
 \begin{scope}[every coordinate/.style={shift={(6,3)}}]

\draw[thick] ([c]An) .. controls ([c]B1m) and ([c]B2m) ..
       ([c]Cn) .. controls ([c]D1n) and ([c]D2n)  .. ([c]En);
    
   \draw[thick] ([c]En)  .. controls ([c]F2n)   ..  ([c]F1n).. controls ([c]F3n) .. ([c]Gn) ;
 \draw[dashed] ([c]P1)-- ( [c]P2)--([c]P3)--([c]P4)--([c]P1);

     \fill[gray] ([c]Cn) circle (3pt);
          \node at ([c]Cn) [yshift=.23cm] {\begin{color}{gray}\scriptsize{$f(E_2)$}\end{color}};

    \end{scope}


    \coordinate (A) at (2,4);
 \coordinate (B) at (1.7,3.8);
 \coordinate (C) at   (2,3.6);
 \coordinate (D) at (2.3,3.4) ;
  \coordinate (E) at (2,3.2) ;
  \coordinate (F) at (1.7,3);
  \coordinate (G) at  (2,2.8);
  \coordinate (H1) at (2.3,4.3);
 \coordinate (H2) at (2.4,4.8);
  \coordinate (J1) at (2.3,2.4);
 \coordinate (J2) at (2.4,2);

\begin{scope}[every coordinate/.style={shift={(14,0)}}]
\draw[thick] ([c]A) .. controls ([c]B) ..
       ([c]C) .. controls ([c]D)  .. ([c]E) .. controls ([c]F).. ([c]G);
      \draw[thick]  plot [smooth, tension=3] coordinates { ([c]A) ([c]H1) ([c]H2)};
       \draw[thick]  plot [smooth, tension=3] coordinates { ([c]G) ([c]J1) ([c]J2)};
       \draw[->,thick] ([c]Fs) to ([c]Ft);
        \node at ([c]Flabel) [above] {\scriptsize{$f$}};
       \end{scope}

  \begin{scope}[every coordinate/.style={shift={(14,0)}}]     
 \draw ([c]Q1)-- ( [c]Q2)--([c]Q3)--([c]Q4)--([c]Q1);   
 \end{scope}
 

 \coordinate (Am) at (0,0);
 \coordinate (Bm) at(0.2,-.3);
 \coordinate (Cm) at  ( .4,0);
 \coordinate (Dm) at (.6,.3)  ;
  \coordinate (Em) at (.8,0) ;
  \coordinate (Fm) at (1,-.3);
  \coordinate (Gm) at  (1.2,0);
  \coordinate (H1m) at (-.5,.3);
 \coordinate (H2m) at (-.2,.2);
\coordinate (J1m) at (1.4,.2);
 \coordinate (J2m) at (1.7,.3);
 \coordinate (N1m) at (3.2,-.3);
 \coordinate (N2m) at (.2,-.3);
  \coordinate (N3m) at (3,.3);

 \begin{scope}[every coordinate/.style={shift={(20,3)}}]
 
 \draw[thick]  plot [smooth, tension=3] coordinates {([c]H1m) ([c]H2m) ([c]Am)  };
 
\draw[thick] ([c]Am) .. controls ([c]Bm) ..
       ([c]Cm) .. controls ([c]Dm)  .. ([c]Em) .. controls ([c]Fm).. ([c]Gm);
    \draw[thick]  plot [smooth, tension=3] coordinates { ([c]Gm) ([c]J1m) ([c]J2m)  };   
     \draw[thick]   ([c]J2m).. controls ([c]N1m) and ([c]N2m).. ([c]N3m);
 \draw[dashed] ([c]P1)-- ( [c]P2)--([c]P3)--([c]P4)--([c]P1);  
    \end{scope}   
       
\node at (5,1) [below] {\scriptsize{${}^4\pazocal E$}};
       
\draw (7.5,1)--(11,-.5);
\draw (19.5,1)--(16.,-.5);

 \node at (12.3,3) [above] { $\cap$};


 \begin{scope}[every coordinate/.style={shift={(8,-6)}}]

 \draw ([c]Q1)-- ( [c]Q2)--([c]Q3long)--([c]Q4long)--([c]Q1);   
 \end{scope}

 \begin{scope}[every coordinate/.style={shift={(8,-6)}}]
\draw[draw=gray] ([c]AE1) .. controls ([c]BE1) ..
       ([c]CE1);
     \draw[thick] ([c]EE2)  .. controls ([c]FE2) .. ([c]GE2);
     
      \draw[draw=gray]  plot [smooth, tension=3] coordinates { ([c]AE1) ([c]H1) ([c]H2)};
       \draw[draw=gray]  plot [smooth, tension=3] coordinates { ([c]CE1) ([c]HE1) ([c]HE2)};
       
        \draw[thick]  plot [smooth, tension=3] coordinates { ([c]EE2) ([c]DE2) ([c]JE1)};
       \draw[thick]  plot [smooth, tension=3] coordinates { ([c]GE2) ([c]J1) ([c]J2)};

       \draw[->,thick] ([c]Fs) to ([c]Ft);
        \node at ([c]Flabel) [above] {\scriptsize{$f$}};
             \node at ([c]J2) [below] {\scriptsize{$E_1$}};
         \node at ([c]H2) [xshift=.2cm,yshift=-.1cm] {\begin{color}{gray}\scriptsize{$E_2$}\end{color}};

                          \draw[->,dotted] ([c]fact1) to ([c]fact2);

                        \draw[->,thick] ([c]Fbs) to ([c]Fbt);         
 \node at ([c]Fblabel) [right] {\scriptsize{$\bar{f}$}};
\node at ([c]Qlongmiddle) [below] {\scriptsize{${}^4\pazocal E\cap\overline{\pazocal M}^{\text{main}}$}};
 \end{scope}
 
\begin{scope}[every coordinate/.style={shift={(10.5,-9.5)}}]

\draw[thick] ([c]BE1) parabola ([c]H2);
 \draw[thick] ([c]BE1) parabola ([c]JE1);
 \draw[thick] ([c]EE2)  .. controls ([c]FE2) .. ([c]GE2);
 
\draw[thick]  plot [smooth, tension=3] coordinates { ([c]EE2) ([c]DE2) ([c]JE1)};
       \draw[thick]  plot [smooth, tension=3] coordinates { ([c]GE2) ([c]J1) ([c]J2)};

\end{scope}


\coordinate (AA1) at (.5,.7);
\coordinate (AA2) at (0,.2);
\coordinate (AS0) at (.1,.2);

\coordinate (AS2) at (-.5,-.1);
\coordinate (AS4) at (.3,-.4);
\coordinate (AS3) at (-.2,-.3);
\coordinate (AS5) at (.7,-1);

\coordinate (AA3) at (1,-.2);
\coordinate (AA4) at (4,1);
\coordinate (AA5) at (1,1);
\coordinate (AA6) at (2.1,-.5);

 \coordinate(P1) at  (-0.2,1);
 \coordinate (P2) at ( 3.5,1);
 \coordinate (P3) at (2.5,-1 );
 \coordinate (P4) at (-1.2,-1);
 
 \begin{scope}[every coordinate/.style={shift={(14,-3)}}]
 
 \draw[thick] ([c]AA3) parabola ([c]AA1); 
  \draw[thick] ([c]AA1) .. controls ([c]AS0) .. ([c]AA2);
   \draw[thick] ([c]AA3) .. controls ([c]AA4) and ([c]AA5).. ([c]AA6); 
    \draw[thick] ([c]AA2) .. controls ([c]AS2) .. ([c]AS3)..controls ([c]AS4) .. ([c]AS5); 
\draw[dashed] ([c]P1)-- ( [c]P2)--([c]P3)--([c]P4)--([c]P1); 
     \fill[gray] ([c]AA1) circle (3pt);
      \node at ([c]AA1) [yshift=.2cm] {\begin{color}{gray}\scriptsize{$f(Z)$}\end{color}};
      \node at ([c]AA1) [xshift=.25cm] {\scriptsize{$A_2$}};
            \node at ([c]AA6) [xshift=-.2cm] {\scriptsize{$A_1$}};
\end{scope}

\end{tikzpicture}
\caption{${}^4\!\pazocal E,\dim=14$; ${}^4\!\pazocal E\cap\overline{\pazocal M}^{\text{main}},\dim=12$}
\label{4E}	
\end{figure}

 \item[${}^3\!\pazocal{E}^{(1)}$ component] The image consists of a smooth cubic $E_1$ and a general line; the second elliptic curve $E_2$ is contracted to one of the three intersection points. The dimension of the boundary component ${}^3\!\pazocal{E}^{(1)}$ is $13$. At the intersection with \emph{main}, we see those maps such that the line is tangent to the cubic. This is again a divisor in \emph{main}, so it has dimension $12$. See Figure \ref{3E1}.

\begin{figure}[h]
   \begin{tikzpicture}[scale=.4]

\coordinate (H1) at (2.3,4.3);
 \coordinate (H2) at (2.4,4.8);
 
 \coordinate (AE1) at (2,4);
 \coordinate (BE1) at (1.7,3.8);
 \coordinate (CE1) at   (2,3.6);
 \coordinate (HE1) at (2.3,3.3);
 \coordinate (HE2) at (2.4,2.8);
 
 \coordinate (JE1) at (2.4,3.6);
 \coordinate (DE2) at (2.3,3.4) ;
 
  \coordinate (EE2) at (2,3.2) ;
  \coordinate (FE2) at (1.7,3);
  \coordinate (GE2) at  (2,2.8);
 \coordinate (J1) at (2.3,2.5);
 \coordinate (J2) at (2.4,2);
 
 \coordinate (T1) at (1.8,4.1);
 \coordinate (T2) at (3.4,4.3);

   \coordinate(fact1) at (2.4,1); 
   \coordinate(fact2) at (3.1,-.4);
   
 \coordinate (Fs) at (3,3.5);
  \coordinate (Ft) at (4.6,3.5);
   \coordinate (Flabel) at (3.8,3.5);
   
    \coordinate (Fbs) at (5.3,-.3);
  \coordinate (Fbt) at (6.4,1);
    \coordinate (Fblabel) at (5.9,.1);

   \begin{scope}[every coordinate/.style={shift={(0,0)}}]
\draw[draw=gray] ([c]AE1) .. controls ([c]BE1) ..
       ([c]CE1);
     \draw[thick] ([c]EE2)  .. controls ([c]FE2) .. ([c]GE2);
     
      \draw[draw=gray]  plot [smooth, tension=3] coordinates { ([c]AE1) ([c]H1) ([c]H2)};
       \draw[draw=gray]  plot [smooth, tension=3] coordinates { ([c]CE1) ([c]HE1) ([c]HE2)};
       
        \draw[thick]  plot [smooth, tension=3] coordinates { ([c]EE2) ([c]DE2) ([c]JE1)};
       \draw[thick]  plot [smooth, tension=3] coordinates { ([c]GE2) ([c]J1) ([c]J2)};
       
        \draw[thick] ([c]T1) -- ([c]T2) node[yshift=.15cm]{\scriptsize{$1$}};
       
       \draw[->,thick] ([c]Fs) to ([c]Ft);
        \node at ([c]Flabel) [above] {\scriptsize{$f$}};
           \node at ([c]J2) [yshift=-.2cm] {\scriptsize{$E_1,g=1,d=3$}};
         \node at ([c]H2) [xshift=-.3cm,yshift=.1cm] {\begin{color}{gray}\scriptsize{$E_2,g=1$}\end{color}};

 \end{scope}

 \coordinate (Q2) at (10,5.5);
 \coordinate (Q3) at (10,1);
 \coordinate (Q4) at (0,1);
 \coordinate (Qmiddle) at (5,1);
 \coordinate (Q3long) at (10,-2.5 );
 \coordinate (Q4long) at (0,-2.5);
\coordinate (Qlongmiddle) at (5.5,-2.5);

 \coordinate(P1) at  (-0.2,1);
 \coordinate (P2) at ( 3.5,1);
 \coordinate (P3) at (2.5,-1 );
 \coordinate (P4) at (-1.2,-1);

  \begin{scope}[every coordinate/.style={shift={(0,0)}}]     
 \draw (Q1)-- ( Q2)--(Q3)--(Q4)--(Q1);   
 \end{scope}
 

 \coordinate (An) at (0,.5);
 \coordinate (B1m) at(.5,-.5);
   \coordinate (Cn) at  (1.3,.3);
 \coordinate (Dn) at (1.6,.8)  ;
  
  \coordinate (En) at (2,-.5);

   \coordinate (F2n) at (-.2,.1);
     \coordinate (F1n) at  (2.2,.1);
     \coordinate (Gn) at  (1.2,.15);

 \begin{scope}[every coordinate/.style={shift={(6,3)}}]

\draw[thick] ([c]An) .. controls ([c]B1m) ..
       ([c]Cn) .. controls ([c]Dn)   .. ([c]En);
    
   \draw[thick]  ([c]F2n) -- ([c]F1n);
 \draw[dashed] ([c]P1)-- ( [c]P2)--([c]P3)--([c]P4)--([c]P1);

     \fill[gray] ([c]Gn) circle (3pt);
          \node at ([c]Gn) [above] {\begin{color}{gray}\scriptsize{$f(E_2)$}\end{color}};

    \end{scope}   

     \node at (5,1) [below] {\scriptsize{${}^3\!\pazocal E^{(1)}$}};


    \coordinate (A) at (2,4);
 \coordinate (B) at (1.7,3.8);
 \coordinate (C) at   (2,3.6);
 \coordinate (D) at (2.3,3.4) ;
  \coordinate (E) at (2,3.2) ;
  \coordinate (F) at (1.7,3);
  \coordinate (G) at  (2,2.8);
  \coordinate (H1) at (2.3,4.3);
 \coordinate (H2) at (2.4,4.8);
  \coordinate (J1) at (2.3,2.4);
 \coordinate (J2) at (2.4,2);

\begin{scope}[every coordinate/.style={shift={(14,0)}}]
\draw[thick] ([c]A) .. controls ([c]B) ..
       ([c]C) .. controls ([c]D)  .. ([c]E) .. controls ([c]F).. ([c]G);
      \draw[thick]  plot [smooth, tension=3] coordinates { ([c]A) ([c]H1) ([c]H2)};
       \draw[thick]  plot [smooth, tension=3] coordinates { ([c]G) ([c]J1) ([c]J2)};
       \draw[->,thick] ([c]Fs) to ([c]Ft);
        \node at ([c]Flabel) [above] {\scriptsize{$f$}};
       \end{scope}

  \begin{scope}[every coordinate/.style={shift={(14,0)}}]     
 \draw ([c]Q1)-- ( [c]Q2)--([c]Q3)--([c]Q4)--([c]Q1);   
 \end{scope}
 

 \coordinate (Am) at (0,0);
 \coordinate (Bm) at(0.2,-.3);
 \coordinate (Cm) at  ( .4,0);
 \coordinate (Dm) at (.6,.3)  ;
  \coordinate (Em) at (.8,0) ;
  \coordinate (Fm) at (1,-.3);
  \coordinate (Gm) at  (1.2,0);
  \coordinate (H1m) at (-.5,.3);
 \coordinate (H2m) at (-.2,.2);
\coordinate (J1m) at (1.4,.2);
 \coordinate (J2m) at (1.7,.3);
 \coordinate (N1m) at (3.2,-.3);
 \coordinate (N2m) at (.2,-.3);
  \coordinate (N3m) at (3,.3);

 \begin{scope}[every coordinate/.style={shift={(20,3)}}]
 
 \draw[thick]  plot [smooth, tension=3] coordinates {([c]H1m) ([c]H2m) ([c]Am)  };
 
\draw[thick] ([c]Am) .. controls ([c]Bm) ..
       ([c]Cm) .. controls ([c]Dm)  .. ([c]Em) .. controls ([c]Fm).. ([c]Gm);
    \draw[thick]  plot [smooth, tension=3] coordinates { ([c]Gm) ([c]J1m) ([c]J2m)  };   
     \draw[thick]   ([c]J2m).. controls ([c]N1m) and ([c]N2m).. ([c]N3m);
 \draw[dashed] ([c]P1)-- ( [c]P2)--([c]P3)--([c]P4)--([c]P1);  
    \end{scope}

\draw (7.5,1)--(11,-.5);
\draw (19.5,1)--(16.,-.5);

 \node at (12.3,3) [above] { $\cap$};

 
  \coordinate (Je1) at (2.9,4.2);
  \coordinate (Je2) at (2.3,4.8);
  \coordinate (Y2) at (2.8,5);
   \coordinate (Y1) at (2.8,3);
 
 \begin{scope}[every coordinate/.style={shift={(8,-6)}}]

 \draw ([c]Q1)-- ( [c]Q2)--([c]Q3long)--([c]Q4long)--([c]Q1);   
 \end{scope}

 \begin{scope}[every coordinate/.style={shift={(8,-6)}}]
\draw[draw=gray] ([c]AE1) .. controls ([c]BE1) ..
       ([c]CE1);
     \draw[thick] ([c]EE2)  .. controls ([c]FE2) .. ([c]GE2);
     
      \draw[draw=gray]  plot [smooth, tension=3] coordinates { ([c]AE1) ([c]H1) ([c]H2)};
       \draw[draw=gray]  plot [smooth, tension=3] coordinates { ([c]CE1) ([c]HE1) ([c]HE2)};
       
        \draw[thick]  plot [smooth, tension=3] coordinates { ([c]EE2) ([c]DE2) ([c]JE1)};
       \draw[thick]  plot [smooth, tension=3] coordinates { ([c]GE2) ([c]J1) ([c]J2)};

       \draw[->,thick] ([c]Fs) to ([c]Ft);
        \node at ([c]Flabel) [above] {\scriptsize{$f$}};
             \node at ([c]J2) [below] {\scriptsize{$E_1$}};
         \node at ([c]H2) [right] {\begin{color}{gray}\scriptsize{$E_2$}\end{color}};
         
          \draw[thick] ([c]T2)  -- ([c]T1);
         
                          \draw[->,dotted] ([c]fact1) to ([c]fact2);

                        \draw[->,thick] ([c]Fbs) to ([c]Fbt);         
 \node at ([c]Fblabel) [right] {\scriptsize{$\bar{f}$}};
\node at ([c]Qlongmiddle) [below] {\scriptsize{${}^3\!\pazocal E^{(1)}\cap\overline{\pazocal M}^{\text{main}}$}};
 \end{scope}
 
\begin{scope}[every coordinate/.style={shift={(9.7,-10)}}]

  \draw[thick] ([c]EE2)  .. controls ([c]FE2) .. ([c]GE2);
     
      \draw[thick]  plot [smooth, tension=3] coordinates { ([c]EE2) ([c]DE2)};
       \draw[thick]  plot [smooth, tension=3] coordinates { ([c]GE2) ([c]J1) ([c]J2)};

        \draw[thick] ([c]Je2)  .. controls ([c]Je1) .. ([c]DE2);
        \draw[thick] ([c]Y2) -- ([c]Y1);
       
\end{scope}

\coordinate (f2n) at (-.2,-.28);
     \coordinate (f1n) at  (2.2,-.28);
\coordinate (gn) at (.5,-.28);
 
 \begin{scope}[every coordinate/.style={shift={(14,-3)}}]
 
 \draw[thick] ([c]An) .. controls ([c]B1m) ..
       ([c]Cn) .. controls ([c]Dn)   .. ([c]En);
    
   \draw[thick]  ([c]f2n) -- ([c]f1n);
 \draw[dashed] ([c]P1)-- ( [c]P2)--([c]P3)--([c]P4)--([c]P1);

     \fill[gray] ([c]gn) circle (3pt);
          \node at ([c]gn) [above] {\begin{color}{gray}\scriptsize{$f(E_2)$}\end{color}};
          \node at ([c]gn) [yshift=-.15cm] {\scriptsize{$A_3$}};
          \node at ([c]gn) [xshift=.85cm,yshift=-.15cm] {\scriptsize{$A_1$}};
\end{scope} 

\end{tikzpicture}
\caption{${}^3\!\pazocal E^{(1)},\dim=14$; ${}^3\!\pazocal E^{(1)}\cap\overline{\pazocal M}^{\text{main}},\dim=12$}
\label{3E1}
\end{figure}

 \item[${}^2\!\pazocal{E}^{(2)}$ component] The map restricts to the two-to-one cover of a line on $E_1$, and the embedding of a conic on the rational tail; the dimension of this component is $13$. The intersection with \emph{main} has two components: 
 \begin{itemize}[leftmargin=.5cm]
  \item factoring through a tacnode means that the line is tangent to the conic;
  \item sprouting, or equivalently replacing $E_2$ with a cusp on $E_1$, means that the conic intersects the line in one of the four branching points of the two-to-one cover.
 \end{itemize}
 Both have dimension $12$. See Figure \ref{2E2}.
 
 Notice that if we instead consider the sprouting at $E_1\cap E_2$, i.e. it is the rational component, then factorisation imposes that the map restrict to a two-to-one cover of a line on the rational component with a ramification point at the intersection with the line $f(E_1)$. These maps are in the closure of the first component of the intersection.

    \begin{figure}[h]
   \begin{tikzpicture}[scale=.5]

\coordinate (H1) at (2.3,4.3);
 \coordinate (H2) at (2.4,4.8);
 
 \coordinate (AE1) at (2,4);
 \coordinate (BE1) at (1.7,3.8);
 \coordinate (CE1) at   (2,3.6);
 \coordinate (HE1) at (2.3,3.3);
 \coordinate (HE2) at (2.4,2.8);
 
 \coordinate (JE1) at (2.4,3.6);
 \coordinate (DE2) at (2.3,3.4) ;
 
  \coordinate (EE2) at (2,3.2) ;
  \coordinate (FE2) at (1.7,3);
  \coordinate (GE2) at  (2,2.8);
 \coordinate (J1) at (2.3,2.5);
 \coordinate (J2) at (2.4,2);
 
 \coordinate (T1) at (1.8,4.1);
 \coordinate (T2) at (3.4,4.3);

   \coordinate(fact1) at (2.4,1);
   \coordinate(fact2) at (3.1,-.4);
   
 \coordinate (Fs) at (3,3.5);
  \coordinate (Ft) at (4.6,3.5);
   \coordinate (Flabel) at (3.8,3.5);
   
    \coordinate (Fbs) at (5.3,-.3);
  \coordinate (Fbt) at (6.4,1);
    \coordinate (Fblabel) at (5.9,.1);

   \begin{scope}[every coordinate/.style={shift={(0,0)}}]
\draw[draw=gray] ([c]AE1) .. controls ([c]BE1) ..
       ([c]CE1);
     \draw[thick] ([c]EE2)  .. controls ([c]FE2) .. ([c]GE2);
     
      \draw[draw=gray]  plot [smooth, tension=3] coordinates { ([c]AE1) ([c]H1) ([c]H2)};
       \draw[draw=gray]  plot [smooth, tension=3] coordinates { ([c]CE1) ([c]HE1) ([c]HE2)};
       
        \draw[thick]  plot [smooth, tension=3] coordinates { ([c]EE2) ([c]DE2) ([c]JE1)};
       \draw[thick]  plot [smooth, tension=3] coordinates { ([c]GE2) ([c]J1) ([c]J2)};
       
           \draw[thick] ([c]T1) --node[xshift=.1cm,yshift=-.15cm]{\scriptsize{$2$}} ([c]T2) ;
       
       \draw[->,thick] ([c]Fs) to ([c]Ft);
        \node at ([c]Flabel) [above] {\scriptsize{$f$}};
             \node at ([c]J2) [xshift=.1cm,below] {\scriptsize{$E_1,g=1,d=2$}};
         \node at ([c]H2) [xshift=-.1cm,yshift=-.1cm] {\begin{color}{gray}\scriptsize{$E_2,g=1$}\end{color}};

 \end{scope}

 \coordinate(Q1) at  (.5,5.5);
 \coordinate (Q2) at (10,5.5);
 \coordinate (Q3) at (10,1 );
 \coordinate (Q4) at (.5,1);
 \coordinate (Q3long) at (10,-2.5 );
 \coordinate (Q4long) at (.5,-2.5);
\coordinate (Qlongmiddle) at (5.5,-2.5);

 \coordinate(P1) at  (-0.2,1);
 \coordinate (P2) at ( 3.5,1);
 \coordinate (P3) at (2.5,-1 );
 \coordinate (P4) at (-1.2,-1);

  \begin{scope}[every coordinate/.style={shift={(0,0)}}]     
 \draw (Q1)-- ( Q2)--(Q3)--(Q4)--(Q1);   
 \end{scope}
 

 \coordinate (An) at (1,.1);

   \coordinate (F2n) at (-.2,.1);
     \coordinate (F1n) at  (2.2,.1);
     \coordinate (Gn) at  (1.68,.1);

 \begin{scope}[every coordinate/.style={shift={(6,3)}}]

   \draw[thick]  ([c]F2n) -- ([c]F1n);
 \draw[dashed] ([c]P1)-- ( [c]P2)--([c]P3)--([c]P4)--([c]P1);  
    \fill[blue] ([c]Gn)[xshift=-1cm] circle (3pt) ([c]Gn)[xshift=-1.6cm] circle (3pt) ([c]Gn)[xshift=-.5cm] circle (3pt) ([c]Gn)[xshift=.3cm] circle (3pt);
     \draw[thick]([c]An) circle (0.7cm);
     \fill[gray] ([c]Gn) circle (3pt);
          \node at ([c]Gn) [xshift=.3cm,yshift=.2cm] {\begin{color}{gray}\scriptsize{$f(E_2)$}\end{color}};

    \end{scope}   

     \node at (5,1) [below] {\scriptsize{${}^2\pazocal E^{(2)}$}};
      \node at (12.3,3) [above] {$\cap$};

   \draw (7.5,1)--(8.5,-.5);
\draw (7.5,1)--(19,-.5);


    \coordinate (A) at (2,4);
 \coordinate (B) at (1.7,3.8);
 \coordinate (C) at   (2,3.6);
 \coordinate (D) at (2.3,3.4) ;
  \coordinate (E) at (2,3.2) ;
  \coordinate (F) at (1.7,3);
  \coordinate (G) at  (2,2.8);
  \coordinate (H1) at (2.3,4.3);
 \coordinate (H2) at (2.4,4.8);
  \coordinate (J1) at (2.3,2.4);
 \coordinate (J2) at (2.4,2);

\begin{scope}[every coordinate/.style={shift={(14,0)}}]
\draw[thick] ([c]A) .. controls ([c]B) ..
       ([c]C) .. controls ([c]D)  .. ([c]E) .. controls ([c]F).. ([c]G);
      \draw[thick]  plot [smooth, tension=3] coordinates { ([c]A) ([c]H1) ([c]H2)};
       \draw[thick]  plot [smooth, tension=3] coordinates { ([c]G) ([c]J1) ([c]J2)};
       \draw[->,thick] ([c]Fs) to ([c]Ft);
        \node at ([c]Flabel) [above] {\scriptsize{$f$}};
       \end{scope}

  \begin{scope}[every coordinate/.style={shift={(14,0)}}]     
 \draw ([c]Q1)-- ( [c]Q2)--([c]Q3)--([c]Q4)--([c]Q1);   
 \end{scope}
 

 \coordinate (Am) at (0,0);
 \coordinate (Bm) at(0.2,-.3);
 \coordinate (Cm) at  ( .4,0);
 \coordinate (Dm) at (.6,.3)  ;
  \coordinate (Em) at (.8,0) ;
  \coordinate (Fm) at (1,-.3);
  \coordinate (Gm) at  (1.2,0);
  \coordinate (H1m) at (-.5,.3);
 \coordinate (H2m) at (-.2,.2);
\coordinate (J1m) at (1.4,.2);
 \coordinate (J2m) at (1.7,.3);
 \coordinate (N1m) at (3.2,-.3);
 \coordinate (N2m) at (.2,-.3);
  \coordinate (N3m) at (3,.3);

 \begin{scope}[every coordinate/.style={shift={(20,3)}}]
 
 \draw[thick]  plot [smooth, tension=3] coordinates {([c]H1m) ([c]H2m) ([c]Am)  };
 
\draw[thick] ([c]Am) .. controls ([c]Bm) ..
       ([c]Cm) .. controls ([c]Dm)  .. ([c]Em) .. controls ([c]Fm).. ([c]Gm);
    \draw[thick]  plot [smooth, tension=3] coordinates { ([c]Gm) ([c]J1m) ([c]J2m)  };   
     \draw[thick]   ([c]J2m).. controls ([c]N1m) and ([c]N2m).. ([c]N3m);
 \draw[dashed] ([c]P1)-- ( [c]P2)--([c]P3)--([c]P4)--([c]P1);  
    \end{scope}



 
  \coordinate (Je1) at (2.9,4.2);
  \coordinate (Je2) at (2.3,4.8);
  \coordinate (Y2) at (2.8,5);
   \coordinate (Y1) at (2.8,3);
 
 \begin{scope}[every coordinate/.style={shift={(3,-6)}}]

 \draw ([c]Q1)-- ( [c]Q2)--([c]Q3long)--([c]Q4long)--([c]Q1);   
 \end{scope}

 \begin{scope}[every coordinate/.style={shift={(3,-6)}}]
\draw[draw=gray] ([c]AE1) .. controls ([c]BE1) ..
       ([c]CE1);
     \draw[thick] ([c]EE2)  .. controls ([c]FE2) .. ([c]GE2);
     
      \draw[draw=gray]  plot [smooth, tension=3] coordinates { ([c]AE1) ([c]H1) ([c]H2)};
       \draw[draw=gray]  plot [smooth, tension=3] coordinates { ([c]CE1) ([c]HE1) ([c]HE2)};
       
        \draw[thick]  plot [smooth, tension=3] coordinates { ([c]EE2) ([c]DE2) ([c]JE1)};
       \draw[thick]  plot [smooth, tension=3] coordinates { ([c]GE2) ([c]J1) ([c]J2)};

       \draw[->,thick] ([c]Fs) to ([c]Ft);
        \node at ([c]Flabel) [above] {\scriptsize{$f$}};
             \node at ([c]J2) [left] {\scriptsize{$E_1$}};
         \node at ([c]H2) [xshift=.2cm,yshift=-.1cm] {\begin{color}{gray}\scriptsize{$E_2$}\end{color}};
         
          \draw[thick] ([c]T2)  -- ([c]T1);
         
                          \draw[->,dotted] ([c]fact1) to ([c]fact2);

                        \draw[->,thick] ([c]Fbs) to ([c]Fbt);         
 \node at ([c]Fblabel) [right] {\scriptsize{$\bar{f}$}};
\node at ([c]Qlongmiddle) [below] {\scriptsize{${}^2\pazocal E^{(2)}\cap\overline{\pazocal M}^{\text{main}}$}};
 \end{scope}
 
\begin{scope}[every coordinate/.style={shift={(4.7,-10)}}]

  \draw[thick] ([c]EE2)  .. controls ([c]FE2) .. ([c]GE2);
     
      \draw[thick]  plot [smooth, tension=3] coordinates { ([c]EE2) ([c]DE2)};
       \draw[thick]  plot [smooth, tension=3] coordinates { ([c]GE2) ([c]J1) ([c]J2)};

        \draw[thick] ([c]Je2)  .. controls ([c]Je1) .. ([c]DE2);
        \draw[thick] ([c]Y2) -- ([c]Y1);
       
\end{scope}

        \coordinate (f2n) at (-.2,-.6);
     \coordinate (f1n) at  (2.2,-.6);
     \coordinate (gn) at  (1,-.6);

 \begin{scope}[every coordinate/.style={shift={(9,-3)}}]

   \draw[thick]  ([c]f2n) -- ([c]f1n);
 \draw[dashed] ([c]P1)-- ( [c]P2)--([c]P3)--([c]P4)--([c]P1);  
    
     \draw[thick]([c]An) circle (0.7cm);
     \fill[gray] ([c]gn) circle (3pt);
     \fill[blue] ([c]gn)[xshift=-1cm] circle (3pt) ([c]gn)[xshift=.8cm] circle (3pt) ([c]gn)[xshift=-.5cm] circle (3pt) ([c]gn)[xshift=.3cm] circle (3pt);
          \node at ([c]gn) [below] {\begin{color}{gray}\scriptsize{$f(E_2)$}\end{color}};
          \node at ([c]gn) [yshift=.2cm] {\scriptsize{$A_3$}};

    \end{scope}




  \coordinate (Z2) at (2.6,4.7);
   \coordinate (Z1) at (4.1,5);
 
 \begin{scope}[every coordinate/.style={shift={(14,-6)}}]

 \draw ([c]Q1)-- ( [c]Q2)--([c]Q3long)--([c]Q4long)--([c]Q1);   
 \end{scope}

 \begin{scope}[every coordinate/.style={shift={(14,-6)}}]
\draw[draw=gray] ([c]AE1) .. controls ([c]BE1) ..
       ([c]CE1);
     \draw[thick] ([c]EE2)  .. controls ([c]FE2) .. ([c]GE2);
     
      \draw[draw=gray]  plot [smooth, tension=3] coordinates { ([c]AE1) ([c]H1) ([c]H2)};
       \draw[draw=gray]  plot [smooth, tension=3] coordinates { ([c]CE1) ([c]HE1) ([c]HE2)};
       
        \draw[thick]  plot [smooth, tension=3] coordinates { ([c]EE2) ([c]DE2) ([c]JE1)};
       \draw[thick]  plot [smooth, tension=3] coordinates { ([c]GE2) ([c]J1) ([c]J2)};

       \draw[->,thick] ([c]Fs) to ([c]Ft);
        \node at ([c]Flabel) [above] {\scriptsize{$f$}};
             \node at ([c]J2) [left] {\scriptsize{$E_1$}};
         \node at ([c]H2) [xshift=.2cm,yshift=-.1cm] {\begin{color}{gray}\scriptsize{$E_2$}\end{color}};
         
          \draw[thick] ([c]T2)  -- ([c]T1);
         
                          \draw[->,dotted] ([c]fact1) to ([c]fact2);

                        \draw[->,thick] ([c]Fbs) to ([c]Fbt);         
 \node at ([c]Fblabel) [right] {\scriptsize{$\bar{f}$}};
\node at ([c]Qlongmiddle) [below] {\scriptsize{${}^2\pazocal E^{(2)}\cap\overline{\pazocal M}^{\text{main}}$ (sprouting)}};
 \end{scope}
 
\begin{scope}[every coordinate/.style={shift={(15.7,-10)}}]

  \draw[thick] ([c]EE2)  .. controls ([c]FE2) .. ([c]GE2);
     
      \draw[thick]  plot [smooth, tension=3] coordinates { ([c]EE2) ([c]DE2)};
       \draw[thick]  plot [smooth, tension=3] coordinates { ([c]GE2) ([c]J1) ([c]J2)};

        \draw[thick] ([c]Je2)  .. controls ([c]Je1) .. ([c]DE2);
        \draw[dashed] ([c]Y2) -- ([c]Y1);
         \draw[thick] ([c]Z2) -- ([c]Z1);
       
\end{scope}

 \begin{scope}[every coordinate/.style={shift={(20,-3)}}]

   \draw[thick]  ([c]F2n) -- ([c]F1n);
 \draw[dashed] ([c]P1)-- ( [c]P2)--([c]P3)--([c]P4)--([c]P1);  
    
     \draw[thick]([c]An) circle (0.7cm);
     \fill[blue] ([c]Gn) circle (3pt);
     \fill[blue] ([c]Gn)[xshift=-1cm] circle (3pt) ([c]Gn)[xshift=-1.6cm] circle (3pt) ([c]Gn)[xshift=-.5cm] circle (3pt);
          \node at ([c]Gn) [xshift=.3cm,yshift=.2cm] {\begin{color}{gray}\scriptsize{$f(E_2)$}\end{color}};

    \end{scope}

\end{tikzpicture}
\caption{${}^2\pazocal E^{(2)},\dim=13$; ${}^2\pazocal E^{(2)}\cap\overline{\pazocal M}^{\text{main}},\dim=12$}
\label{2E2}	
\end{figure}
 
 \item[${}^2\!\pazocal{E}^{(1^2)}$ component] 

 This is the locus generically parametrising maps $f\colon E_1\cup E_2\cup T_1\cup T_2\to\mathbb P^2$ where $E_1\cup E_2$ is, as in the example above, a reducible core given by two nodally attached elliptic curves, with $f$ contracting $E_2$ and with degree $2$ on $E_1$, and $T_i$ rational tails attached to $E_2$ on which the map has degree $1.$
  This locus has dimension $12$ and
 it is  actually contained in \emph{main}, since the image determines the elliptic $3$-fold point through which the map factors.

 \item[$\pazocal{E}^4\!\pazocal{E}$ component] The general image is a three-nodal quartic, and the two elliptic curves are contracted anywhere. The dimension is $15$. The intersection with main one component:
 a map $f$ whose weighted dual graph is generic for this component is smoothable if and only if the elliptic curves contracted to cusps, and the image is of type $A_1-A_2^2$.  Note that $\pazocal{V\!Z}_2(\PP^2,4)$ has general fibre $\PP^1$ over this locus, corresponding to the comparison of the height (i.e. the values taken by $\lambda$) of the two elliptic curves.  The factorisation property does not depend on it. See Figure \ref{E4E}.
The dimension is $11.$

\begin{figure}[h]
   \begin{tikzpicture}[scale=.4]

\coordinate (Hh1) at (2,4);
 \coordinate (Hh2) at (2.1,4.5);
 
 \coordinate (AE1) at (1.7,3.7);
 \coordinate (BE1) at (1.4,3.5);
 \coordinate (CE1) at   (1.7,3.3);
 \coordinate (HE1) at (2,3);
 \coordinate (HE2) at (2.1,2.5);

\coordinate (Jj1) at (2.5,4);
 \coordinate (Jj2) at (2.6,4.5);
 
 \coordinate (AE2) at (2.9,3.7);
 \coordinate (BE2) at (3.2,3.5);
 \coordinate (CE2) at   (2.9,3.3);
 \coordinate (JE1) at (2.6,3);
 \coordinate (JE2) at (2.5,2.5);

 \coordinate (T1) at (1.2,3);
 \coordinate (T2) at (3,3);

   \coordinate(fact1) at (2.4,1.2);
   \coordinate(fact2) at (3.1,-.2);
   
 \coordinate (Fs) at (3.1,3.2);
  \coordinate (Ft) at (4.5,3.2);
   \coordinate (Flabel) at (3.8,3.5);
   
    \coordinate (Fbs) at (5.3,-.3);
  \coordinate (Fbt) at (6.4,1);
    \coordinate (Fblabel) at (5.9,.1);

   \begin{scope}[every coordinate/.style={shift={(0,0)}}]
\draw[draw=gray] ([c]AE1) .. controls ([c]BE1) ..
       ([c]CE1);
       \draw[draw=gray]  plot [smooth, tension=3] coordinates { ([c]AE1) ([c]Hh1) ([c]Hh2)};
       \draw[draw=gray]  plot [smooth, tension=3] coordinates { ([c]CE1) ([c]HE1) ([c]HE2)};

       \draw[draw=gray] ([c]AE2) .. controls ([c]BE2) ..
       ([c]CE2);
       \draw[draw=gray]  plot [smooth, tension=3] coordinates { ([c]AE2) ([c]Jj1) ([c]Jj2)};
       \draw[draw=gray]  plot [smooth, tension=3] coordinates { ([c]CE2) ([c]JE1) ([c]JE2)};

        \draw[thick] ([c]T2) -- node[yshift=-.2cm]{\scriptsize{$4$}}([c]T1);
       
       \draw[->,thick] ([c]Fs) to ([c]Ft);
        \node at ([c]Flabel) [above] {\scriptsize{$f$}};
             \node at ([c]AE1) [xshift=-.1cm,yshift=.4cm] {\begin{color}{gray}\scriptsize{$E_1$}\end{color}};
         \node at ([c]AE2) [xshift=.1cm,yshift=.4cm] {\begin{color}{gray}\scriptsize{$E_2$}\end{color}};

 \end{scope}

 \coordinate(Q1) at  (.5,5.5);
 \coordinate (Q2) at (10,5.5);
 \coordinate (Q3) at (10,1);
 \coordinate (Q4) at (.5,1);
 \coordinate (Qmiddle) at (5.25,1);
 \coordinate (Q3long) at (10,-2.5 );
 \coordinate (Q4long) at (.5,-2.5);
\coordinate (Qlongmiddle) at (5.5,-2.5);
 
 \coordinate(P1) at  (-0.2,1);
 \coordinate (P2) at ( 3.5,1);
 \coordinate (P3) at (2.5,-1 );
 \coordinate (P4) at (-1.2,-1);

  \begin{scope}[every coordinate/.style={shift={(0,0)}}]     
 \draw (Q1)-- ( Q2)--(Q3)--(Q4)--(Q1);   
 \end{scope}
 

 \coordinate (An) at (-.5,.5);
 \coordinate (B1m) at(2,-.2);
  \coordinate (B2m) at(-1.8,-.2);
 \coordinate (Cn) at  (1,.5);

 \coordinate (D1n) at (3,-.2)  ;
  \coordinate (D2n) at (-.5,-.2) ;
  \coordinate (En) at (2,.5);
  \coordinate (F1n) at  (4,-.2);
   \coordinate (F2n) at (1,-.2);
   \coordinate (Gn) at (3,.5);

 \coordinate(P1) at  (-0.2,1);
 \coordinate (P2) at ( 3.5,1);
 \coordinate (P3) at (2.5,-1 );
 \coordinate (P4) at (-1.2,-1);
 
 \begin{scope}[every coordinate/.style={shift={(6,3)}}]

\draw[thick] ([c]An) .. controls ([c]B1m) and ([c]B2m) ..
       ([c]Cn) .. controls ([c]D1n) and ([c]D2n)  .. ([c]En) .. controls ([c]F1n) and ([c]F2n)  .. ([c]Gn);
   
 \draw[dashed] ([c]P1)-- ( [c]P2)--([c]P3)--([c]P4)--([c]P1);  
    
     \fill[gray] ([c]Cn) circle (3pt);
          \node at ([c]Cn) [yshift=.15cm,xshift=-.2cm] {\begin{color}{gray}\scriptsize{$f(E_1)$}\end{color}};
           \fill[gray] ([c]En) circle (3pt);
          \node at ([c]En) [yshift=.15cm,xshift=.2cm] {\begin{color}{gray}\scriptsize{$f(E_2)$}\end{color}};

    \end{scope}   

     \node at (5,1) [below] {\scriptsize{$\pazocal E^4\pazocal E$}};


    \coordinate (A) at (2,4);
 \coordinate (B) at (1.7,3.8);
 \coordinate (C) at   (2,3.6);
 \coordinate (D) at (2.3,3.4) ;
  \coordinate (E) at (2,3.2) ;
  \coordinate (F) at (1.7,3);
  \coordinate (G) at  (2,2.8);
  \coordinate (H1) at (2.3,4.3);
 \coordinate (H2) at (2.4,4.8);
  \coordinate (J1) at (2.3,2.4);
 \coordinate (J2) at (2.4,2);

\begin{scope}[every coordinate/.style={shift={(14,0)}}]
\draw[thick] ([c]A) .. controls ([c]B) ..
       ([c]C) .. controls ([c]D)  .. ([c]E) .. controls ([c]F).. ([c]G);
      \draw[thick]  plot [smooth, tension=3] coordinates { ([c]A) ([c]H1) ([c]H2)};
       \draw[thick]  plot [smooth, tension=3] coordinates { ([c]G) ([c]J1) ([c]J2)};
       \draw[->,thick] ([c]Fs) to ([c]Ft);
        \node at ([c]Flabel) [above] {\scriptsize{$f$}};
       \end{scope}

  \begin{scope}[every coordinate/.style={shift={(14,0)}}]     
 \draw ([c]Q1)-- ( [c]Q2)--([c]Q3)--([c]Q4)--([c]Q1);   
 \end{scope}
 

 \coordinate (Am) at (0,0);
 \coordinate (Bm) at(0.2,-.3);
 \coordinate (Cm) at  ( .4,0);
 \coordinate (Dm) at (.6,.3)  ;
  \coordinate (Em) at (.8,0) ;
  \coordinate (Fm) at (1,-.3);
  \coordinate (Gm) at  (1.2,0);
  \coordinate (H1m) at (-.5,.3);
 \coordinate (H2m) at (-.2,.2);
\coordinate (J1m) at (1.4,.2);
 \coordinate (J2m) at (1.7,.3);
 \coordinate (N1m) at (3.2,-.3);
 \coordinate (N2m) at (.2,-.3);
  \coordinate (N3m) at (3,.3);

 \begin{scope}[every coordinate/.style={shift={(20,3)}}]
 
 \draw[thick]  plot [smooth, tension=3] coordinates {([c]H1m) ([c]H2m) ([c]Am)  };
 
\draw[thick] ([c]Am) .. controls ([c]Bm) ..
       ([c]Cm) .. controls ([c]Dm)  .. ([c]Em) .. controls ([c]Fm).. ([c]Gm);
    \draw[thick]  plot [smooth, tension=3] coordinates { ([c]Gm) ([c]J1m) ([c]J2m)  };   
     \draw[thick]   ([c]J2m).. controls ([c]N1m) and ([c]N2m).. ([c]N3m);
 \draw[dashed] ([c]P1)-- ( [c]P2)--([c]P3)--([c]P4)--([c]P1);  
    \end{scope}

\draw (7.5,1)--(11,-.5);
\draw (19.5,1)--(16.,-.5);

 \node at (12.3,3) [above] { $\cap$};

  
 \begin{scope}[every coordinate/.style={shift={(8,-6)}}]

 \draw ([c]Q1)-- ( [c]Q2)--([c]Q3long)--([c]Q4long)--([c]Q1);   
 \end{scope}

 \begin{scope}[every coordinate/.style={shift={(8,-6)}}]

\draw[draw=gray] ([c]AE1) .. controls ([c]BE1) ..
       ([c]CE1);
       \draw[draw=gray]  plot [smooth, tension=3] coordinates { ([c]AE1) ([c]Hh1) ([c]Hh2)};
       \draw[draw=gray]  plot [smooth, tension=3] coordinates { ([c]CE1) ([c]HE1) ([c]HE2)};

       \draw[draw=gray] ([c]AE2) .. controls ([c]BE2) ..
       ([c]CE2);
       \draw[draw=gray]  plot [smooth, tension=3] coordinates { ([c]AE2) ([c]Jj1) ([c]Jj2)};
       \draw[draw=gray]  plot [smooth, tension=3] coordinates { ([c]CE2) ([c]JE1) ([c]JE2)};

        \draw[thick] ([c]T2) -- ([c]T1);

        \draw[thick] ([c]T2) -- ([c]T1);
       
       \draw[->,thick] ([c]Fs) to ([c]Ft);
        \node at ([c]Flabel) [above] {\scriptsize{$f$}};
       
\draw[->,dotted] ([c]fact1) to ([c]fact2);
                          
     \draw[->,thick] ([c]Fbs) to ([c]Fbt);         
 \node at ([c]Fblabel) [right] {\scriptsize{$\bar{f}$}};
\node at ([c]Qlongmiddle) [below] {\scriptsize{$\pazocal E^4\pazocal E\cap\overline{\pazocal M}^{\text{main}}$}};
 \end{scope}
 
 
 \coordinate (DS) at (2.3,3.35) ;

\begin{scope}[every coordinate/.style={shift={(10.5,-9.5)}}]

\draw[thick] ([c]B) parabola ([c]H2);
   \draw[thick]   ([c]B) parabola  ([c]D);
      \draw[thick] ([c]F)  parabola ([c]DS);
        \draw[thick] ([c]F)  parabola ([c]J2);
    \draw[thick]  plot [smooth, tension=3] coordinates { ([c]DS) ([c]D)};
   
\end{scope}


  \coordinate (GSm) at  (1.2,.2);
  
\coordinate (J1Sm) at (1.4,.4);
 \coordinate (J2Sm) at (1.7,.5);
 
  \coordinate (N3Sm) at (2.7,.3);

 \coordinate (DSm) at( 0.65,.3);

 \begin{scope}[every coordinate/.style={shift={(14,-3)}}]
 
 \draw[thick] ([c]H1m) parabola ([c]Bm);
  \draw[thick] ([c]Dm) parabola ([c]Bm);
  
 \draw[thick]  plot [smooth, tension=3] coordinates {([c]Dm)  ([c]DSm)  };
 \node at ([c]DSm) [xshift=-.3cm,yshift=-.4cm] {\begin{color}{gray}\scriptsize{$f(E_1)$}\end{color}};
 \node at ([c]DSm) [xshift=.4cm,yshift=-.4cm] {\begin{color}{gray}\scriptsize{$f(E_2)$}\end{color}};
  \draw[thick] ([c]DSm) parabola ([c]Fm);
   \draw[thick] ([c]GSm) parabola ([c]Fm);
 
    \draw[thick]  plot [smooth, tension=3] coordinates { ([c]GSm) ([c]J1Sm) ([c]J2Sm)  };   
     \draw[thick]   ([c]J2Sm).. controls ([c]N1m) and ([c]N2m).. ([c]N3Sm);
 \draw[dashed] ([c]P1)-- ( [c]P2)--([c]P3)--([c]P4)--([c]P1);  
    \end{scope}

\end{tikzpicture}
\caption{$\pazocal E^4\!\pazocal E,\dim=15$; $\pazocal E^4\!\pazocal E\cap\overline{\pazocal M}^{\text{main}},\dim=11$}
\label{E4E}	
\end{figure}

 \item[$\pazocal{E}^3\!\pazocal{E}^{(1)}$ component] In general, the image is a one-nodal cubic with a general line. The dimension is $14$. A map is smoothable if its image consists of a cuspidal cubic with a tangent line, type $A_1-A_2-A_3$. The dimension is $11$. See Figure \ref{E3E1}.

\begin{figure}[h]
   \begin{tikzpicture}[scale=.4]

\coordinate (Hh1) at (2,4);
 \coordinate (Hh2) at (2.1,4.5);
 
 \coordinate (AE1) at (1.7,3.7);
 \coordinate (BE1) at (1.4,3.5);
 \coordinate (CE1) at   (1.7,3.3);
 \coordinate (HE1) at (2,3);
 \coordinate (HE2) at (2.1,2.5);

\coordinate (Jj1) at (2.5,4);
 \coordinate (Jj2) at (2.6,4.5);
 
 \coordinate (AE2) at (2.9,3.7);
 \coordinate (BE2) at (3.2,3.5);
 \coordinate (CE2) at   (2.9,3.3);
 \coordinate (JE1) at (2.6,3);
 \coordinate (JE2) at (2.5,2.5);

 \coordinate (T1) at (1.7,3);
 \coordinate (T2) at (3,3);
 \coordinate (R1) at (2.6,3.5);
 \coordinate (R2) at (3.6,4.5);

   \coordinate(fact1) at (2.4,1.2);
   \coordinate(fact2) at (3.1,-.2);
   
 \coordinate (Fs) at (3.1,3.2);
  \coordinate (Ft) at (4.5,3.2);
   \coordinate (Flabel) at (3.8,3.5);
   
    \coordinate (Fbs) at (5.3,-.3);
  \coordinate (Fbt) at (6.4,1);
    \coordinate (Fblabel) at (5.9,.1);

   \begin{scope}[every coordinate/.style={shift={(0,0)}}]
\draw[draw=gray] ([c]AE1) .. controls ([c]BE1) .. ([c]CE1);
       \draw[draw=gray]  plot [smooth, tension=3] coordinates { ([c]AE1) ([c]Hh1) ([c]Hh2)};
       \draw[draw=gray]  plot [smooth, tension=3] coordinates { ([c]CE1) ([c]HE1) ([c]HE2)};

       \draw[draw=gray] ([c]AE2) .. controls ([c]BE2) ..
       ([c]CE2);
       \draw[draw=gray]  plot [smooth, tension=3] coordinates { ([c]AE2) ([c]Jj1) ([c]Jj2)};
       \draw[draw=gray]  plot [smooth, tension=3] coordinates { ([c]CE2) ([c]JE1) ([c]JE2)};

        \draw[thick] ([c]T2) -- ([c]T1);
        \draw[thick] ([c]R1) -- ([c]R2);
       
       \draw[->,thick] ([c]Fs) to ([c]Ft);
        \node at ([c]Flabel) [above] {\scriptsize{$f$}};
           \node at ([c]CE1) [xshift=-.1cm,yshift=.4cm] {\begin{color}{gray}\scriptsize{$E_1$}\end{color}};
         \node at ([c]JE2) [xshift=.2cm] {\begin{color}{gray}\scriptsize{$E_2$}\end{color}};

  \node at ([c]T1) [below] {\scriptsize{$3$}};
    \node at ([c]R2) [left] {\scriptsize{$1$}};

 \end{scope}

 \coordinate(Q1) at  (1,5);
 \coordinate (Q2) at (10,5);
 \coordinate (Q3) at (10,1 );
 \coordinate (Q4) at (1,1);
 \coordinate (Q3long) at (10,-2.5 );
 \coordinate (Q4long) at (1,-2.5);
\coordinate (Qlongmiddle) at (5.5,-2.5);

 \coordinate(P1) at  (-0.2,1);
 \coordinate (P2) at ( 3.5,1);
 \coordinate (P3) at (2.5,-1 );
 \coordinate (P4) at (-1.2,-1);

  \begin{scope}[every coordinate/.style={shift={(0,0)}}]     
 \draw (Q1)-- ( Q2)--(Q3)--(Q4)--(Q1);   
 \end{scope}
 

  \coordinate (G1) at (0,.5);
 \coordinate (G2) at ( 4,-.6);
  \coordinate (G3) at (-1.5,-.6);
 \coordinate (G4) at  (1.8,.5);
 \coordinate (G5) at (-.5,-.1) ;
  \coordinate (G6) at (2.5,-.1) ;
   \coordinate (G7) at (.7,-.1) ;
  \coordinate (G8) at  (2.3,.5);
  \coordinate (G9) at  (2.1,-.4);
  
   \begin{scope}[every coordinate/.style={shift={(6,3)}}]
\draw[thick] ([c]G1) .. controls ([c]G2) and ([c]G3) ..
       ([c]G4) .. controls ([c]G8).. ([c]G9);
       \draw[thick] ([c]G5)--([c]G6);
   
 \draw[dashed] ([c]P1)-- ( [c]P2)--([c]P3)--([c]P4)--([c]P1);  
    
     \fill[gray] ([c]G7) circle (3pt);
          \node at ([c]G7) [below] {\begin{color}{gray}\scriptsize{$f(E_2)$}\end{color}};
\fill[gray] (8,3.5) circle (3pt);
          \node at (8,3.5) [yshift=.2cm] {\begin{color}{gray}\scriptsize{$f(E_1)$}\end{color}};
    \end{scope}

 \node at (5,1) [below] {\scriptsize{$\pazocal E^3\pazocal E^{(1)}$}};


    \coordinate (A) at (2,4);
 \coordinate (B) at (1.7,3.8);
 \coordinate (C) at   (2,3.6);
 \coordinate (D) at (2.3,3.4) ;
  \coordinate (E) at (2,3.2) ;
  \coordinate (F) at (1.7,3);
  \coordinate (G) at  (2,2.8);
  \coordinate (H1) at (2.3,4.3);
 \coordinate (H2) at (2.4,4.8);
  \coordinate (J1) at (2.3,2.4);
 \coordinate (J2) at (2.4,2);

\begin{scope}[every coordinate/.style={shift={(14,0)}}]
\draw[thick] ([c]A) .. controls ([c]B) ..
       ([c]C) .. controls ([c]D)  .. ([c]E) .. controls ([c]F).. ([c]G);
      \draw[thick]  plot [smooth, tension=3] coordinates { ([c]A) ([c]H1) ([c]H2)};
       \draw[thick]  plot [smooth, tension=3] coordinates { ([c]G) ([c]J1) ([c]J2)};
       \draw[->,thick] ([c]Fs) to ([c]Ft);
        \node at ([c]Flabel) [above] {\scriptsize{$f$}};
       \end{scope}

  \begin{scope}[every coordinate/.style={shift={(14,0)}}]     
 \draw ([c]Q1)-- ( [c]Q2)--([c]Q3)--([c]Q4)--([c]Q1);   
 \end{scope}
 

 \coordinate (Am) at (0,0);
 \coordinate (Bm) at(0.2,-.3);
 \coordinate (Cm) at  ( .4,0);
 \coordinate (Dm) at (.6,.3)  ;
  \coordinate (Em) at (.8,0) ;
  \coordinate (Fm) at (1,-.3);
  \coordinate (Gm) at  (1.2,0);
  \coordinate (H1m) at (-.5,.3);
 \coordinate (H2m) at (-.2,.2);
\coordinate (J1m) at (1.4,.2);
 \coordinate (J2m) at (1.7,.3);
 \coordinate (N1m) at (3.2,-.3);
 \coordinate (N2m) at (.2,-.3);
  \coordinate (N3m) at (3,.3);

 \begin{scope}[every coordinate/.style={shift={(20,3)}}]
 
 \draw[thick]  plot [smooth, tension=3] coordinates {([c]H1m) ([c]H2m) ([c]Am)  };
 
\draw[thick] ([c]Am) .. controls ([c]Bm) ..
       ([c]Cm) .. controls ([c]Dm)  .. ([c]Em) .. controls ([c]Fm).. ([c]Gm);
    \draw[thick]  plot [smooth, tension=3] coordinates { ([c]Gm) ([c]J1m) ([c]J2m)  };   
     \draw[thick]   ([c]J2m).. controls ([c]N1m) and ([c]N2m).. ([c]N3m);
 \draw[dashed] ([c]P1)-- ( [c]P2)--([c]P3)--([c]P4)--([c]P1);  
    \end{scope}

\draw (7.5,1)--(13.,-1);
\draw (19.5,1)--(16.,-1);

 \node at (12.3,3) [above] { $\cap$};


\begin{scope}[every coordinate/.style={shift={(8,-6)}}]

 \draw ([c]Q1)-- ( [c]Q2)--([c]Q3long)--([c]Q4long)--([c]Q1);   
 \end{scope}

\begin{scope}[every coordinate/.style={shift={(8,-6)}}]

\draw[draw=gray] ([c]AE1) .. controls ([c]BE1) ..
       ([c]CE1);
       \draw[draw=gray]  plot [smooth, tension=3] coordinates { ([c]AE1) ([c]Hh1) ([c]Hh2)};
       \draw[draw=gray]  plot [smooth, tension=3] coordinates { ([c]CE1) ([c]HE1) ([c]HE2)};

       \draw[draw=gray] ([c]AE2) .. controls ([c]BE2) ..
       ([c]CE2);
       \draw[draw=gray]  plot [smooth, tension=3] coordinates { ([c]AE2) ([c]Jj1) ([c]Jj2)};
       \draw[draw=gray]  plot [smooth, tension=3] coordinates { ([c]CE2) ([c]JE1) ([c]JE2)};

        \draw[thick] ([c]T2) -- ([c]T1);

        \draw[thick] ([c]T2) -- ([c]T1);
       
       \draw[->,thick] ([c]Fs) to ([c]Ft);
        \node at ([c]Flabel) [above] {\scriptsize{$f$}};
       
\draw[->,dotted] ([c]fact1) to ([c]fact2);
                          
     \draw[->,thick] ([c]Fbs) to ([c]Fbt);         
 \node at ([c]Fblabel) [right] {\scriptsize{$\bar{f}$}};
\node at ([c]Qlongmiddle) [below] {\scriptsize{$\pazocal E^3\pazocal E^{(1)}\cap\overline{\pazocal M}^{\text{main}}$}};
 \end{scope}
 
 
 \coordinate (DS) at (2.3,3.35) ;
 
 \coordinate (ES) at (2.6,2.8) ;
 \coordinate (FS) at (1.5,2) ;
 \coordinate (S1) at (2.52,4) ;
 \coordinate (S2) at (2.52,1.8) ;

\begin{scope}[every coordinate/.style={shift={(10.1,-9.5)}}]

\draw[thick] ([c]B) parabola ([c]H2);
   \draw[thick]   ([c]B) parabola  ([c]DS);
      \draw[thick] ([c]DS) .. controls ([c]ES)..([c]FS);

\draw[thick] ([c]S1) -- ([c]S2);
   
\end{scope}

   
    \coordinate (BI) at(-.3,-.3);
    \coordinate (CI) at(.8,.7);
    \coordinate (DI) at(1.6,0);
      \coordinate (EI) at(1.8,0);
       \coordinate (FI) at(2.4,.7);
 \coordinate (L1) at(-.5,0);
 \coordinate (L2) at(2.5,0);
 \begin{scope}[every coordinate/.style={shift={(14,-3)}}]
 
 \draw[thick] ([c]BI) parabola ([c]CI);
  \draw[thick] ([c]DI) parabola ([c]CI);
  
 \draw[thick]  plot [smooth, tension=3] coordinates {([c]DI)  ([c]EI)  };
 
\draw[thick] ([c]EI) parabola ([c]FI);
   \draw[thick] ([c]L1) -- ([c]L2);

 \draw[dashed] ([c]P1)-- ( [c]P2)--([c]P3)--([c]P4)--([c]P1);  
 
      \fill[gray] ([c]CI) circle (3pt);
           \fill[gray] ([c]EI) circle (3pt);
 
   \node at ([c]CI) [left] {\begin{color}{gray}\scriptsize{$f(E_1)$}\end{color}};
         \node at ([c]EI) [below] {\begin{color}{gray}\scriptsize{$f(E_2)$}\end{color}};

\end{scope}

\end{tikzpicture}
\caption{$\pazocal E^3\!\pazocal E^{(1)},\dim=14$; $\pazocal E^3\!\pazocal E^{(1)}\cap\overline{\pazocal M}^{\text{main}},\dim=11$}
\label{E3E1}	
\end{figure}

 \item[$\pazocal{E}^2\!\pazocal{E}^{(2)}$ component] 
 The generic point of this component correspond to maps from a source curve $C=E_1\cup B\cup E_2\cup T$ where $E_1$ and $E_2$ are two elliptic curve contracted by the map and separated by a rational bridge $B$ on which the map has degree $2$, and $T$ is a rational tail attached to $E_2.$

 The image consists of two conics. The dimension is $14$. The intersection with \emph{main} consists of two components, both of dimension $11$: in both cases $E_1$ (which is the elliptic curve without rational tail attached) is replaced by a cusp, forcing the separating bridge to doubly cover a line, with one ramification point situated at the attaching of $E_1$. Then (see Figure \ref{E2E2}):
 \begin{itemize}[leftmargin=.5cm]
  \item either $E_2$ is replaced by a tacnode, in which case the second conic has to be tangent to the line;
  \item or there is a sprouting and the conic has to pass through the second ramification point of the two-to-one cover.
 \end{itemize}

\begin{figure}[h]
   \begin{tikzpicture}[scale=.4]

\coordinate (Hh1) at (2,4);
 \coordinate (Hh2) at (2.1,4.5);
 
 \coordinate (AE1) at (1.7,3.7);
 \coordinate (BE1) at (1.4,3.5);
 \coordinate (CE1) at   (1.7,3.3);
 \coordinate (HE1) at (2,3);
 \coordinate (HE2) at (2.1,2.5);

\coordinate (Jj1) at (2.5,4);
 \coordinate (Jj2) at (2.6,4.5);
 
 \coordinate (AE2) at (2.9,3.7);
 \coordinate (BE2) at (3.2,3.5);
 \coordinate (CE2) at   (2.9,3.3);
 \coordinate (JE1) at (2.6,3);
 \coordinate (JE2) at (2.5,2.5);

 \coordinate (T1) at (1.7,3);
 \coordinate (T2) at (3,3);
 \coordinate (R1) at (2.6,3.5);
 \coordinate (R2) at (3.6,4.5);

   \coordinate(fact1) at (2.4,1.2);
   \coordinate(fact2) at (3.1,-.2);
   
 \coordinate (Fs) at (3.1,3.2);
  \coordinate (Ft) at (4.5,3.2);
   \coordinate (Flabel) at (3.8,3.5);
   
    \coordinate (Fbs) at (5.3,-.3);
  \coordinate (Fbt) at (6.4,1);
    \coordinate (Fblabel) at (5.9,.1);

   \begin{scope}[every coordinate/.style={shift={(0,0)}}]
\draw[draw=gray] ([c]AE1) .. controls ([c]BE1) ..
       ([c]CE1);
       \draw[draw=gray]  plot [smooth, tension=3] coordinates { ([c]AE1) ([c]Hh1) ([c]Hh2)};
       \draw[draw=gray]  plot [smooth, tension=3] coordinates { ([c]CE1) ([c]HE1) ([c]HE2)};

       \draw[draw=gray] ([c]AE2) .. controls ([c]BE2) ..
       ([c]CE2);
       \draw[draw=gray]  plot [smooth, tension=3] coordinates { ([c]AE2) ([c]Jj1) ([c]Jj2)};
       \draw[draw=gray]  plot [smooth, tension=3] coordinates { ([c]CE2) ([c]JE1) ([c]JE2)};

        \draw[thick] ([c]T2) -- ([c]T1);
        \draw[thick] ([c]R1) -- ([c]R2);
       
       \draw[->,thick] ([c]Fs) to ([c]Ft);
        \node at ([c]Flabel) [above] {\scriptsize{$f$}};
           \node at ([c]CE1) [xshift=-.1cm,yshift=.4cm] {\begin{color}{gray}\scriptsize{$E_1$}\end{color}};
         \node at ([c]JE2) [xshift=.2cm] {\begin{color}{gray}\scriptsize{$E_2$}\end{color}};
  \node at ([c]T1) [below] {\scriptsize{$2$}};
    \node at ([c]R2) [left] {\scriptsize{$2$}};

 \end{scope}

 \coordinate(Q1) at  (1,5);
 \coordinate (Q2) at (10,5);
 \coordinate (Q3) at (10,1 );
 \coordinate (Q4) at (1,1);
 \coordinate (Q3long) at (10,-2.5 );
 \coordinate (Q4long) at (1,-2.5);
\coordinate (Qlongmiddle) at (5.5,-2.5);

 \coordinate(P1) at  (-0.2,1);
 \coordinate (P2) at ( 3.5,1);
 \coordinate (P3) at (2.5,-1 );
 \coordinate (P4) at (-1.2,-1);

  \begin{scope}[every coordinate/.style={shift={(0,0)}}]     
 \draw (Q1)-- ( Q2)--(Q3)--(Q4)--(Q1);   
 \end{scope}
 
  
  \coordinate (G1) at (1,0);
 \coordinate (G2) at (1.4,.45);
  
   \begin{scope}[every coordinate/.style={shift={(6,3)}}]
\draw[thick] ([c]G1)  ellipse (1.3cm and 0.5cm);
      \draw[thick] ([c]G1)  ellipse (.5cm and .9cm);

 \draw[dashed] ([c]P1)-- ( [c]P2)--([c]P3)--([c]P4)--([c]P1);  
    
     \fill[gray] ([c]G2) circle (3pt);
          \node at ([c]G2) [right] {\begin{color}{gray}\scriptsize{$f(E_2)$}\end{color}};
    \fill[gray] ([c]G2)[xshift=-.55cm,yshift=-1.3cm] circle (3pt);
      \node at ([c]G2)[xshift=-.5cm,yshift=-.8cm] {\begin{color}{gray}\scriptsize{$f(E_1)$}\end{color}};
    \end{scope}   
  
\node at (5,1) [below] {\scriptsize{$\pazocal E^2\pazocal E^{(2)}$}};

\begin{scope}[every coordinate/.style={shift={(14,0)}}]
\draw[thick] ([c]A) .. controls ([c]B) ..
       ([c]C) .. controls ([c]D)  .. ([c]E) .. controls ([c]F).. ([c]G);
      \draw[thick]  plot [smooth, tension=3] coordinates { ([c]A) ([c]H1) ([c]H2)};
       \draw[thick]  plot [smooth, tension=3] coordinates { ([c]G) ([c]J1) ([c]J2)};
       \draw[->,thick] ([c]Fs) to ([c]Ft);
        \node at ([c]Flabel) [above] {\scriptsize{$f$}};
       \end{scope}
 \coordinate(Q1) at  (1,5);
 \coordinate (Q2) at (10,5);
 \coordinate (Q3) at (10,1 );
 \coordinate (Q4) at (1,1);

  \begin{scope}[every coordinate/.style={shift={(14,0)}}]     
 \draw ([c]Q1)-- ( [c]Q2)--([c]Q3)--([c]Q4)--([c]Q1);   
 \end{scope}

 \coordinate (Am) at (0,0);
 \coordinate (Bm) at(0.2,-.3);
 \coordinate (Cm) at  ( .4,0);
 \coordinate (Dm) at (.6,.3)  ;
  \coordinate (Em) at (.8,0) ;
  \coordinate (Fm) at (1,-.3);
  \coordinate (Gm) at  (1.2,0);
  \coordinate (H1m) at (-.5,.3);
 \coordinate (H2m) at (-.2,.2);
\coordinate (J1m) at (1.4,.2);
 \coordinate (J2m) at (1.7,.3);
 \coordinate (N1m) at (3.2,-.3);
 \coordinate (N2m) at (.2,-.3);
  \coordinate (N3m) at (3,.3);
 \coordinate(P1) at  (-0.2,1);
 \coordinate (P2) at ( 3.5,1);
 \coordinate (P3) at (2.5,-1 );
 \coordinate (P4) at (-1.2,-1);
 
 \begin{scope}[every coordinate/.style={shift={(20,3)}}]
 \draw[thick]  plot [smooth, tension=3] coordinates {([c]H1m) ([c]H2m) ([c]Am)  };
\draw[thick] ([c]Am) .. controls ([c]Bm) ..
       ([c]Cm) .. controls ([c]Dm)  .. ([c]Em) .. controls ([c]Fm).. ([c]Gm);
    \draw[thick]  plot [smooth, tension=3] coordinates { ([c]Gm) ([c]J1m) ([c]J2m)  };   
     \draw[thick]   ([c]J2m).. controls ([c]N1m) and ([c]N2m).. ([c]N3m);
 \draw[dashed] ([c]P1)-- ( [c]P2)--([c]P3)--([c]P4)--([c]P1);  
    \end{scope}   
       
       
\draw (7.5,1)--(8.5,-1);
\draw (7.5,1)--(19,-1);

 \node at (12.3,3) [above] { $\cap$};


\begin{scope}[every coordinate/.style={shift={(3,-6)}}]     
 \draw ([c]Q1)-- ( [c]Q2)--([c]Q3long)--([c]Q4long)--([c]Q1);   
 \end{scope}

\begin{scope}[every coordinate/.style={shift={(3,-6)}}]

\draw[draw=gray] ([c]AE1) .. controls ([c]BE1) ..
       ([c]CE1);
       \draw[draw=gray]  plot [smooth, tension=3] coordinates { ([c]AE1) ([c]Hh1) ([c]Hh2)};
       \draw[draw=gray]  plot [smooth, tension=3] coordinates { ([c]CE1) ([c]HE1) ([c]HE2)};

       \draw[draw=gray] ([c]AE2) .. controls ([c]BE2) ..
       ([c]CE2);
       \draw[draw=gray]  plot [smooth, tension=3] coordinates { ([c]AE2) ([c]Jj1) ([c]Jj2)};
       \draw[draw=gray]  plot [smooth, tension=3] coordinates { ([c]CE2) ([c]JE1) ([c]JE2)};

        \draw[thick] ([c]T2) -- ([c]T1);

         \draw[thick] ([c]R1) -- ([c]R2);
       
       \draw[->,thick] ([c]Fs) to ([c]Ft);
        \node at ([c]Flabel) [above] {\scriptsize{$f$}};
       
\draw[->,dotted] ([c]fact1) to ([c]fact2);
                          
     \draw[->,thick] ([c]Fbs) to ([c]Fbt);         
 \node at ([c]Fblabel) [right] {\scriptsize{$\bar{f}$}};
\node at ([c]Qlongmiddle) [below] {\scriptsize{$\pazocal E^2\pazocal E^{(2)}\cap\overline{\pazocal M}^{\text{main}}$}};
 \end{scope}
 
 
 \coordinate (DS) at (2.3,3.35) ;
 
 \coordinate (ES) at (2.6,2.8) ;
 \coordinate (FS) at (1.5,2) ;
 \coordinate (S1) at (2.52,4) ;
 \coordinate (S2) at (2.52,1.8) ;

\begin{scope}[every coordinate/.style={shift={(5.1,-9.5)}}]

\draw[thick] ([c]B) parabola ([c]H2);
   \draw[thick]   ([c]B) parabola  ([c]DS);
      \draw[thick] ([c]DS) .. controls ([c]ES)..([c]FS);

\draw[thick] ([c]S1) -- ([c]S2);
   
\end{scope}


  \coordinate (GS1) at (-.5,-.5);
 \coordinate (GS2) at (2.4,.-.5);
  
   \begin{scope}[every coordinate/.style={shift={(9,-3)}}]
\draw[thick] ([c]G1)  ellipse (1cm and 0.5cm);
      \draw[thick] ([c]GS1) -- ([c]GS2);

 \draw[dashed] ([c]P1)-- ( [c]P2)--([c]P3)--([c]P4)--([c]P1);  
  
  \fill[gray] ([c]GS2)[xshift=-1.5cm] circle (3pt);
  \node at ([c]GS2) [xshift=-.5cm,yshift=-.2cm] {\begin{color}{gray}\scriptsize{$f(E_2)$}\end{color}};
  
  \fill[blue] ([c]GS2)[xshift=-2.5cm] circle (3pt);
  \node at ([c]GS2) [xshift=-1.2cm,yshift=-.2cm] {\begin{color}{gray}\scriptsize{$f(E_1)$}\end{color}};
  
  \fill[blue] ([c]GS2)[xshift=-.5cm] circle (3pt);

    \end{scope}


\begin{scope}[every coordinate/.style={shift={(14,-6)}}]     
 \draw ([c]Q1)-- ( [c]Q2)--([c]Q3long)--([c]Q4long)--([c]Q1);   
 \end{scope}

\begin{scope}[every coordinate/.style={shift={(14,-6)}}]

\draw[draw=gray] ([c]AE1) .. controls ([c]BE1) ..
       ([c]CE1);
       \draw[draw=gray]  plot [smooth, tension=3] coordinates { ([c]AE1) ([c]Hh1) ([c]Hh2)};
       \draw[draw=gray]  plot [smooth, tension=3] coordinates { ([c]CE1) ([c]HE1) ([c]HE2)};

       \draw[draw=gray] ([c]AE2) .. controls ([c]BE2) ..
       ([c]CE2);
       \draw[draw=gray]  plot [smooth, tension=3] coordinates { ([c]AE2) ([c]Jj1) ([c]Jj2)};
       \draw[draw=gray]  plot [smooth, tension=3] coordinates { ([c]CE2) ([c]JE1) ([c]JE2)};

        \draw[thick] ([c]T2) -- ([c]T1);

          \draw[thick] ([c]R1) -- ([c]R2);
       
       \draw[->,thick] ([c]Fs) to ([c]Ft);
        \node at ([c]Flabel) [above] {\scriptsize{$f$}};
       
\draw[->,dotted] ([c]fact1) to ([c]fact2);
                          
     \draw[->,thick] ([c]Fbs) to ([c]Fbt);         
 \node at ([c]Fblabel) [right] {\scriptsize{$\bar{f}$}};
\node at ([c]Qlongmiddle) [below] {\scriptsize{$\pazocal E^2\pazocal E^{(2)}\cap\overline{\pazocal M}^{\text{main}}$ (sprouting)}};
 \end{scope}
 
 
 \coordinate (DS) at (2.3,3.35) ;
 
 \coordinate (ES) at (2.6,2.8) ;
 \coordinate (FS) at (1.5,2) ;
 \coordinate (S1) at (2.52,4) ;
 \coordinate (S2) at (2.52,1.8) ;
 \coordinate (W1) at (3.52,2.7) ;
 \coordinate (W2) at (2.2,1.9) ;
 
\begin{scope}[every coordinate/.style={shift={(16.6,-9.5)}}]

\draw[thick] ([c]B) parabola ([c]H2);
   \draw[thick]   ([c]B) parabola  ([c]DS);
      \draw[thick] ([c]DS) .. controls ([c]ES)..([c]FS);

\draw[dashed] ([c]S1) -- ([c]S2);
\draw[thick] ([c]W1) -- ([c]W2);
   
\end{scope}


  
  \coordinate (Gs1) at (-.5,0);
 \coordinate (Gs2) at (2.4,0);
   \coordinate (Gs3) at (2,0);
   \begin{scope}[every coordinate/.style={shift={(20,-3)}}]
\draw[thick] ([c]G1)  ellipse (1cm and 0.5cm);
      \draw[thick] ([c]Gs1) -- ([c]Gs2);

 \draw[dashed] ([c]P1)-- ( [c]P2)--([c]P3)--([c]P4)--([c]P1);  

          \node at ([c]Gs2) [below] {\begin{color}{gray}\scriptsize{$f(E_2)$}\end{color}};
            \fill[blue] ([c]Gs3) circle (3pt);
            \node at ([c]Gs1) [xshift=.2cm,below] {\begin{color}{gray}\scriptsize{$f(E_1)$}\end{color}};
            \fill[blue] ([c]Gs1)[xshift=.2cm] circle (3pt);
    \end{scope}   

  \end{tikzpicture}
  \caption{$\pazocal E^2\!\pazocal E^{(2)},\dim=14$; $\pazocal E^2\!\pazocal E^{(2)}\cap\overline{\pazocal M}^{\text{main}},\dim=11$}
  \label{E2E2}
\end{figure}

 \item[$\pazocal{E}^2\!\pazocal{E}^{(1^2)}$ component] The image consists of a conic and two lines, all of them concurring in a point. The dimension is $13$. A map is smoothable if the separating bridge covers a line two-to-one, and $E_1$ is attached to one of the ramification points. This locus has dimension $11$. See Figure \ref{E2E11}.

\begin{figure}[h]
   \begin{tikzpicture}[scale=.4]

\coordinate (Hh1) at (2,4);
 \coordinate (Hh2) at (2.1,4.5);
 
 \coordinate (AE1) at (1.7,3.7);
 \coordinate (BE1) at (1.4,3.5);
 \coordinate (CE1) at   (1.7,3.3);
 \coordinate (HE1) at (2,3);
 \coordinate (HE2) at (2.1,2.5);

\coordinate (Jj1) at (2.5,4);
 \coordinate (Jj2) at (2.6,4.5);
 
 \coordinate (AE2) at (2.9,3.7);
 \coordinate (BE2) at (3.2,3.5);
 \coordinate (CE2) at   (2.9,3.3);
 \coordinate (JE1) at (2.6,3);
 \coordinate (JE2) at (2.5,2.5);

 \coordinate (T1) at (1.7,3);
 \coordinate (T2) at (3,3);
 \coordinate (R1) at (2.6,3.5);
 \coordinate (R2) at (3.6,4.5);
 \coordinate (R1b) at (2.3,3.8);
 \coordinate (R2b) at (3.6,4.8);
 
   \coordinate(fact1) at (2.4,1.2);
   \coordinate(fact2) at (3.1,-.2);
   
 \coordinate (Fs) at (3.3,3.2);
  \coordinate (Ft) at (4.5,3.2);
   \coordinate (Flabel) at (3.8,3.5);
   
    \coordinate (Fbs) at (5.3,-.3);
  \coordinate (Fbt) at (6.4,1);
    \coordinate (Fblabel) at (5.9,.1);

   \begin{scope}[every coordinate/.style={shift={(0,0)}}]
\draw[draw=gray] ([c]AE1) .. controls ([c]BE1) ..
       ([c]CE1);
       \draw[draw=gray]  plot [smooth, tension=3] coordinates { ([c]AE1) ([c]Hh1) ([c]Hh2)};
       \draw[draw=gray]  plot [smooth, tension=3] coordinates { ([c]CE1) ([c]HE1) ([c]HE2)};

       \draw[draw=gray] ([c]AE2) .. controls ([c]BE2) ..
       ([c]CE2);
       \draw[draw=gray]  plot [smooth, tension=3] coordinates { ([c]AE2) ([c]Jj1) ([c]Jj2)};
       \draw[draw=gray]  plot [smooth, tension=3] coordinates { ([c]CE2) ([c]JE1) ([c]JE2)};

        \draw[thick] ([c]T2) -- ([c]T1);
        \draw[thick] ([c]R1) -- ([c]R2)node[xshift=.1cm,yshift=-.1cm]{\scriptsize{$1$}};
          \draw[thick] ([c]R1b) -- ([c]R2b)node[xshift=.1cm]{\scriptsize{$1$}};
       
       \draw[->,thick] ([c]Fs) to ([c]Ft);
        \node at ([c]Flabel) [xshift=.15cm,yshift=.1cm] {\scriptsize{$f$}};
       \node at ([c]CE1) [xshift=-.1cm,yshift=.4cm] {\begin{color}{gray}\scriptsize{$E_1$}\end{color}};
         \node at ([c]JE2) [xshift=.2cm] {\begin{color}{gray}\scriptsize{$E_2$}\end{color}};
  \node at ([c]T1) [below] {\scriptsize{$2$}};

 \end{scope}

 \coordinate(Q1) at  (1,5);
 \coordinate (Q2) at (10,5);
 \coordinate (Q3) at (10,1 );
 \coordinate (Q4) at (1,1);
 \coordinate (Q3long) at (10,-2.5 );
 \coordinate (Q4long) at (1,-2.5);
\coordinate (Qlongmiddle) at (5.5,-2.5);

 \coordinate(P1) at  (-0.2,1);
 \coordinate (P2) at ( 3.5,1);
 \coordinate (P3) at (2.5,-1 );
 \coordinate (P4) at (-1.2,-1);

  \begin{scope}[every coordinate/.style={shift={(0,0)}}]     
 \draw (Q1)-- ( Q2)--(Q3)--(Q4)--(Q1);   
 \end{scope}
 

  
  \coordinate (G1) at (1,0);
 \coordinate (G2) at (1.35,.45);
 \coordinate (G3) at (.5,-.8);
 \coordinate (G4) at (1.1,-.8);
 \coordinate (G5) at (1.65,.8);
 \coordinate (G6) at (1.4,.8);
  
   \begin{scope}[every coordinate/.style={shift={(6,3)}}]
\draw[thick] ([c]G1)  ellipse (1.3cm and 0.5cm);
      \draw[thick] ([c]G5)--([c]G3) ;
      \draw[thick] ([c]G6)--([c]G4) ;

 \draw[dashed] ([c]P1)-- ( [c]P2)--([c]P3)--([c]P4)--([c]P1);  
    
     \fill[gray] ([c]G2) circle (3pt);
          \node at ([c]G2) [right] {\begin{color}{gray}\scriptsize{$f(E_2)$}\end{color}};
     \fill[gray] ([c]G2)[xshift=-1.6cm,yshift=-.6cm] circle (3pt);
          \node at ([c]G2)[xshift=-.65cm,yshift=-.45cm] {\begin{color}{gray}\scriptsize{$f(E_1)$}\end{color}};

    \end{scope}

  \node at (5,1) [below] {\scriptsize{$\pazocal E^2\pazocal E^{(1^2)}$}};

\begin{scope}[every coordinate/.style={shift={(14,0)}}]
\draw[thick] ([c]A) .. controls ([c]B) ..
       ([c]C) .. controls ([c]D)  .. ([c]E) .. controls ([c]F).. ([c]G);
      \draw[thick]  plot [smooth, tension=3] coordinates { ([c]A) ([c]H1) ([c]H2)};
       \draw[thick]  plot [smooth, tension=3] coordinates { ([c]G) ([c]J1) ([c]J2)};
       \draw[->,thick] ([c]Fs) to ([c]Ft);
        \node at ([c]Flabel) [xshift=.15cm,yshift=.1cm] {\scriptsize{$f$}};
       \end{scope}
 \coordinate(Q1) at  (1,5);
 \coordinate (Q2) at (10,5);
 \coordinate (Q3) at (10,1 );
 \coordinate (Q4) at (1,1);

  \begin{scope}[every coordinate/.style={shift={(14,0)}}]     
 \draw ([c]Q1)-- ( [c]Q2)--([c]Q3)--([c]Q4)--([c]Q1);   
 \end{scope}

 \coordinate (Am) at (0,0);
 \coordinate (Bm) at(0.2,-.3);
 \coordinate (Cm) at  ( .4,0);
 \coordinate (Dm) at (.6,.3)  ;
  \coordinate (Em) at (.8,0) ;
  \coordinate (Fm) at (1,-.3);
  \coordinate (Gm) at  (1.2,0);
  \coordinate (H1m) at (-.5,.3);
 \coordinate (H2m) at (-.2,.2);
\coordinate (J1m) at (1.4,.2);
 \coordinate (J2m) at (1.7,.3);
 \coordinate (N1m) at (3.2,-.3);
 \coordinate (N2m) at (.2,-.3);
  \coordinate (N3m) at (3,.3);
 \coordinate(P1) at  (-0.2,1);
 \coordinate (P2) at ( 3.5,1);
 \coordinate (P3) at (2.5,-1 );
 \coordinate (P4) at (-1.2,-1);
 
 \begin{scope}[every coordinate/.style={shift={(20,3)}}]
 \draw[thick]  plot [smooth, tension=3] coordinates {([c]H1m) ([c]H2m) ([c]Am)  };
\draw[thick] ([c]Am) .. controls ([c]Bm) ..
       ([c]Cm) .. controls ([c]Dm)  .. ([c]Em) .. controls ([c]Fm).. ([c]Gm);
    \draw[thick]  plot [smooth, tension=3] coordinates { ([c]Gm) ([c]J1m) ([c]J2m)  };   
     \draw[thick]   ([c]J2m).. controls ([c]N1m) and ([c]N2m).. ([c]N3m);
 \draw[dashed] ([c]P1)-- ( [c]P2)--([c]P3)--([c]P4)--([c]P1);  
    \end{scope}   
       

  \draw (7.5,1)--(13.,-1);
\draw (19.5,1)--(16.,-1);

 \node at (12.3,3) [above] { $\cap$};


\begin{scope}[every coordinate/.style={shift={(8,-6)}}]

 \draw ([c]Q1)-- ( [c]Q2)--([c]Q3long)--([c]Q4long)--([c]Q1);   
 \end{scope}

\begin{scope}[every coordinate/.style={shift={(8,-6)}}]

\draw[draw=gray] ([c]AE1) .. controls ([c]BE1) ..
       ([c]CE1);
       \draw[draw=gray]  plot [smooth, tension=3] coordinates { ([c]AE1) ([c]Hh1) ([c]Hh2)};
       \draw[draw=gray]  plot [smooth, tension=3] coordinates { ([c]CE1) ([c]HE1) ([c]HE2)};

       \draw[draw=gray] ([c]AE2) .. controls ([c]BE2) ..
       ([c]CE2);
       \draw[draw=gray]  plot [smooth, tension=3] coordinates { ([c]AE2) ([c]Jj1) ([c]Jj2)};
       \draw[draw=gray]  plot [smooth, tension=3] coordinates { ([c]CE2) ([c]JE1) ([c]JE2)};

        \draw[thick] ([c]T2) -- ([c]T1);
       \draw[thick] ([c]R2) -- ([c]R1);
       \draw[thick] ([c]R2b) -- ([c]R1b);

       \draw[->,thick] ([c]Fs) to ([c]Ft);
        \node at ([c]Flabel) [xshift=.15cm,yshift=.1cm] {\scriptsize{$f$}};
       
\draw[->,dotted] ([c]fact1) to ([c]fact2);
                          
     \draw[->,thick] ([c]Fbs) to ([c]Fbt);         
 \node at ([c]Fblabel) [right] {\scriptsize{$\bar{f}$}};
\node at ([c]Qlongmiddle) [below] {\scriptsize{$\pazocal E^2\pazocal E^{(1^2)}\cap\overline{\pazocal M}^{\text{main}}$}};
 \end{scope}
 
 
 \coordinate (DS) at (2.3,3.35) ;
 
 \coordinate (ES) at (2.6,2.8) ;
 \coordinate (FS) at (1.5,2) ;
 \coordinate (S1) at (2,4) ;
 \coordinate (S2) at (3.2,2.5) ;
 \coordinate (S1b) at (3.2,4) ;
 \coordinate (S2b) at (1.5,2.9) ;
 
\begin{scope}[every coordinate/.style={shift={(10.1,-9.5)}}]

\draw[thick] ([c]B) parabola ([c]H2);
   \draw[thick]   ([c]B) parabola  ([c]DS);
      \draw[thick] ([c]DS) .. controls ([c]ES)..([c]FS);
\end{scope}
\begin{scope}[every coordinate/.style={shift={(10,-9.9)}}]
\draw[thick] ([c]S1) -- ([c]S2);
   
\end{scope}

\begin{scope}[every coordinate/.style={shift={(10,-10.2)}}]

   \draw[thick] ([c]S1b) -- ([c]S2b);
\end{scope}

 \coordinate (F5) at (-.3,.2) ;
  \coordinate (F6) at (2,.2) ;
   \coordinate (F8) at (.1,.9) ;
   \coordinate (F9) at (1.8,-.21) ;
   
   \coordinate (F7) at (.9,.2) ;
  
  \coordinate (F10) at (1.7,.2) ;

  \begin{scope}[every coordinate/.style={shift={(14,-3)}}]

       \draw[thick] ([c]F5)--([c]F6);
        \draw[thick] ([c]E8)--(15.2,-2);
         \draw[thick] ([c]F8)--([c]F9);
   
 \draw[dashed] ([c]P1)-- ( [c]P2)--([c]P3)--([c]P4)--([c]P1);

     \fill[blue] ([c]F7)[xshift=-.5cm] circle (3pt);
          \node at ([c]F7)[xshift=-.4cm,yshift=-.3cm] {\begin{color}{gray}\scriptsize{$f(E_1)$}\end{color}};
          
     \fill[gray] ([c]F7)[xshift=.25cm] circle (3pt);
        \node at ([c]F7)[xshift=.3cm,yshift=-.3cm] {\begin{color}{gray}\scriptsize{$f(E_2)$}\end{color}};
   
 \fill[blue] ([c]F7)[xshift=.8cm] circle (3pt);
\end{scope}
   \end{tikzpicture}
   \caption{$\pazocal E^2\!\pazocal E^{(1^2)},\dim=13$; $\pazocal E^2\!\pazocal E^{(1^2)}\cap\overline{\pazocal M}^{\text{main}},\dim=11$}
   \label{E2E11}
\end{figure}

 \item[${}^{(1)}\!\pazocal{E}^2\!\pazocal{E}^{(1)}$ component] The image consists of a conic and two lines. The intersection with \emph{main} has three components, all of dimension $10$: the separating bridge covers a line two-to-one; and the other lines can either be equal (tangent) to this one, or meet it in a branch point of the cover. Again, we note that aligning creates two-dimensional fibres, entirely contained in the factorisation locus. See Figures \ref{1E2E1} and \ref{fig:1E2E1}.
 
\begin{figure}[h]
   \begin{tikzpicture}[scale=.5]

\coordinate (Hh1) at (2,4);
 \coordinate (Hh2) at (2.1,4.5);
 
 \coordinate (AE1) at (1.7,3.7);
 \coordinate (BE1) at (1.4,3.5);
 \coordinate (CE1) at   (1.7,3.3);
 \coordinate (HE1) at (2,3);
 \coordinate (HE2) at (2.1,2.5);

\coordinate (Jj1) at (2.5,4);
 \coordinate (Jj2) at (2.6,4.5);
 
 \coordinate (AE2) at (2.9,3.7);
 \coordinate (BE2) at (3.2,3.5);
 \coordinate (CE2) at   (2.9,3.3);
 \coordinate (JE1) at (2.6,3);
 \coordinate (JE2) at (2.5,2.5);

 \coordinate (T1) at (1.7,3);
 \coordinate (T2) at (3,3);
 \coordinate (R1) at (2.6,3.5);
 \coordinate (R2) at (3.6,4.5);
 \coordinate (R1b) at (2,3.5);
 \coordinate (R2b) at (1.2,4.5);

   \coordinate(fact1) at (2.4,1.2);
   \coordinate(fact2) at (3.1,-.2);
   
 \coordinate (Fs) at (3.1,3.2);
  \coordinate (Ft) at (4.5,3.2);
   \coordinate (Flabel) at (3.8,3.5);
   
    \coordinate (Fbs) at (5.3,-.3);
  \coordinate (Fbt) at (6.4,1);
    \coordinate (Fblabel) at (5.9,.1);

   \begin{scope}[every coordinate/.style={shift={(0,0)}}]
\draw[draw=gray] ([c]AE1) .. controls ([c]BE1) ..
       ([c]CE1);
       \draw[draw=gray]  plot [smooth, tension=3] coordinates { ([c]AE1) ([c]Hh1) ([c]Hh2)};
       \draw[draw=gray]  plot [smooth, tension=3] coordinates { ([c]CE1) ([c]HE1) ([c]HE2)};

       \draw[draw=gray] ([c]AE2) .. controls ([c]BE2) ..
       ([c]CE2);
       \draw[draw=gray]  plot [smooth, tension=3] coordinates { ([c]AE2) ([c]Jj1) ([c]Jj2)};
       \draw[draw=gray]  plot [smooth, tension=3] coordinates { ([c]CE2) ([c]JE1) ([c]JE2)};

        \draw[thick] ([c]T2) -- ([c]T1);
        \draw[thick] ([c]R1) -- ([c]R2);
               \draw[thick] ([c]R1b) -- ([c]R2b);

       \draw[->,thick] ([c]Fs) to ([c]Ft);
        \node at ([c]Flabel) [xshift=.15cm,yshift=.1cm] {\scriptsize{$f$}};
       \node at ([c]CE1) [xshift=.05cm,yshift=.6cm] {\begin{color}{gray}\scriptsize{$E_1$}\end{color}};
         \node at ([c]JE2) [xshift=.2cm] {\begin{color}{gray}\scriptsize{$E_2$}\end{color}};
  \node at ([c]T1) [below] {\scriptsize{$2$}};
    
 \end{scope}
 
 \coordinate(Q1) at  (.5,5.5);
 \coordinate (Q2) at (10,5.5);
 \coordinate (Q3) at (10,1 );
 \coordinate (Q4) at (.5,1);
 \coordinate (Q3long) at (10,-2.5 );
 \coordinate (Q4long) at (.5,-2.5);
\coordinate (Qlongmiddle) at (5.5,-2.5);

 \coordinate(P1) at  (-0.2,1);
 \coordinate (P2) at ( 3.5,1);
 \coordinate (P3) at (2.5,-1 );
 \coordinate (P4) at (-1.2,-1);

  \begin{scope}[every coordinate/.style={shift={(0,0)}}]     
 \draw (Q1)-- ( Q2)--(Q3)--(Q4)--(Q1);   
 \end{scope}
 
  
  
  \coordinate (G1) at (1,0);
 \coordinate (G2) at (1.35,.45);
 \coordinate (G3) at (.5,-.8);
 \coordinate (G4) at (1.5,-.8);
 \coordinate (G5) at (1.65,.8);
 \coordinate (G6) at (.7,1);
  
   \begin{scope}[every coordinate/.style={shift={(6,3)}}]
\draw[thick] ([c]G1)  ellipse (1.3cm and 0.5cm);
      \draw[thick] ([c]G5)--([c]G3) ;
      \draw[thick] ([c]G6)--([c]G4) ;

 \draw[dashed] ([c]P1)-- ( [c]P2)--([c]P3)--([c]P4)--([c]P1);  
    
     \fill[gray] ([c]G2)[xshift=.05cm] circle (3pt);
          \node at ([c]G2) [right] {\begin{color}{gray}\scriptsize{$f(E_2)$}\end{color}};
     \fill[gray] ([c]G2)[xshift=-.4cm] circle (3pt);
          \node at ([c]G2) [xshift=-.6cm] {\begin{color}{gray}\scriptsize{$f(E_1)$}\end{color}};

    \end{scope}

\node at (5,1) [below] {\scriptsize{${}^{(1)}\pazocal E^2\pazocal E^{(1)}$}};

\begin{scope}[every coordinate/.style={shift={(14,0)}}]
\draw[thick] ([c]A) .. controls ([c]B) ..
       ([c]C) .. controls ([c]D)  .. ([c]E) .. controls ([c]F).. ([c]G);
      \draw[thick]  plot [smooth, tension=3] coordinates { ([c]A) ([c]H1) ([c]H2)};
       \draw[thick]  plot [smooth, tension=3] coordinates { ([c]G) ([c]J1) ([c]J2)};
       \draw[->,thick] ([c]Fs) to ([c]Ft);
         \node at ([c]Flabel) [xshift=.15cm,yshift=.1cm] {\scriptsize{$f$}};
       \end{scope}
 \coordinate(Q1) at  (.5,5.5);
 \coordinate (Q2) at (10,5.5);
 \coordinate (Q3) at (10,1 );
 \coordinate (Q4) at (.5,1);
 \coordinate (Q3long) at (10,-2.5 );
 \coordinate (Q4long) at (.5,-2.5);
\coordinate (Qlongmiddle) at (5.5,-2.5);
  \begin{scope}[every coordinate/.style={shift={(14,0)}}]     
 \draw ([c]Q1)-- ( [c]Q2)--([c]Q3)--([c]Q4)--([c]Q1);   
 \end{scope}

 \coordinate (Am) at (0,0);
 \coordinate (Bm) at(0.2,-.3);
 \coordinate (Cm) at  ( .4,0);
 \coordinate (Dm) at (.6,.3)  ;
  \coordinate (Em) at (.8,0) ;
  \coordinate (Fm) at (1,-.3);
  \coordinate (Gm) at  (1.2,0);
  \coordinate (H1m) at (-.5,.3);
 \coordinate (H2m) at (-.2,.2);
\coordinate (J1m) at (1.4,.2);
 \coordinate (J2m) at (1.7,.3);
 \coordinate (N1m) at (3.2,-.3);
 \coordinate (N2m) at (.2,-.3);
  \coordinate (N3m) at (3,.3);
 \coordinate(P1) at  (-0.2,1);
 \coordinate (P2) at ( 3.5,1);
 \coordinate (P3) at (2.5,-1 );
 \coordinate (P4) at (-1.2,-1);
 
 \begin{scope}[every coordinate/.style={shift={(20,3)}}]
 \draw[thick]  plot [smooth, tension=3] coordinates {([c]H1m) ([c]H2m) ([c]Am)  };
\draw[thick] ([c]Am) .. controls ([c]Bm) ..
       ([c]Cm) .. controls ([c]Dm)  .. ([c]Em) .. controls ([c]Fm).. ([c]Gm);
    \draw[thick]  plot [smooth, tension=3] coordinates { ([c]Gm) ([c]J1m) ([c]J2m)  };   
     \draw[thick]   ([c]J2m).. controls ([c]N1m) and ([c]N2m).. ([c]N3m);
 \draw[dashed] ([c]P1)-- ( [c]P2)--([c]P3)--([c]P4)--([c]P1);  
    \end{scope}   
       

\draw (7.5,1)--(4,-1.5);
\draw (7.5,1)--(13.,-1.5);
\draw (7.5,1)--(24,-1.5);

 \node at (12.3,3) [above] { $\cap$};

\begin{scope}[every coordinate/.style={shift={(-2,-7)}}]     
 \draw ([c]Q1)-- ( [c]Q2)--([c]Q3long)--([c]Q4long)--([c]Q1);   
  \node at ([c]Qlongmiddle) [below] {\scriptsize{${}^{(1)}\pazocal E^2\pazocal E^{(1)}\cap\overline{\pazocal M}^{\text{main}}$, tacnodes}};
 \end{scope}

\begin{scope}[every coordinate/.style={shift={(-2,-7)}}]
\draw[draw=gray] ([c]AE1) .. controls ([c]BE1) ..
       ([c]CE1);
       \draw[draw=gray]  plot [smooth, tension=3] coordinates { ([c]AE1) ([c]Hh1) ([c]Hh2)};
       \draw[draw=gray]  plot [smooth, tension=3] coordinates { ([c]CE1) ([c]HE1) ([c]HE2)};

       \draw[draw=gray] ([c]AE2) .. controls ([c]BE2) ..
       ([c]CE2);
       \draw[draw=gray]  plot [smooth, tension=3] coordinates { ([c]AE2) ([c]Jj1) ([c]Jj2)};
       \draw[draw=gray]  plot [smooth, tension=3] coordinates { ([c]CE2) ([c]JE1) ([c]JE2)};

        \draw[thick] ([c]T2) -- ([c]T1);
        \draw[thick] ([c]R1) -- ([c]R2);
               \draw[thick] ([c]R1b) -- ([c]R2b);

       \draw[->,thick] ([c]Fs) to ([c]Ft);
          \node at ([c]Flabel) [xshift=.15cm,yshift=.1cm] {\scriptsize{$f$}};

    \draw[->,dotted] ([c]fact1) to ([c]fact2);
                          
     \draw[->,thick] ([c]Fbs) to ([c]Fbt);         
 \node at ([c]Fblabel) [right] {\scriptsize{$\bar{f}$}};
    
 \end{scope}

 
 \coordinate (D1S) at (1.5,3) ;
 \coordinate (E1S) at (2.5,2.5) ;
 \coordinate (F1S) at (1.5,2) ;
 
 \coordinate (D2S) at (1.5,4.8) ;
 \coordinate (E2S) at (2.5,4.2) ;
 \coordinate (F2S) at (1.5,3.5) ;
 
 \coordinate (S1) at (2.5,5) ;
 \coordinate (S2) at (2.5,1.8) ;

\begin{scope}[every coordinate/.style={shift={(0.1,-10.5)}}]

\draw[thick] ([c]F1S) parabola ([c]E1S);
   \draw[thick] ([c]D1S) parabola ([c]E1S);
   
   \draw[thick] ([c]F2S) parabola ([c]E2S);
   \draw[thick] ([c]D2S) parabola ([c]E2S);

\draw[thick] ([c]S1) -- ([c]S2);
   
\end{scope}


\coordinate (GS12) at (-.5,0.2);
 \coordinate (GS22) at (2.4,0.2);
    \coordinate (GS11) at (-.5,0.1);
 \coordinate (GS21) at (2.4,0.1);
  \coordinate (GS1) at (-.5,0);
 \coordinate (GS2) at (2.4,.0);
   \coordinate (GB) at (1.35,0);
   \coordinate (GB2) at (.55,0);
   
   \begin{scope}[every coordinate/.style={shift={(4,-4)}}]
      \draw[thick] ([c]GS1) -- ([c]GS2);
    \draw[thick] ([c]GS11) -- ([c]GS21);
     \draw[thick] ([c]GS12) -- ([c]GS22);
   
 \draw[dashed] ([c]P1)-- ( [c]P2)--([c]P3)--([c]P4)--([c]P1);  
 \fill[blue] ([c]GB) circle (3pt);
      \fill[blue] ([c]GB2) circle (3pt);

    \end{scope}


\begin{scope}[every coordinate/.style={shift={(8,-7)}}]     
 \draw ([c]Q1)-- ( [c]Q2)--([c]Q3long)--([c]Q4long)--([c]Q1);   
  \node at ([c]Qlongmiddle) [below] {\scriptsize{${}^{(1)}\pazocal E^2\pazocal E^{(1)}\cap\overline{\pazocal M}^{\text{main}}$,  sprouting}};
 \end{scope}

\begin{scope}[every coordinate/.style={shift={(8,-7)}}]
\draw[draw=gray] ([c]AE1) .. controls ([c]BE1) ..
       ([c]CE1);
       \draw[draw=gray]  plot [smooth, tension=3] coordinates { ([c]AE1) ([c]Hh1) ([c]Hh2)};
       \draw[draw=gray]  plot [smooth, tension=3] coordinates { ([c]CE1) ([c]HE1) ([c]HE2)};

       \draw[draw=gray] ([c]AE2) .. controls ([c]BE2) ..
       ([c]CE2);
       \draw[draw=gray]  plot [smooth, tension=3] coordinates { ([c]AE2) ([c]Jj1) ([c]Jj2)};
       \draw[draw=gray]  plot [smooth, tension=3] coordinates { ([c]CE2) ([c]JE1) ([c]JE2)};

        \draw[thick] ([c]T2) -- ([c]T1);
        \draw[thick] ([c]R1) -- ([c]R2);
               \draw[thick] ([c]R1b) -- ([c]R2b);

       \draw[->,thick] ([c]Fs) to ([c]Ft);
          \node at ([c]Flabel) [xshift=.15cm,yshift=.1cm] {\scriptsize{$f$}};

    \draw[->,dotted] ([c]fact1) to ([c]fact2);
                          
     \draw[->,thick] ([c]Fbs) to ([c]Fbt);         
 \node at ([c]Fblabel) [right] {\scriptsize{$\bar{f}$}};
    
 \end{scope}


 \coordinate (Ss1) at (2,5.8) ;
 \coordinate (Ss2) at (2,4.2) ;

\begin{scope}[every coordinate/.style={shift={(10.1,-10.5)}}]

\draw[thick] ([c]F1S) parabola ([c]E1S);
   \draw[thick] ([c]D1S) parabola ([c]E1S);
   
   \draw[dashed] ([c]F2S) parabola ([c]E2S);
   \draw[dashed] ([c]D2S) parabola ([c]E2S);

\draw[thick] ([c]S1) -- ([c]S2);
  \draw[thick] ([c]Ss1) -- ([c]Ss2);   
\end{scope}


\coordinate (GI1) at (1.2,- 0.7);
 \coordinate (GI2) at (1.5,.7);
     \coordinate (GB) at (1.35,0);
     \coordinate (GB2) at (.55,0);
     
   \begin{scope}[every coordinate/.style={shift={(14,-4)}}]
      \draw[thick] ([c]GS1) -- ([c]GS2);
    \draw[thick] ([c]GS11) -- ([c]GS21);
     \draw[thick] ([c]GI1) -- ([c]GI2);
      \fill[blue] ([c]GB) circle (3pt);
      \fill[blue] ([c]GB2) circle (3pt);
   
 \draw[dashed] ([c]P1)-- ( [c]P2)--([c]P3)--([c]P4)--([c]P1);

    \end{scope}

\begin{scope}[every coordinate/.style={shift={(18,-7)}}]     
 \draw ([c]Q1)-- ( [c]Q2)--([c]Q3long)--([c]Q4long)--([c]Q1);   
  \node at ([c]Qlongmiddle) [below] {\scriptsize{${}^{(1)}\pazocal E^2\pazocal E^{(1)}\cap\overline{\pazocal M}^{\text{main}}$, ribbon}};
 \end{scope}

\begin{scope}[every coordinate/.style={shift={(18,-7)}}]
\draw[draw=gray] ([c]AE1) .. controls ([c]BE1) ..
       ([c]CE1);
       \draw[draw=gray]  plot [smooth, tension=3] coordinates { ([c]AE1) ([c]Hh1) ([c]Hh2)};
       \draw[draw=gray]  plot [smooth, tension=3] coordinates { ([c]CE1) ([c]HE1) ([c]HE2)};

       \draw[draw=gray] ([c]AE2) .. controls ([c]BE2) ..
       ([c]CE2);
       \draw[draw=gray]  plot [smooth, tension=3] coordinates { ([c]AE2) ([c]Jj1) ([c]Jj2)};
       \draw[draw=gray]  plot [smooth, tension=3] coordinates { ([c]CE2) ([c]JE1) ([c]JE2)};

        \draw[thick] ([c]T2) -- ([c]T1);
        \draw[thick] ([c]R1) -- ([c]R2);
               \draw[thick] ([c]R1b) -- ([c]R2b);

       \draw[->,thick] ([c]Fs) to ([c]Ft);
          \node at ([c]Flabel) [xshift=.15cm,yshift=.1cm] {\scriptsize{$f$}};

    \draw[->,dotted] ([c]fact1) to ([c]fact2);
                          
     \draw[->,thick] ([c]Fbs) to ([c]Fbt);         
 \node at ([c]Fblabel) [right] {\scriptsize{$\bar{f}$}};
    
 \end{scope}


 \coordinate (YS1) at (0,4.8) ;
 \coordinate (YS2) at (0,2) ;
  \coordinate (YS11) at (0.2,4.8) ;
 \coordinate (YS12) at (0.2,2) ;

  \coordinate (ZS1) at (-.3,4.5) ;
 \coordinate (ZS2) at (1.5,4.5) ;
 
  \coordinate (WS1) at (-.3,2.8) ;
 \coordinate (WS2) at (1.5,2.8) ;
 
\begin{scope}[every coordinate/.style={shift={(22.1,-10.5)}}]

\draw[thick] ([c]YS1) -- ([c]YS2)--([c]YS12)--([c]YS11)--([c]YS1);
 \draw[thick] ([c]ZS1) -- ([c]ZS2);
 \draw[thick] ([c]WS1) -- ([c]WS2);
\end{scope}


\coordinate (GI1) at (1.2,- 0.7);
 \coordinate (GI2) at (1.5,.7);
 \coordinate (GI11) at (.4,- 0.7);
 \coordinate (GI12) at (.7,.7);
     \coordinate (GB) at (1.35,0);
      \coordinate (GB2) at (.55,0);
   \begin{scope}[every coordinate/.style={shift={(24,-4)}}]
      \draw[thick] ([c]GS1) -- ([c]GS2);
       \draw[thick] ([c]GI11) -- ([c]GI12);
     \draw[thick] ([c]GI1) -- ([c]GI2);
     \fill[blue] ([c]GB) circle (3pt);
      \fill[blue] ([c]GB2) circle (3pt);
   
 \draw[dashed] ([c]P1)-- ( [c]P2)--([c]P3)--([c]P4)--([c]P1);

    \end{scope}   

  \end{tikzpicture}
  \caption{${}^{(1)}\pazocal E^2\pazocal E^{(1)},\dim=13$; ${}^{(1)}\pazocal E^2\pazocal E^{(1)}\cap\overline{\pazocal M}^{\text{main}},\dim=10$}
  \label{1E2E1}
\end{figure}

 \begin{figure}[hbt]
 \begin{tikzpicture}[scale=.6]
  \foreach \x in {(0,0),(2,0),(5,0),(7,0),(9,0),(11.5,-.5),(13.5,0),(15,.5),(16.5,0)}
  \draw[fill=black] \x circle (2pt);
  \draw (0,0)--node[above]{\tiny $\ell_1$}(1,1)--node[above]{\tiny $\ell_2$}(2,0)--node[above]{\tiny $\ell_3$}(3.5,1.5)--node[above]{\tiny $\ell_4$}(5,0);
  \draw (7,0) --node[above]{\tiny $\ell_1$} (8,1) --node[above]{\tiny $\ell_2$} (9,0) --node[above]{\tiny $\ell_3$} (10,1) --node[above]{\tiny $\ell_4$} (11.5,-.5);
  \draw (13.5,0) --node[above]{\tiny $\ell_1$} (14.5,1) --node[above]{\tiny $\ell_2$} (15,.5) --node[above]{\tiny $\ell_3$} (15.5,1) --node[above]{\tiny $\ell_4$} (16.5,0);
  \foreach \x in {(1,1),(3.5,1.5),(8,1),(10,1),(14.5,1),(15.5,1)}
  \draw[fill=white] \x circle (2pt);
  \draw (2.5,-1) node{$\ell_1=\ell_2,\ell_3=\ell_4$};
  \draw (9.25,-1) node{$\ell_1=\ell_2=\ell_3$};
  \draw (15,-1) node{$\ell_1=\ell_4,\ell_2=\ell_3$};
  \draw (1,-2) -- (2,-2) (3,-2) -- (4,-2) (1.75,-2.1)--(3.25,-2.1)--(3.25,-1.9)--(1.75,-1.9)--(1.75,-2.1);
  \fill[blue] (2.25,-2) circle (1pt) (2.75,-2) circle (1pt);
  \draw (7.75,-2) -- (8.75,-2) (9.75,-2) -- (9.75,-3) (8.5,-2.1)--(10,-2.1)--(10,-1.9)--(8.5,-1.9)--(8.5,-2.1);
  \fill[blue] (9.25,-2) circle (1pt) (9.75,-2) circle (1pt);
  \draw (14.5,-2) -- (14.5,-3) (15.5,-2) -- (15.5,-3) (14.25,-2.1)--(15.75,-2.1)--(15.75,-1.9)--(14.25,-1.9)--(14.25,-2.1);
  \fill[blue] (14.5,-2) circle (1pt) (15.5,-2) circle (1pt);
 \end{tikzpicture}
\caption{Inconsequential alignment and factorisation.}
\label{fig:1E2E1}
\end{figure}
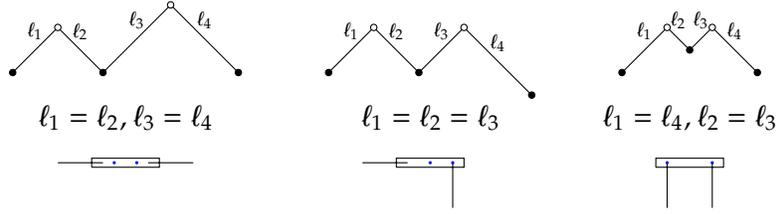
 \item[${}^{\mu_1}\!\pazocal{E}^1\!\pazocal{E}^{\mu_2}$ component] With $\sum\mu_1+\sum\mu_2=3$. These components meet \emph{main} where $\pazocal D^{\mu_1\sqcup(1)\sqcup\mu_2}$ do. We refer the reader to the above discussion of such loci.
 
 \item[${}^{\text{br}=4}\!\pazocal{E}$ component] The image is a three-nodal quartic, the elliptic curve is contracted to one of the nodes. The dimension is $13$. A map is smoothable if the image contains a tacnode, i.e. type $A_1-A_3$. The dimension is $12$.
 See Figure \ref{br4E}.
 
\begin{figure}[h]
   \begin{tikzpicture}[scale=.4]

\coordinate (Hh1) at (2,4);
 \coordinate (Hh2) at (2.1,4.5);
 
 \coordinate (AE1) at (1.7,3.7);
 \coordinate (BE1) at (1.4,3.5);
 \coordinate (CE1) at   (1.7,3.3);
 \coordinate (HE1) at (2,3);
 \coordinate (HE2) at (2.1,2.5);

 \coordinate (T1) at (1.5,3);
 \coordinate (T2) at (3,3.5);
  \coordinate (T3) at (1.5,4);

   \coordinate(fact1) at (2.4,1.2);
   \coordinate(fact2) at (3.1,-.2);
   
 \coordinate (Fs) at (3.1,3.2);
  \coordinate (Ft) at (4.5,3.2);
   \coordinate (Flabel) at (3.8,3.5);
   
    \coordinate (Fbs) at (5.3,-.3);
  \coordinate (Fbt) at (6.4,1);
    \coordinate (Fblabel) at (5.9,.1);

   \begin{scope}[every coordinate/.style={shift={(0,0)}}]
\draw[draw=gray] ([c]AE1) .. controls ([c]BE1) ..
       ([c]CE1);
       \draw[draw=gray]  plot [smooth, tension=3] coordinates { ([c]AE1) ([c]Hh1) ([c]Hh2)};
       \draw[draw=gray]  plot [smooth, tension=3] coordinates { ([c]CE1) ([c]HE1) ([c]HE2)};

        \draw[thick] ([c]T1) .. controls ([c]T2) .. ([c]T3);
       \node at ([c]T2) [xshift=.01cm,yshift=.1cm]{\scriptsize{$4$}};
       
       \draw[->,thick] ([c]Fs) to ([c]Ft);
        \node at ([c]Flabel) [above] {\scriptsize{$f$}};
            
         \node at ([c]HE2) [below,xshift=.1cm] {\begin{color}{gray}\scriptsize{$E,g=1$}\end{color}};

 \end{scope}

 \coordinate(Q1) at  (1,5);
 \coordinate (Q2) at (10,5);
 \coordinate (Q3) at (10,1 );
 \coordinate (Q4) at (1,1);
 \coordinate (Q3long) at (10,-2.5 );
 \coordinate (Q4long) at (1,-2.5);
\coordinate (Qlongmiddle) at (5.5,-2.5);

 \coordinate(P1) at  (-0.2,1);
 \coordinate (P2) at ( 3.5,1);
 \coordinate (P3) at (2.5,-1 );
 \coordinate (P4) at (-1.2,-1);

  \begin{scope}[every coordinate/.style={shift={(0,0)}}]     
 \draw (Q1)-- ( Q2)--(Q3)--(Q4)--(Q1);   
 \end{scope}
 

 \coordinate (An) at (-.5,.5);
 \coordinate (B1m) at(2,-.2);
  \coordinate (B2m) at(-1.8,-.2);
 \coordinate (Cn) at  (1,.5);

 \coordinate (D1n) at (3,-.2)  ;
  \coordinate (D2n) at (-.5,-.2) ;
  \coordinate (En) at (2,.5);
  \coordinate (F1n) at  (4,-.2);
   \coordinate (F2n) at (1,-.2);
   \coordinate (Gn) at (3,.5);

 \coordinate(P1) at  (-0.2,1);
 \coordinate (P2) at ( 3.5,1);
 \coordinate (P3) at (2.5,-1 );
 \coordinate (P4) at (-1.2,-1);
 
 \begin{scope}[every coordinate/.style={shift={(6,3)}}]

\draw[thick] ([c]An) .. controls ([c]B1m) and ([c]B2m) ..
       ([c]Cn) .. controls ([c]D1n) and ([c]D2n)  .. ([c]En) .. controls ([c]F1n) and ([c]F2n)  .. ([c]Gn);
   
 \draw[dashed] ([c]P1)-- ( [c]P2)--([c]P3)--([c]P4)--([c]P1);  
    
     \fill[gray] ([c]Cn) circle (3pt);
          \node at ([c]Cn) [above] {\begin{color}{gray}\scriptsize{$f(E)$}\end{color}};
           
    \end{scope}   

     \node at (5,1) [below] {\scriptsize{${}^{\textbf{br}=4}\pazocal E$}};


    \coordinate (A) at (2,4);
 \coordinate (B) at (1.7,3.8);
 \coordinate (C) at   (2,3.6);
 \coordinate (D) at (2.3,3.4) ;
  \coordinate (E) at (2,3.2) ;
  \coordinate (F) at (1.7,3);
  \coordinate (G) at  (2,2.8);
  \coordinate (H1) at (2.3,4.3);
 \coordinate (H2) at (2.4,4.8);
  \coordinate (J1) at (2.3,2.4);
 \coordinate (J2) at (2.4,2);

\begin{scope}[every coordinate/.style={shift={(14,0)}}]
\draw[thick] ([c]A) .. controls ([c]B) ..
       ([c]C) .. controls ([c]D)  .. ([c]E) .. controls ([c]F).. ([c]G);
      \draw[thick]  plot [smooth, tension=3] coordinates { ([c]A) ([c]H1) ([c]H2)};
       \draw[thick]  plot [smooth, tension=3] coordinates { ([c]G) ([c]J1) ([c]J2)};
       \draw[->,thick] ([c]Fs) to ([c]Ft);
        \node at ([c]Flabel) [above] {\scriptsize{$f$}};
       \end{scope}

  \begin{scope}[every coordinate/.style={shift={(14,0)}}]     
 \draw ([c]Q1)-- ( [c]Q2)--([c]Q3)--([c]Q4)--([c]Q1);   
 \end{scope}
 

 \coordinate (Am) at (0,0);
 \coordinate (Bm) at(0.2,-.3);
 \coordinate (Cm) at  ( .4,0);
 \coordinate (Dm) at (.6,.3)  ;
  \coordinate (Em) at (.8,0) ;
  \coordinate (Fm) at (1,-.3);
  \coordinate (Gm) at  (1.2,0);
  \coordinate (H1m) at (-.5,.3);
 \coordinate (H2m) at (-.2,.2);
\coordinate (J1m) at (1.4,.2);
 \coordinate (J2m) at (1.7,.3);
 \coordinate (N1m) at (3.2,-.3);
 \coordinate (N2m) at (.2,-.3);
  \coordinate (N3m) at (3,.3);

 \begin{scope}[every coordinate/.style={shift={(20,3)}}]
 
 \draw[thick]  plot [smooth, tension=3] coordinates {([c]H1m) ([c]H2m) ([c]Am)  };
 
\draw[thick] ([c]Am) .. controls ([c]Bm) ..
       ([c]Cm) .. controls ([c]Dm)  .. ([c]Em) .. controls ([c]Fm).. ([c]Gm);
    \draw[thick]  plot [smooth, tension=3] coordinates { ([c]Gm) ([c]J1m) ([c]J2m)  };   
     \draw[thick]   ([c]J2m).. controls ([c]N1m) and ([c]N2m).. ([c]N3m);
 \draw[dashed] ([c]P1)-- ( [c]P2)--([c]P3)--([c]P4)--([c]P1);  
    \end{scope}

\draw (7.5,1)--(11,-1);
\draw (19.5,1)--(16.,-1);

 \node at (12.3,3) [above] { $\cap$};

 \begin{scope}[every coordinate/.style={shift={(8,-6)}}]

 \draw ([c]Q1)-- ( [c]Q2)--([c]Q3long)--([c]Q4long)--([c]Q1);   
 \end{scope}


 \begin{scope}[every coordinate/.style={shift={(8,-6)}}]

\draw[draw=gray] ([c]AE1) .. controls ([c]BE1) ..
       ([c]CE1);
       \draw[draw=gray]  plot [smooth, tension=3] coordinates { ([c]AE1) ([c]Hh1) ([c]Hh2)};
       \draw[draw=gray]  plot [smooth, tension=3] coordinates { ([c]CE1) ([c]HE1) ([c]HE2)};

        \draw[thick] ([c]T1) .. controls ([c]T2) .. ([c]T3);
       
       \draw[->,thick] ([c]Fs) to ([c]Ft);
        \node at ([c]Flabel) [above] {\scriptsize{$f$}};
       
\draw[->,dotted] ([c]fact1) to ([c]fact2);
                          
     \draw[->,thick] ([c]Fbs) to ([c]Fbt);         
 \node at ([c]Fblabel) [right] {\scriptsize{$\bar{f}$}};
\node at ([c]Qlongmiddle) [below] {\scriptsize{${}^{\textbf{br}=4}\pazocal E\cap\overline{\pazocal M}^{\text{main}}$}};
 \end{scope}
 
 
  \coordinate (YS1) at (-.4,4);
 \coordinate (YS2) at (-.4,2) ;
  \coordinate (X1) at (.1,4.3);
   \coordinate (X2) at (.4,4);
  \coordinate (X3) at (.4,3.8);
  \coordinate (X4) at (-.6,3);
  \coordinate (X5) at (.7,2);
\begin{scope}[every coordinate/.style={shift={(12.1,-9.5)}}]

\draw[thick] ([c]YS1) -- ([c]YS2);
   
    \draw[thick]  plot [smooth, tension=2] coordinates { ([c]YS1) ([c]X1)([c]X2) ([c]X3)};
    \draw[thick] ([c]X3) .. controls ([c]X4) .. ([c]X5);
\end{scope}


  \coordinate (GSm) at  (.4,.5);

  \coordinate (Xs1) at (0,.1);
   \coordinate (Xs2) at (.2,-.3);
  \coordinate (Xs3) at (.4,-.4);
  \coordinate (Xs4) at (1.2,.8);
  \coordinate (Xs5) at (1.4,-.5);
  
  \coordinate (Xs) at (1.1,.55);

 \coordinate (J2Sm) at (1.7,.5);
 
  \coordinate (N3Sm) at (2.7,.3);

 \coordinate (DSm) at( 0.65,.3);

 \begin{scope}[every coordinate/.style={shift={(14,-3)}}]
 
 \draw[thick]  plot [smooth, tension=2] coordinates { ([c]GSm) ([c]Xs1)([c]Xs2) ([c]Xs3)};
    \draw[thick] ([c]Xs3) .. controls ([c]Xs4) .. ([c]Xs5);
 
 \draw[thick] ([c]J2Sm) --  ([c]GSm);
     \draw[thick]   ([c]J2Sm).. controls ([c]N1m) and ([c]N2m).. ([c]N3Sm);
 \draw[dashed] ([c]P1)-- ( [c]P2)--([c]P3)--([c]P4)--([c]P1);  
 
       \fill[gray] ([c]Xs) circle (3pt);
          \node at ([c]Xs) [above] {\begin{color}{gray}\scriptsize{$f(E)$}\end{color}};
          \node at ([c]Xs) [yshift=-.3cm,xshift=-.07cm] {\scriptsize{$A_3$}};
          \node at ([c]Xs) [yshift=-.3cm,xshift=.7cm] {\scriptsize{$A_1$}};
    \end{scope}  
\end{tikzpicture}
\caption{${}^{\textbf{br}=4}\pazocal E,\dim=13$; ${}^{\textbf{br}=4}\pazocal E\cap\overline{\pazocal M}^{\text{main}},\dim=12$}
\label{br4E}
\end{figure}

 \item[${}^{\text{br}=3}\!\pazocal{E}^{(1)}$ component] The image is a nodal cubic with a line passing through the node. The dimension is $12$. This locus is contained in \emph{main}, as the image determines the elliptic $3$-fold point through which the map factors.

 Notice that in this case, as well as in the next two examples, the factorisation condition is in fact a genus one condition.

 \item[${}^{\text{br}=2}\!\pazocal{E}^{(2)}$ component] A degree $2$ map can be non-injective only if it is the double cover of a line, so the image consists of a line and a conic; the elliptic curve is contracted to their intersection point. The dimension is $12$, and this locus is contained in \emph{main}.
 
 \item[${}^{\text{br}=2}\!\pazocal{E}^{(1^2)}$ component] This component has dimension $11$, and it contained in \emph{main}.The non-disconnecting bridge maps two to one to a line and the rational components map to concurrent lines, i.e. the image is a 3-fold singularity.
 \item[${}^{\text{br}=1}\!\pazocal{E}^{\mu}$ component] 
 
 Since we can't have a non injective degree one morphism, the generic point of this component actually correspond to a map from a source curve $C=R\cup Z\cup  T^{\mu}$ where: the core $Z$, on which the map is constant, is given by the union of an elliptic curve $E$ and a non separating bridge $B$;  there is a rational tail $R$ on which the map has degree one cleaving to $B$; the rest of the degree is distributed among some tails $T^{\mu}$ cleaving to $E$.
 In particular this loci lie in the closure of the components $\pazocal D^{(1)\sqcup\mu}$ with one degree $1$ tail studied before, so they
intersect \emph{main} within $\pazocal D^{(1)\sqcup\mu}$.

 \item[${}^{\text{hyp}}\!\pazocal{D}^{(2)}$ component] The core covers a line two-to-one, and the tail embeds as a conic. The dimension is $13$. Factoring through a hyperelliptic ribbon, we see that the map on the tail must be ramified at the node, i.e. it must be the double cover of a line ramified at the attaching point. The dimension is $11$. See Figure \ref{hyp}. Notice that this locus remains unaltered in $\pazocal{V\!Z}_2(\PP^2,4)$.
 
\begin{figure}[h]
   \begin{tikzpicture}[scale=.4]
 \coordinate (A) at (2,4);
 \coordinate (B) at (1.7,3.8);
 \coordinate (C) at   (2,3.6);
 \coordinate (D) at (2.3,3.4) ;
  \coordinate (E) at (2,3.2) ;
  \coordinate (F) at (1.7,3);
  \coordinate (G) at  (2,2.8);
  \coordinate (H1) at (2.3,4.3);
 \coordinate (H2) at (2.4,4.8);
 \coordinate (R1) at (1.9,4.3);
 \coordinate (R2) at (3.5,4.5);
  \coordinate (R1b) at (2.1,2.2);
 \coordinate (R2b) at (3.5,1.8);
 \coordinate (J1) at (2.3,2.4);
 \coordinate (J2) at (2.4,2);
 
   \coordinate(fact1) at (2.4,1);
   \coordinate(fact2) at (3.1,-.4);
   
 \coordinate (Fs) at (3,3.5);
  \coordinate (Ft) at (4.6,3.5);
   \coordinate (Flabel) at (3.8,3.5);
   
    \coordinate (Fbs) at (5.3,-.3);
  \coordinate (Fbt) at (6.4,1);
    \coordinate (Fblabel) at (5.9,.1);

  \begin{scope}[every coordinate/.style={shift={(0,0)}}]
\draw[draw=blue] ([c]A) .. controls ([c]B) ..
       ([c]C) .. controls ([c]D)  .. ([c]E) .. controls ([c]F).. ([c]G);
      \draw[draw=blue]  plot [smooth, tension=3] coordinates { ([c]A) ([c]H1) ([c]H2)};
       \draw[draw=blue]  plot [smooth, tension=3] coordinates { ([c]G) ([c]J1) ([c]J2)};
       \draw[thick] ([c]R1) to ([c]R2);
       \draw[->,thick] ([c]Fs) to ([c]Ft);
        \node at ([c]Flabel) [above] {\scriptsize{$f$}};
        \node at ([c]R2) [above] {\scriptsize{$2$}};
             \node at ([c]J2) [below] {\begin{color}{blue}\scriptsize{$Z,g=2,d=2$}\end{color}};
 \end{scope}

\coordinate(Q1) at  (0,5.5);
 \coordinate (Q2) at (10,5.5);
 \coordinate (Q3) at (10,1);
 \coordinate (Q4) at (0,1);
 \coordinate (Qmiddle) at (5,1);
 \coordinate (Q3long) at (10,-2.5 );
 \coordinate (Q4long) at (0,-2.5);
\coordinate (Qlongmiddle) at (5,-2.5);

  \begin{scope}[every coordinate/.style={shift={(0,0)}}]     
 \draw (Q1)-- ( Q2)--(Q3)--(Q4)--(Q1);   
 \end{scope}
 

 \coordinate (An) at (1,.1);

   \coordinate (F2n) at (-.2,.1);
     \coordinate (F1n) at  (2.2,.1);
     \coordinate (Gn) at  (1.68,.1);

 \begin{scope}[every coordinate/.style={shift={(6,3)}}]

   \draw[blue]  ([c]F2n) -- ([c]F1n);
 \draw[dashed] ([c]P1)-- ( [c]P2)--([c]P3)--([c]P4)--([c]P1);  
    
     \draw[thick]([c]An) circle (0.7cm);

    \end{scope}   

     \node at (5,1) [below] {\scriptsize{${}^{\textbf{hyp}}\pazocal D^{(2)}$}};
      \node at (12.3,3) [above] {$\cap$};


    \coordinate (A) at (2,4);
 \coordinate (B) at (1.7,3.8);
 \coordinate (C) at   (2,3.6);
 \coordinate (D) at (2.3,3.4) ;
  \coordinate (E) at (2,3.2) ;
  \coordinate (F) at (1.7,3);
  \coordinate (G) at  (2,2.8);
  \coordinate (H1) at (2.3,4.3);
 \coordinate (H2) at (2.4,4.8);
  \coordinate (J1) at (2.3,2.4);
 \coordinate (J2) at (2.4,2);

\begin{scope}[every coordinate/.style={shift={(14,0)}}]
\draw[thick] ([c]A) .. controls ([c]B) ..
       ([c]C) .. controls ([c]D)  .. ([c]E) .. controls ([c]F).. ([c]G);
      \draw[thick]  plot [smooth, tension=3] coordinates { ([c]A) ([c]H1) ([c]H2)};
       \draw[thick]  plot [smooth, tension=3] coordinates { ([c]G) ([c]J1) ([c]J2)};
       \draw[->,thick] ([c]Fs) to ([c]Ft);
        \node at ([c]Flabel) [above] {\scriptsize{$f$}};
       \end{scope}

  \begin{scope}[every coordinate/.style={shift={(14,0)}}]     
 \draw ([c]Q1)-- ( [c]Q2)--([c]Q3)--([c]Q4)--([c]Q1);   
 \end{scope}
 

 \coordinate (Am) at (0,0);
 \coordinate (Bm) at(0.2,-.3);
 \coordinate (Cm) at  ( .4,0);
 \coordinate (Dm) at (.6,.3)  ;
  \coordinate (Em) at (.8,0) ;
  \coordinate (Fm) at (1,-.3);
  \coordinate (Gm) at  (1.2,0);
  \coordinate (H1m) at (-.5,.3);
 \coordinate (H2m) at (-.2,.2);
\coordinate (J1m) at (1.4,.2);
 \coordinate (J2m) at (1.7,.3);
 \coordinate (N1m) at (3.2,-.3);
 \coordinate (N2m) at (.2,-.3);
  \coordinate (N3m) at (3,.3);

 \begin{scope}[every coordinate/.style={shift={(20,3)}}]
 
 \draw[thick]  plot [smooth, tension=3] coordinates {([c]H1m) ([c]H2m) ([c]Am)  };
 
\draw[thick] ([c]Am) .. controls ([c]Bm) ..
       ([c]Cm) .. controls ([c]Dm)  .. ([c]Em) .. controls ([c]Fm).. ([c]Gm);
    \draw[thick]  plot [smooth, tension=3] coordinates { ([c]Gm) ([c]J1m) ([c]J2m)  };   
     \draw[thick]   ([c]J2m).. controls ([c]N1m) and ([c]N2m).. ([c]N3m);
 \draw[dashed] ([c]P1)-- ( [c]P2)--([c]P3)--([c]P4)--([c]P1);  
    \end{scope}

\draw (7.5,1)--(11,-.5);
\draw (19.5,1)--(16.,-.5);

\begin{scope}[every coordinate/.style={shift={(8,-6)}}]

 \draw ([c]Q1)-- ( [c]Q2)--([c]Q3long)--([c]Q4long)--([c]Q1);   
 \end{scope}

\begin{scope}[every coordinate/.style={shift={(8,-6)}}]

    \draw[draw=blue] ([c]A) .. controls ([c]B) ..
       ([c]C) .. controls ([c]D)  .. ([c]E) .. controls ([c]F).. ([c]G);
      \draw[draw=blue]  plot [smooth, tension=3] coordinates { ([c]A) ([c]H1) ([c]H2)};
       \draw[draw=blue]  plot [smooth, tension=3] coordinates { ([c]G) ([c]J1) ([c]J2)};
       \draw[thick] ([c]R1) to ([c]R2);
       \draw[->,thick] ([c]Fs) to ([c]Ft);
        \node at ([c]Flabel) [above] {\scriptsize{$f$}};
        
    \draw[->,dotted] ([c]fact1) to ([c]fact2);
                          
     \draw[->,thick] ([c]Fbs) to ([c]Fbt);         
 \node at ([c]Fblabel) [right] {\scriptsize{$\bar{f}$}};

 \node at ([c]Qlongmiddle) [below] {\scriptsize{${}^{\textbf{hyp}}\pazocal D^{(2)}\cap\overline{\pazocal M}^{\text{main}}$}};
    
 \end{scope}


 \coordinate (YS1) at (0,4.8) ;
 \coordinate (YS2) at (0,2) ;
  \coordinate (YS11) at (0.2,4.8) ;
 \coordinate (YS12) at (0.2,2) ;

  \coordinate (WS1) at (-.3,2.8) ;
 \coordinate (WS2) at (1.5,2.8) ;
 
\begin{scope}[every coordinate/.style={shift={(12.1,-10)}}]

\draw[thick] ([c]YS1) -- ([c]YS2)--([c]YS12)--([c]YS11)--([c]YS1);
 \draw[thick] ([c]WS1) -- ([c]WS2);
\end{scope}


\coordinate (GI1) at (1.2,- 0.7);
 \coordinate (GI2) at (1.5,.7);
 \coordinate (GB) at (1.35,0);
    
   \begin{scope}[every coordinate/.style={shift={(14,-3)}}]
      \draw[blue] ([c]GS1) -- ([c]GS2);
    
     \draw[thick] ([c]GI1) -- ([c]GI2);
     \fill[blue] ([c]GB) circle (3pt);
       \fill[blue] ([c]GB)[xshift=.1cm,yshift=.5cm] circle (3pt);
   
 \draw[dashed] ([c]P1)-- ( [c]P2)--([c]P3)--([c]P4)--([c]P1);

    \end{scope}   

\end{tikzpicture}
\caption{${}^{\textbf{hyp}}\pazocal D^{(2)},\dim=13$; ${}^{\textbf{hyp}}\pazocal D^{(2)}\cap\overline{\pazocal M}^{\text{main}},\dim=11$}
\label{hyp}	
\end{figure}

 \item[${}^{\text{hyp}}\!\pazocal{D}^{(1^2)}$] The image consists of a line $\ell$ (doubly covered by the core) together with two other lines. The dimension is $12$. The locus is entirely contained in \emph{main}. Notice that aligning introduces a $1$-dimensional fibre. On the other hand, if $x_0$ denotes the coordinate vanishing along $\ell$, $x_0$ will only descend to one of the possible ribbons (so, the correct ribbon is determined by the image of the map).
\end{description}

\section*{Acknowledgements} We thank Dhruv Ranganathan for suggesting that we should explore our construction in an example. We are grateful to an anonymous referee for the suggestions that led to an improved exposition of the material contained in this paper. L.B. is supported by the Deutsche Forschungsgemeinschaft (DFG, German Research Foundation) under Germany’s Excellence Strategy EXC-2181/1 - 390900948 (the Heidelberg STRUCTURES Cluster of Excellence).

\medskip

\noindent Luca Battistella\\
Mathematisches Institut, Ruprecht-Karls-Universit\"at Heidelberg \\
\texttt{lbattistella@mathi.uni-heidelberg.de}\\

\noindent Francesca Carocci\\
Ecole Polytechnique F\'ed\'erale de Lausanne \\
\texttt{francesca.carocci@epfl.ch}

\end{document}